\documentclass[11pt, oneside]{amsart}
\usepackage[text={5.58in,8.5in},centering,letterpaper,dvips]{geometry}
\usepackage[dvipsnames]{xcolor}
\usepackage{graphicx}
\usepackage{subcaption}
\usepackage{amsfonts}
\usepackage{epsf}
\usepackage{amssymb}
\usepackage{amsmath}
\usepackage{amscd}
\usepackage{tikz}
\usepackage{pdfpages}
\usepackage{fancyhdr}
\usepackage{setspace}
\usepackage{hyperref}
\usepackage[all]{xy}
\usetikzlibrary{matrix}
\usepackage{verbatim}
\usepackage{enumerate}

\theoremstyle{theorem}
\newtheorem{theorem}{Theorem}[section]
\newtheorem{proposition}[theorem]{Proposition}
\newtheorem{lemma}[theorem]{Lemma}
\newtheorem{question}[theorem]{Question}
\newtheorem{corollary}[theorem]{Corollary}

\theoremstyle{definition}

\newtheorem{remark}[theorem]{Remark}

\newcommand{\Z}{\mathbb{Z}}
\newcommand{\C}{\mathcal{C}}

\newcommand{\R}{\mathbb{R}}

\newcommand{\CP}{\mathbb{CP}}
\newcommand{\Cc}{\mathbb{C}}

\newcommand{\D}{\mathcal D}

\newcommand{\Ss}{\mathcal S}

\newcommand{\wh}[1]{\widehat{#1}}

\newcommand{\A}{\alpha}
\newcommand{\n}{\beta}
\newcommand{\g}{\gamma}
\newcommand{\pd}{\partial}

\newcommand{\emp}{\emptyset}
\newcommand{\X}{\times}

\newcommand{\be}{\begin{enumerate}}
\newcommand{\ee}{\end{enumerate}}
\newcommand{\eps}{\epsilon}
\newcommand{\K}{\mathcal{K}}
\newcommand{\aaa}{\mathfrak{a}}
\newcommand{\bbb}{\mathfrak{b}}
\newcommand{\ccc}{\mathfrak{c}}
\newcommand{\ac}{\mathfrak{ac}}
\newcommand{\ca}{\mathfrak{ca}}
\newcommand{\ab}{\mathfrak{ab}}
\newcommand{\ba}{\mathfrak{ba}}
\newcommand{\bc}{\mathfrak{bc}}
\newcommand{\cb}{\mathfrak{ac}}

\newcommand{\T}{\mathcal T}
\newcommand{\la}{\langle}
\newcommand{\ra}{\rangle}
\newcommand{\abp}{\aaa'\bbb'}
\newcommand{\bcp}{\bbb'\ccc'}
\newcommand{\capp}{\ccc'\aaa'}
\newcommand{\V}{\mathcal{V}}
\newcommand{\abc}{(\aaa,\bbb,\ccc)}
\newcommand{\abcp}{(\aaa',\bbb',\ccc')}
\newcommand{\abcs}{(\aaa^*,\bbb^*,\ccc^*)}

\providecommand{\numberTblEq}[1]{\refstepcounter{tblEqCounter}\label{#1}\thetag{\thetblEqCounter}}

\makeatletter
\def\@seccntformat#1{%
  \protect\textup{\protect\@secnumfont
    \ifnum\pdfstrcmp{subsection}{#1}=0 \bfseries\fi
    \csname the#1\endcsname
    \protect\@secnumpunct
  }%
}  
\makeatother

\makeatletter
\newtheorem*{rep@theorem}{\rep@title}
\newcommand{\newreptheorem}[2]{%
\newenvironment{rep#1}[1]{%
 \def\rep@title{#2 \ref{##1}}%
 \begin{rep@theorem}}%
 {\end{rep@theorem}}}
\makeatother

\newreptheorem{theorem}{Theorem}
\newreptheorem{lemma}{Lemma}
\newreptheorem{question}{Question}
\newreptheorem{corollary}{Corollary}
\newreptheorem{proposition}{Proposition}

\begin{document}

\rhead{\thepage}
\lhead{\author}
\thispagestyle{empty}


\raggedbottom
\pagenumbering{arabic}
\setcounter{section}{0}


\title{Hexagonal lattice diagrams for complex curves in $\CP^2$}

\author{Alexander Zupan}
\address{Department of Mathematics, University of Nebraska-Lincoln, Lincoln, NE 68588}
\email{zupan@unl.edu}
\urladdr{http://www.math.unl.edu/azupan2}

\begin{abstract}
We demonstrate that the geometric, topological, and combinatorial complexities of certain surfaces in $\CP^2$ are closely related:  We prove that a positive genus surface $\K$ in $\CP^2$ that minimizes genus in its homology class is isotopic to a complex curve $\mathcal{C}_d$ if and only if $\K$ admits a hexagonal lattice diagram, a special type of shadow diagram in which arcs meet only at bridge points and tile the central surface of the standard trisection of $\CP^2$ by hexagons.  There are eight families of these diagrams, two of which represent surfaces in efficient bridge position.  Combined with a result of Lambert-Cole relating symplectic surfaces and bridge trisections, this allows us to provide a purely combinatorial reformulation of the symplectic isotopy problem in $\CP^2$.  Finally, we show that that the varieties $\V_d = \{[z_1:z_2:z_3] \in \CP^2 : z_1z_2^{d-1} + z_2z_3^{d-1} + z_3z_1^{d-1} = 0\}$ and $\V'_d = \{[z_1:z_2:z_3] \in \CP^2 : z_1^{d-1}z_2 + z_2^{d-1}z_3 + z_3^{d-1}z_1 = 0\}$ are in efficient bridge position with respect to the standard Stein trisection of $\CP^2$, and their shadow diagrams agree with the two families of efficient hexagonal lattice diagrams.  As a corollary, we prove that two infinite families of complex hypersurfaces in $\CP^3$ admit efficient Stein trisections, partially answering a question of Lambert-Cole and Meier.
\end{abstract}

\maketitle

\section{Introduction}

\newcounter{tblEqCounter}

Mathematics, at its core, involves the search for patterns and structure.  This paper centers on the search for structure in the setting of 4-manifold trisections and in the interactions of the theory of trisections with symplectic and complex geometry.  Mounting evidence reveals deep connections between these fields.  We demonstrate a strong association between the geometric and topological simplicity of complex curves in $\CP^2$ and the combinatorial simplicity of their \emph{hexagonal lattice diagrams}.

\begin{figure}[h!]
  \centering
  \includegraphics[width=.9\linewidth]{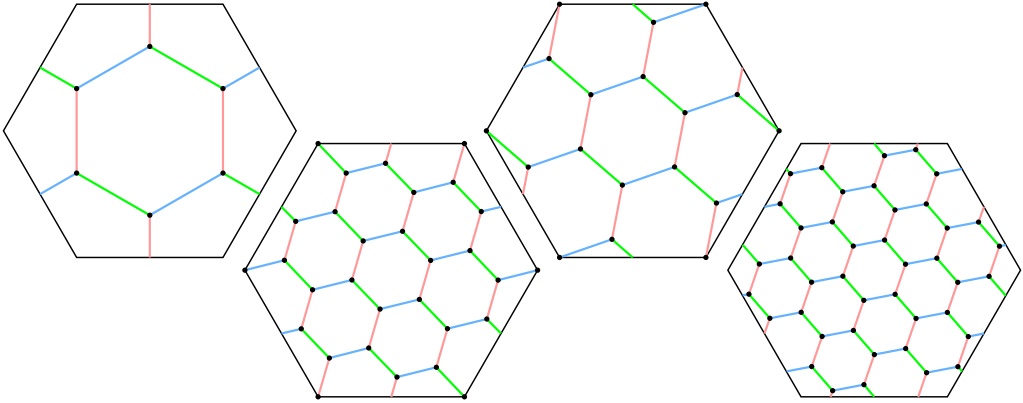}
    \caption{Efficient hexagonal lattice diagrams for $\C_3$ (top left), $\C_4$ (top right), $\C_5$ (bottom left), and $\C_6$ (bottom right)}
  \label{fig:alleff}
\end{figure}

The standard trisection of $\CP^2$ cuts it into three 4-balls that meet pairwise in solid tori and whose triple intersection is a torus, $\Sigma$.  With Jeffrey Meier, we proved that every surface $\K \subset \CP^2$ can be isotoped into bridge position, so that it meets the 4-balls in trivial disks and the solid tori in trivial arcs~\cite{MZB2}.  Moreover, the surface is determined by the arcs in the solid tori, which can be isotoped into the central surface $\Sigma$ to create a \emph{shadow diagram} $\abc$ for $\K$.  In the very special case that these shadows have disjoint interiors and tile $\Sigma$ by hexagons, we call such a shadow diagram a \emph{hexagonal lattice diagram}.  For examples, see Figure~\ref{fig:alleff}.  We prove the following, where $\C_d$ denotes the complex curve of degree $d$ in $\CP^2$:

\begin{theorem}\label{thm:main1}
Let $\K \subset \CP^2$ be a positive genus surface such that $g(\K)$ is minimal in its homology class.  Then $\K$ is isotopic to $\C_d$ if and only if $\K$ admits a hexagonal lattice diagram.
\end{theorem}

The forward direction was proved by Lambert-Cole and Meier in~\cite{LCM}.  For the reverse direction, we prove

\begin{theorem}\label{thm:support}
Suppose $\K \subset \CP^2$ is a positive genus surface such that $g(\K)$ is minimal in its homology class and such that $\K$ admits a hexagonal lattice diagram $\abc$.  Then this diagram falls into one of eight possible families.  Moreover, every diagram in each of these families corresponds to a surface isotopic to some complex curve $\C_d$.
\end{theorem}

The condition that $\K$ has positive genus is explained in detail in Remark~\ref{rmk:positive}.

Recent work of Lambert-Cole~\cite{PLCsymp} and Lambert-Cole, Meier, and Starkston~\cite{LMS} has uncovered intriguing connections between trisection theory and symplectic topology in dimension four.  Theorem~\ref{thm:main1} is similar in spirit to the following result of Lambert-Cole:

\begin{theorem}\cite{PLCsymp}\label{thm:PLC}
Let $\K \subset \CP^2$ be a surface such that $g(\K)$ is minimal in its homology class.  Then $\K$ is isotopic to a symplectic surface if and only if $\K$ admits a transverse bridge position.
\end{theorem}

The \emph{symplectic isotopy problem} asks whether every symplectic surface is isotopic through symplectic surfaces to some complex curve $\C_d$.  This problem has been answered in the affirmative for surfaces of degree $d \leq 17$ by Gromov, Shevchishin, Sikorav, and Siebert-Tian~\cite{gromov, shev, sikorav, siebert-tian}.  By combining Theorems~\ref{thm:main1} and~\ref{thm:PLC}, we can obtain a reformulation of the symplectic isotopy problem in purely combinatorial terms.

\begin{question}[Combinatorial symplectic isotopy problem]
Suppose that $\K \subset \CP^2$ is a surface such that $g(\K)$ is minimal in its homology class and such that $\K$ is admits a transverse bridge position $\T$.  Is $\T$ equivalent to a hexagonal lattice diagram?
\end{question}

The most interesting families uncovered in Theorem~\ref{thm:support} are the families denoted $(A)_d$ and $(E)_d$, which can be seen to be regular tilings of the torus $\Sigma$~\cite{perimeter}, and which are the only families yielding \emph{efficient} bridge trisections, in which the corresponding surfaces are decomposed into three disks.  Diagrams $(E)_d$ for $d=3,4,5,6$ are shown in Figure~\ref{fig:alleff}.  In~\cite{LCM}, Lambert-Cole and Meier proved that every complex curve $\C_d$ is isotopic to a surface with an efficient bridge trisection.  We prove a stronger result, that these diagrams correspond to surfaces that are trisected \emph{as complex curves}.  Specifically,

\begin{theorem}\label{thm:main2}
For $d > 0$, let $\V_d$ and $\V_d'$ denote the varieties
\begin{eqnarray*}
\V_d &=& \{[z_1:z_2:z_3] \in \CP^2 : z_1z_2^{d-1} + z_2z_3^{d-1} + z_3z_1^{d-1} = 0\}, \\
\V'_d &=& \{[z_1:z_2:z_3] \in \CP^2 : z_1^{d-1}z_2 + z_2^{d-1}z_3 + z_3^{d-1}z_1= 0\}.
\end{eqnarray*}
Then $\V_d$ and $\V_d'$ are bridge trisected with respect to the standard (Stein) trisection of $\CP^2$, and these bridge trisections correspond to the hexagonal lattice diagrams $(A)_d$ and $(E)_d$, respectively.
\end{theorem}

Finally, we obtain new information about Stein trisections.  Stein trisections were introduced in~\cite{LCM}.  A \emph{Stein trisection} of a complex 4-manifold $X$ is a trisection $X = X_1 \cup X_2 \cup X_3$ such that each 4-dimensional handlebody $X_j$ is an analytic polyhedron (see~\cite{PLCstein} for further details and examples).  Lambert-Cole and Meier showed that the standard trisection of $\CP^2$ is Stein and asked whether every complex projective surface admits a Stein trisection.  As a corollary to Theorem~\ref{thm:main2}, we obtain a partial positive answer to this question.

\begin{corollary}\label{cor:stein}
The complex surfaces $\Ss_d$ and $\Ss'_d$ given by
\begin{eqnarray*}
\Ss_d &=& \{[z_1:z_2:z_3:z_4] \in \CP^3 : z_1z_2^{d-1} + z_2z_3^{d-1} + z_3z_1^{d-1} + z_4^d= 0\}, \\
\Ss'_d &=& \{[z_1:z_2:z_3:z_4] \in \CP^3 : z_1^{d-1}z_2 + z_2^{d-1}z_3 + z_3^{d-1}z_1 + z_4^d= 0\}
\end{eqnarray*}
admit efficient Stein trisections.
\end{corollary}

\subsection{Layout of the paper}
In Section~\ref{sec:prelim}, we introduce relevant definitions and other preliminary material.  At then end of this section, we classify hexagonal lattice diagrams with Theorem~\ref{thm:family}, whose proof is a tedious case-by-case analysis that has been relegated to an Appendix in Section~\ref{sec:appendix}.  In Section~\ref{sec:tools}, we discuss various tools and techniques that will be used to show that the diagrams in the families classified by Theorem~\ref{thm:family} represent surfaces isotopic to complex curves, which we prove in Section~\ref{sec:fam}, completing the proof of Theorem~\ref{thm:main1}.  In Section~\ref{sec:stein}, we analyze the intersection of the varieties $\V_d$ and $\V_d'$ with the standard trisection of $\CP^2$ in order to prove Theorem~\ref{thm:main2} and Corollary~\ref{cor:stein}.  Finally, in Section~\ref{sec:questions}, we include a few questions for further investigation.

\subsection{Acknowledgements}
We are very grateful to Peter Lambert-Cole for discussing this problem, explaining his work, and sharing his insights, in particular pointing out the proof of Corollary~\ref{cor:stein} using Theorem~\ref{thm:main2}.  We also thank Jeffrey Meier and Laura Starkston for interesting discussions related to this work.  Finally, we appreciate the hospitality of the Max Planck Institute for Mathematics in Bonn, Germany, which accommodated the author as a guest researcher during part of the completion of this work.  The author is supported by NSF grant DMS-2005518.

\section{Preliminaries and classifying hexagonal lattice diagrams}\label{sec:prelim}

We work in the smooth category, and all manifolds are oriented and connected, unless otherwise noted.  Recall that $\CP^n$ is the quotient of $(\Cc^{n+1})^*$ by the equivalence relation $\vec z \sim \vec z'$ if there is some $w \in \Cc^*$ such that $\vec z'=w \vec z$.  We let $[z_1:\ldots:z_{n+1}]$ denote the equivalence class of $\vec z = (z_1,\dots,z_{n+1}) \in (\Cc^{n+1})^*$.  The focus of this paper is a collection of surfaces, called complex curves, in $\CP^2$.  We define the \emph{degree $d$ complex curve} $\C_d \subset \CP^2$ to be the surface
\[ \C_d = \{[z_1:z_2:z_3] : z_1^d + z_2^d + z_3^d = 0\},\]
considered up to isotopy.  Later on, when we wish to consider a zero set as a rigid object (without allowing for isotopy), we will use the term \emph{variety}.

In the definitions that follow, we use the conventions set in~\cite{LCM,PLCsymp,PLCthom}:   The \emph{standard trisection} $\T$ of $\CP^2$, induced by the moment polytope, is the decomposition $\CP^2 = X_1 \cup X_2 \cup X_3$, where (with indices taken modulo 3)
\[ X_j = \{[z_1:z_2:z_3]: |z_j|,|z_{j+1}|\leq |z_{j+2}|\},\]
so that each $X_j$ is diffeomorphic to a 4-ball.  The pairwise intersections $H_j = X_j \cap X_{j-1}$ can be described as
\[ H_j = \{[z_1:z_2:z_3]: |z_j| \leq |z_{j+1}| = |z_{j+2}|\},\]
so that $H_j$ is a solid torus $S_1 \X D^2$.  Finally, the triple intersection $\Sigma = X_1 \cap X_2 \cap X_3$ is
\[ \Sigma = \{[z_1:z_2:z_3]: |z_1| = |z_2| = |z_3|\},\]
a torus in $\CP^2$.  In addition, the natural orientation of $\CP^2$ induces orientations of each $X_j$, and we orient the solid tori so that $\pd X_j = H_j \cup -H_{j+1}$.  A trisection is uniquely determined up to diffeomorphism by a trisection diagram, which in this case consists of three oriented curves $\A$, $\n$, and $\g$ in $\Sigma$, where the curves $\A$, $\n$, and $\g$ bound meridian disks in the solid tori $H_1$, $H_2$, and $H_3$, respectively.  Fixing $z_3 = 1$, we can see that
\[ \Sigma = \{[e^{i\theta}: e^{i\psi} : 1]: \theta,\psi \in [0,2\pi]\},\]
and thus as parameterized curves in $\Sigma$,
\begin{eqnarray*}
\A &=& \{[e^{i\theta}:1:1]:\theta\in [0,2\pi]\} \\
\n &=& \{[1:e^{i\psi}:1]:\psi\in [0,2\pi]\} \\
\g &=& \{[e^{-i\theta}:e^{-i\theta}:1]:\theta\in [0,2\pi]\}.
\end{eqnarray*}
By \emph{curve}, we mean a collection of pairwise disjoint (possibly multiple parallel copies) of simple closed curves in $\Sigma$, and since $\pi_1(\Sigma)$ and $H_1(\Sigma;\Z)$ coincide, we will blur the distinction between a curve in $\Sigma$, its homotopy class, and its homology class.  Letting $\langle \cdot, \cdot \rangle$ denote the intersection pairing on $H_1(\Sigma;\Z)$ (i.e. the algebraic intersection number of two curves in $\Sigma$), we have
\[ \la \A,\n \ra = \la \n,\g \ra = \la \g,\A \ra = 1.\]
and as elements of $H_1(\Sigma)$, the pair $(\A,\n)$ is a symplectic basis with $\g = -\A - \n$ (and the same holds for any cyclic permutation of $\A$, $\n$, and $\g$).  Note that if $(\mu,\lambda)$ is an symplectic basis for $H_1(\Sigma)$, then
\[ \la a\mu + b\lambda, c\mu + d\lambda \ra= ad - bc.\]
To emphasize the three-fold symmetry of this construction, we will draw the torus $\Sigma$ not in the usual rectangular way but rather as a hexagon with opposite sides identified.  We set the convention that the curves $\A$, $\n$, and $\g$ are drawn in red, blue, and green, respectively, as shown in Figure~\ref{fig:cp2} below.  These curves remain the same throughout the hexagonal figures in this paper, but they are usually suppressed.

\begin{figure}[h!]
	\centering
	\includegraphics[width=.3\textwidth]{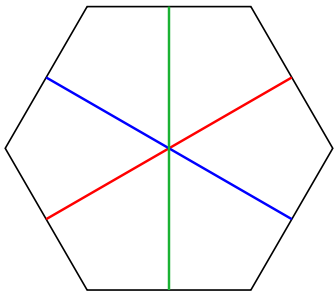}
	\caption{Diagram for the standard trisection of $\CP^2$.  Opposite sides of the hexagon are identified to form the torus $\Sigma$.}
	\label{fig:cp2}
\end{figure}

Next, we turn our attention to surfaces in $\CP^2$.  Suppose that $\K \subset \CP^2$ is an embedded surface.  We say that $\K$ is in $(b;c_1,c_2,c_3)$-\emph{bridge position} with respect to the standard trisection if for each index $j$
\be
\item $\K \cap X_j$ is a collection of $c_j$ disks $\D_j$ that are isotopic rel boundary into $\pd X_j$, and
\item $\K \cap H_j$ is a collection of $b$ arcs $\tau_j$ that are isotopic rel boundary into $\pd H_j$.
\ee
It follows that $\K \cap \Sigma$ is a collection of $2b$ points.  With Meier, we proved that

\begin{theorem}\cite{MZB2}\label{thm:bridge}
Any surface $\K \subset \CP^2$ can be isotoped into some $(b;c_1,c_2,c_3)$-bridge position with respect to the standard trisection $\T$ of $\CP^2$.
\end{theorem}

In addition, we showed that $\K$ is uniquely determined by the $b$-strand trivial tangles $\tau_1$, $\tau_2$, and $\tau_3$.  A set of \emph{shadows} for $\tau_j$ is a collection $\aaa$ of pairwise disjoint embedded arcs in $\Sigma$ which are the image of $\tau_j$ under an isotopy pushing $\tau_j$ into $\Sigma = \pd H_j$, and a \emph{shadow diagram} is a triple $(\aaa,\bbb,\ccc)$ of sets of shadows determined by the tangles $\tau_1$, $\tau_2$, and $\tau_3$, respectively.  As a corollary of Theorem~\ref{thm:bridge}, every surface in $\CP^2$ can be encoded by a shadow diagram.

A special type of shadow diagram that will be the focus of this paper is a hexagonal lattice diagram.  We say that a shadow diagram $(\aaa,\bbb,\ccc)$ is a \emph{hexagonal lattice diagram} if the shadow arcs meet only at bridge points, and if the union $\aaa \cup \bbb \cup \ccc$ tiles the torus $\Sigma$ by hexagons so that every hexagon contains two opposite edges from each of the three sets $\aaa$, $\bbb$, and $\ccc$.  See Figure~\ref{fig:alleff} for examples.

\begin{remark}
Hexagonal tilings of the torus are interesting objects of study in their own right.  We are grateful to work in~\cite{perimeter} and~\cite{wedd} for helping us understand the nature of these tilings.
\end{remark}

Following~\cite{LCM}, we can see that given a hexagonal lattice diagram $(\aaa,\bbb,\ccc)$, each bridge point has a \emph{sign} of $\pm 1$ (not to be confused with an orientation, defined below), depending on whether a counterclockwise loop traverses arcs from $\aaa$, $\bbb$, $\ccc$ in order (sign $+1$) or vice versa (sign $-1$).  See Figure~\ref{fig:sign}.  It follows from the definition of a hexagonal lattice diagram that all bridge points will have the same sign, and thus we define the \emph{sign} $\eps = \pm 1$ of $(\aaa,\bbb,\ccc)$ to be the sign of any of its bridge points.

\begin{figure}[h!]
\begin{subfigure}{.45\textwidth}
  \centering
  \includegraphics[width=.3\linewidth]{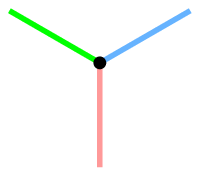}
  \label{fig:sgn+}
  \caption{$\eps = 1$}
\end{subfigure}
\begin{subfigure}{.45\textwidth}
  \centering
  \includegraphics[width=.3\linewidth]{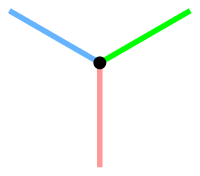}
  \label{fig:sgn-}
  \caption{$\eps = -1$}
\end{subfigure}
	\caption{The two possible signs of a bridge point}
\label{fig:sign}
\end{figure}

An \emph{orientation} of a hexagonal lattice diagram $(\aaa,\bbb,\ccc)$ is a choice of orientation of the bridge points so that vertices of every hexagon alternate between positive and negative.  Note that given a hexagonal lattice diagram, there are exactly two orientations, and an orientation of the bridge points induces an orientation of each arc in the diagram by requiring that arcs travel from negative points to positive points.  Note also that if $\K$ is an oriented surface in $\CP^2$ with a hexagonal lattice diagram $(\aaa,\bbb,\ccc)$, then $\K$ induces an orientation of $(\aaa,\bbb,\ccc)$ (see Lemma 2.1 of~\cite{MTZ} for a proof).  Since the arcs $\aaa$ and $\bbb$ meet only at bridge points, $\aaa \cup -\bbb$ is an oriented curve in $\Sigma$.  As shorthand, we let $\ab$ denote the curve
\[ \ab = \aaa \cup (-\bbb),\]
with similar notation $\ac$, $\ba$, $\bc$, $\ca$, and $\cb$ defined analogously.  By construction, a hexagonal lattice diagram determines $\ab$ and the homology classes from the other five pairings.  Note that in $H_1(\Sigma)$, these curves will be related by linear equations; for example,
\[ \ab + \bc = [\aaa \cup (-\bbb)] + [\bbb \cup (-\ccc)] = [\aaa \cup (-\ccc)] = \ac.\]

\begin{lemma}\label{lem:dist}
For a given hexagonal lattice diagram $(\aaa,\bbb,\ccc)$, we have $\la \ac,\bc \ra = \eps b$.  In particular, the homology classes $\ac$ and $\bc$ must be distinct.
\end{lemma}

\begin{proof}
Each arc in $\ccc$ gives rise to a point of intersection of the curves $\ab$ and $\bc$, and by the definition of hexagonal lattice diagram, each of these intersection points has the same sign, as shown in Figure~\ref{fig:smooth}.  Thus, $\la \ac, \bc \ra = \eps |\ccc| = \eps b$, so $\ac \neq \bc$.
\end{proof}

\begin{figure}[h!]
\begin{subfigure}{.48\textwidth}
  \centering
  \includegraphics[width=.6\linewidth]{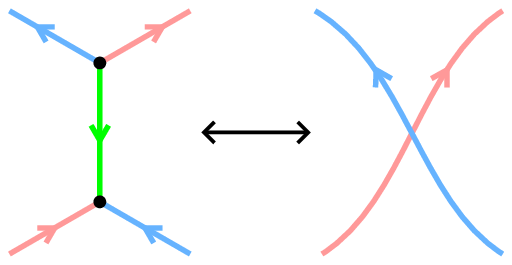}
  \label{fig:sm1}
\end{subfigure}
\begin{subfigure}{.48\textwidth}
  \centering
  \includegraphics[width=.6\linewidth]{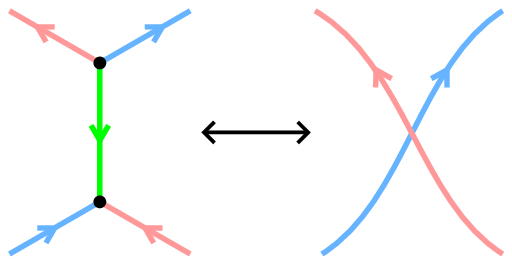}
  \label{fig:sm2}
\end{subfigure}
	\caption{A local picture of translating between a hexagonal lattice diagram with positively (left) and negatively (right) signed bridge points and the curves $\ac$ and $\bc$.}
\label{fig:smooth}
\end{figure}

We also have next lemma, which will allow us to understand hexagonal lattice diagrams by examining their constituent curves $\ab$ and $\bc$.

\begin{lemma}\label{lem:create}
A hexagonal lattice diagram $(\aaa,\bbb,\ccc)$ is determined by the homology classes $\ac$ and $\bc$.
\end{lemma}

\begin{proof}
Note that $\ac$ and $\bc$ determine unique curves up to homotopy, and assuming these curves intersect efficiently, all intersection points have the same sign.  It follows that each intersection point can be resolved by removing a neighborhood of the intersection point and inserting two bridge points and an arc in $\ccc$, as shown from right to left in Figure~\ref{fig:smooth}.  Carrying out this procedure at each intersection point yields the hexagonal lattice diagram $(\aaa,\bbb,\ccc)$.
\end{proof}

An example of the procedure described in Lemma~\ref{lem:create} is shown in Figure~\ref{fig:create}.

\begin{figure}[h!]
\begin{subfigure}{.48\textwidth}
  \centering
  \includegraphics[width=.6\linewidth]{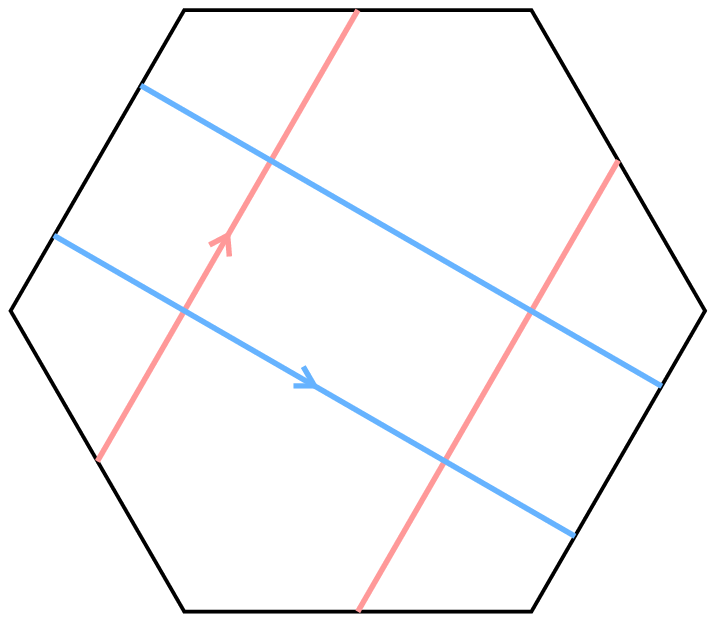}
  \label{fig:cr1}
\end{subfigure}
\begin{subfigure}{.48\textwidth}
  \centering
  \includegraphics[width=.6\linewidth]{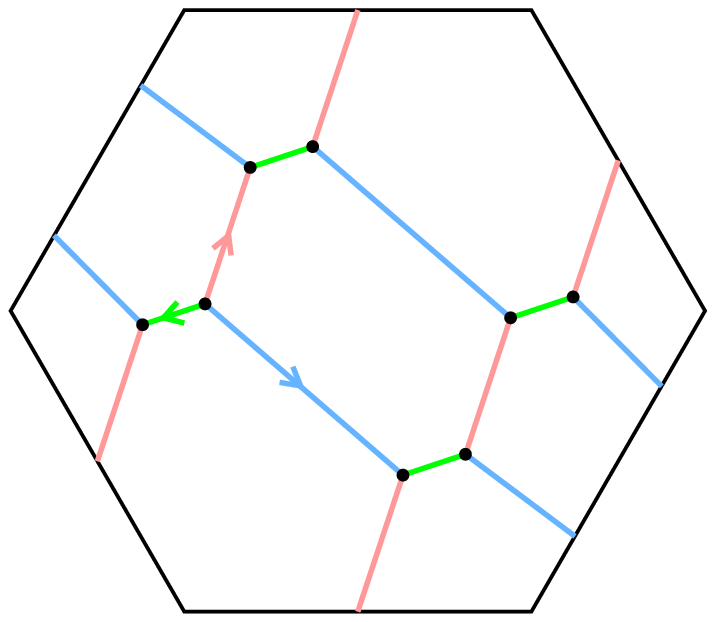}
  \label{fig:cr2}
\end{subfigure}
	\caption{Constructing a potential hexagonal lattice diagram from two curves $\ac$ and $\bc$.}
	\label{fig:create}
\end{figure}

Recall that $\A$, $\n$, and $\g$ are the curves bounding disks in the solid tori $H_1$, $H_2$, and $H_3$, respectively.  For the remainder of this section and in Section~\ref{sec:appendix}, we set the convention that
\[ \ab  = p_1[\A] + q_1[\n], \quad \bc = p_2[\n] + q_2[\g], \quad \text{ and } \ca = p_3[\g] + q_3[\A].\]
A caution to the reader:  Given two curves $\ab$ and $\bc$, it is not necessarily the case that the triple of arcs $(\aaa,\bbb,\ccc)$ constructed in Lemma~\ref{lem:create} will determine a hexagonal lattice diagram, since one requirement of a hexagonal lattice diagram is that the pairings $\ab$, $\bc$, and $\ca$ bound disks in the appropriate copies of $S^3 = \pd X_j$.  However, we know the following must be true:

\begin{lemma}\label{lem:possible}
Suppose that $(\aaa,\bbb,\ccc)$ is a hexagonal lattice diagram.  Then one of the following must hold:
\[ \ab = k\A, \quad \ab = k\n, \quad \la \ab,\A \ra = \pm 1,\,\, \,\, \text{or}\,\, \,\, \la \ab,\n \ra = \pm 1.\]
Similar statements hold for $\bc$ and $\ca$.
\end{lemma}

\begin{proof}
Note that $\ab$ is a $(p_1,q_1)$-torus link in $S^3 = H_1 \cup H_2$.  It follows that $\ab$ is an unlink if and only if $q_1 = 0$, $p_1=0$, $q_1 = \pm 1$, or $p_1 = \pm 1$.  These four possibilities translate to the four choices given in the statement of the lemma.
\end{proof}

We let $[\K]^2$ denote the algebraic self-intersection number of the surface $\K$.  Moreover, we have that $H_2(\CP^2;\Z) = \Z$ and can be generated by $[\C_1]$, so that $[\K] = d[\C_1]$ for some integer $d$, which is called the \emph{degree} of $\K$.  It follows that $[\K]^2 = (d[\C_1])^2 = d^2[\C_1]^2 = d^2$.  Using results from~\cite{PLCthom}, we establish the next formula to compute $[\K]^2$ using the data from a hexagonal lattice diagram:

\begin{lemma}\label{int}
Suppose $\K$ is represented by a hexagonal lattice diagram $(\aaa,\bbb,\ccc)$.  Then the self-intersection $[\K]^2$ is given by
\[ [\K]^2 = p_1q_1 + p_2q_2 + p_3q_3 + \eps b.\]
\end{lemma}

\begin{proof}
By Proposition 2.3 of~\cite{PLCthom}, we have $p_2 = d-q_1$, $q_2 = d - p_3$, and $q_3 = d-p_1$.  Then, by Proposition 2.5 of~\cite{PLCthom},
\begin{eqnarray*}
d^2 &=& d(p_1+p_2+p_3) - p_1p_2 - p_2p_3 - p_3p_1+\eps b \\
&=& (p_2+q_1)(p_1) + (p_3+q_2)(p_2) + (p_1+q_3)(p_3) - p_1p_2 - p_2p_3 - p_3p_1 + \eps b \\
&=& p_1q_1 + p_2q_2 + p_3q_3 + \eps b.
\end{eqnarray*}
Note that each of the terms $w_j(\Ss(\K))$ appearing in~\cite{PLCthom} is zero, since the curves $\ab$, $\bc$, and $\ca$ contain no crossings, and in addition, the $\frac{1}{2}\sum_{v} \eps_{\Ss}(v)$ term from~\cite{PLCthom} is equal to $\eps b$, since all bridge points in a hexagonal lattice diagram have the same sign.

An alternative proof worth noting here uses the fact that $\ab$ is a $(p_1,q_1)$-torus knot in $\pd X_1$, and thus if $\K'$ is a pushoff of $\K$, the trivial disks $\D_1$ bounded by $\ab$ will intersect their pushoff a total of $p_1q_1$ times in the interior of $X_1$.  Similarly, the disks $\D_2$ intersect a pushoff $p_2q_2$ times in the interior of $X_2$, and the disks $\D_3$ intersect a pushoff $p_3q_3$ times in the interior of $X_3$.  The remaining intersection points of $\K$ and $\K'$ can be seen in the surface $\Sigma$; there are precisely $b$ of them and they have sign $\eps$.
\end{proof}

The next computation is given in~\cite{MZB2}.

\begin{lemma}\label{lem:genus}
Suppose $\K$ is represented by a hexagonal lattice diagram $(\aaa,\bbb,\ccc)$.  Then the genus of $\K$ is given by
\[ g(\K) = \frac{b - c_1 - c_2 - c_3}{2} + 1.\]
\end{lemma}

Finally, we will use the Thom Conjecture, which was first proved by Kronheimer-Mrowka~\cite{KMThom} (with a reproof using bridge trisections by Lambert-Cole in~\cite{PLCthom}).

\begin{theorem}\label{thm:thom}
The degree $d$ surface $\K$ is genus-minimizing in its homology class if and only if
\[ g(\K) = \frac{(d-1)(d-2)}{2}.\]
\end{theorem}

Putting these ingredients together, we can classify all possible hexagonal lattice diagrams for positive genus surfaces that minimize genus in their homology class.

\begin{theorem}\label{thm:family}
Suppose that $(\aaa,\bbb,\ccc)$ is a hexagonal lattice diagram for a surface $\K$ that minimizes genus in its homology class.  Then, either the degree of $\K$ is 0, $\pm1$, or $\pm 2$, or up to reversing the orientation of $\K$ and a cyclic permutation of the collections of arcs $(\aaa,\bbb,\ccc)$, the hexagonal lattice diagram is in one of the following eight classes: \vspace{.2cm} \\
\begin{tabular}{llll}
(A) & $\ab = \A+(d-1)\n$, & $\bc = \n+(d-1)\g$, & $\ca = \g+(d-1)\A$ \\
(B) & $\ab = (d-1)\n$, & $\bc = \n + (d-1)\g$, & $\ca = \g + d\A$ \\
(C) & $\ab = d\n$, & $\bc = (d-1)\g$, & $\ca = \g+d\A$ \\
(D) & $\ab = d\n$, & $\bc = d\g$, & $\ca = d\A$ \\
(E) & $\ab = (d-1)\A + \n$, & $\bc = (d-1)\n+\g$, & $\ca = (d-1)\g + \A$ \\
(F) & $\ab = (d-1)\A$, & $\bc = d\n +\g$, & $\ca = (d-1)\g + \A$ \\
(G) & $\ab = d\A $, & $\bc = d\n+\g$, & $\ca = (d-1)\g$ \\
(H) & $\ab = d\A$, & $\bc = d\n$, & $\ca = d\g$ \\
\end{tabular}

\end{theorem}

\begin{proof}
Suppose that $(\aaa,\bbb,\ccc)$ is a hexagonal lattice diagram for a some surface $\K$.  By Lemma~\ref{lem:possible}, there are six cases to check for each of $\ab$, $\bc$, and $\ca$; combined, these yield 216 potential cases to check.  Using Lemmas~\ref{int} and~\ref{lem:genus}, we complete a case-by-case analysis, which is tedious and not particularly enlightening.  As such, it has been relegated to the Appendix, completed as Proposition~\ref{prop:work} in Section~\ref{sec:appendix}.  After eliminating those cases in which $\K$ has small degree or in which the degree and genus of $\K$ do not satisfy the equation in Theorem~\ref{thm:thom}, there are 32 cases remaining (as shown in Table~\ref{table1}).  Of these 32 cases, 24 of them are contained in classes (B), (C), (F), and (G) above, each of which covers 6 cases related by cyclic permutations and reversing orientation.  The remaining 8 cases are contained in the classes (A), (D), (E), and (H), which are symmetric under cyclic permutations, and each of which covers 2 cases related by reversing orientation.
\end{proof}

We will let $(A)_d$ denote the degree $d$ diagram from family (A), and we notate diagrams from other families similarly.

\begin{remark}\label{rmk:positive}
Unfortunately, we include the ``positive genus" qualification above because spheres of degree one and two do not fit nicely into this scheme.  For example, consider the hexagonal lattice diagram shown in Figure~\ref{fig:spin}.  The result of any number of positive or negative Dehn twists of the arc in $\bbb$ about the gray annulus (a neighborhood of $\n$ in $\Sigma$) yields a hexagonal lattice diagram, and it can be shown that all of these diagrams are slide equivalent (defined in Section~\ref{sec:questions}).  While it is likely that Theorem~\ref{thm:main1} holds for these surfaces as well, the cases are too numerous to pleasantly exhaust.  Moreover, Theorem~\ref{thm:PLC} and Corollary~\ref{cor:17}, described in detail in the next section, imply that any hexagonal lattice diagram for a 2-sphere $\K$ with $d=1,2$ which satisfies the additional condition of admitting a transverse orientation is isotopic to $\C_1$ or $\C_2$, and this condition is straightforward to verify in practice.
\end{remark}

\begin{figure}[h!]
	\centering
	\includegraphics[width=.3\textwidth]{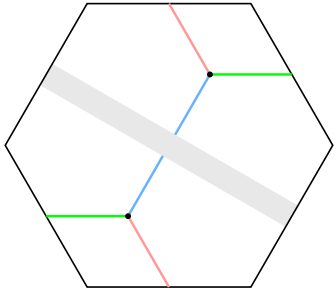}
	\caption{A family of hexagonal lattice diagrams for 2-spheres can be obtained by Dehn twisting the single arc in $\bbb$ about the curve $\n$}
	\label{fig:spin}
\end{figure}

\section{Tools for construction and classification}\label{sec:tools}

In this section, we develop the tools needed to prove that the surfaces $\K$ resulting from the eight families of hexagonal lattice diagrams included in Theorem~\ref{thm:family} are isotopic to complex curves $\C_d$.  Our general strategy is to overlap two hexagonal lattice diagrams and to show that the diagram obtained by smoothing all intersection points describes the surface obtained by resolving all intersections of the corresponding surfaces.  First, we invoke work of Lambert-Cole to help understand cases with small degrees.

\subsection{Transverse bridge trisections and symplectic surfaces}

In~\cite{PLCsymp}, Lambert-Cole defined transverse bridge position to study symplectic surfaces in $\CP^2$ (with respect to the standard symplectic form on $\CP^2$, the Fubini-Study K\"ahler form $\omega_{FS}$).  We include the definition of a transverse bridge position here:  The standard trisection $\CP^2 = X_1 \cup X_2 \cup X_3$ induces three different foliations of the central surface $\Sigma$, one whose leaves are parallel copies of $\A$ curves, one whose leaves are $\n$ curves, and one whose leaves are $\g$ curves.  These foliations can be given a transverse orientation, as shown in Figure~\ref{fig:transverse}.

\begin{figure}[h!]
	\centering
	\includegraphics[width=.3\textwidth]{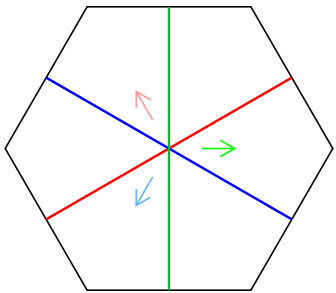}
	\caption{Transverse orientations of the three foliations of $\Sigma$}
	\label{fig:transverse}
\end{figure}

A bridge trisection $\T$ for a surface $\K$ with shadow diagram $(\aaa,\bbb,\ccc)$ is said to be a \emph{transverse bridge position} if the arcs of $\aaa$ are transverse to the $\A$-foliation of $\Sigma$, the arcs of $\bbb$ are transverse to the $\n$-foliation of $\Sigma$, the arcs of $\ccc$ are transverse to the $\g$-foliation, the orientations of all arcs agree with their respective foliations, and each bridge point has sign $\eps = +1$.  Lambert-Cole proved Theorem~\ref{thm:PLC}, that for a surface $\K \subset \CP^2$ that minimizes genus in its homology class, $\K$ is isotopic to a symplectic surface if and only if $\K$ admits a transverse bridge position.

As mentioned in the introduction, the symplectic isotopy problem is known to be true for surfaces of degree $d \leq 17$, and thus as a corollary to the previous theorem, we have

\begin{corollary}\label{cor:17}
Suppose that $\K \subset \CP^2$ is genus-minimizing in its homology class, the degree $d$ of $\K$ satisfies $d \leq 17$, and $\K$ admits a transverse bridge position.  Then $\K$ is isotopic to a complex surface $\C_d$.  In particular, this is true if $\K$ admits a transversely oriented hexagonal lattice diagram and $d \leq 17$.
\end{corollary}

\subsection{Smoothing and resolutions}

Suppose $\abc$ and $\abcp$ are two shadow diagrams considered on the same surface $\Sigma$ with distinct bridge points and such that arcs intersect transversely, with $\aaa \cap \aaa' = \bbb \cap \bbb' = \ccc \cap \ccc' = \emp$, so that the only intersecting shadows correspond to differently colored arcs in our diagrams.  We call the pair $\abc \cup \abcp$ \emph{overlapping shadow diagrams}.  See Figure~\ref{fig:overlap} for examples.

\begin{figure}[h!]
\begin{subfigure}{.48\textwidth}
  \centering
  \includegraphics[width=.6\linewidth]{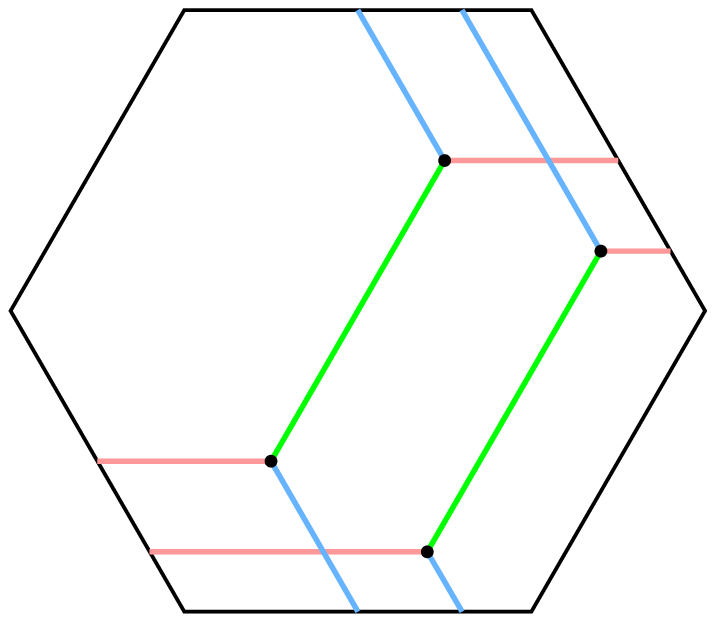}
  \caption{Two copies of $(D)_1$}
  \label{fig:over1}
\end{subfigure}
\begin{subfigure}{.48\textwidth}
  \centering
  \includegraphics[width=.6\linewidth]{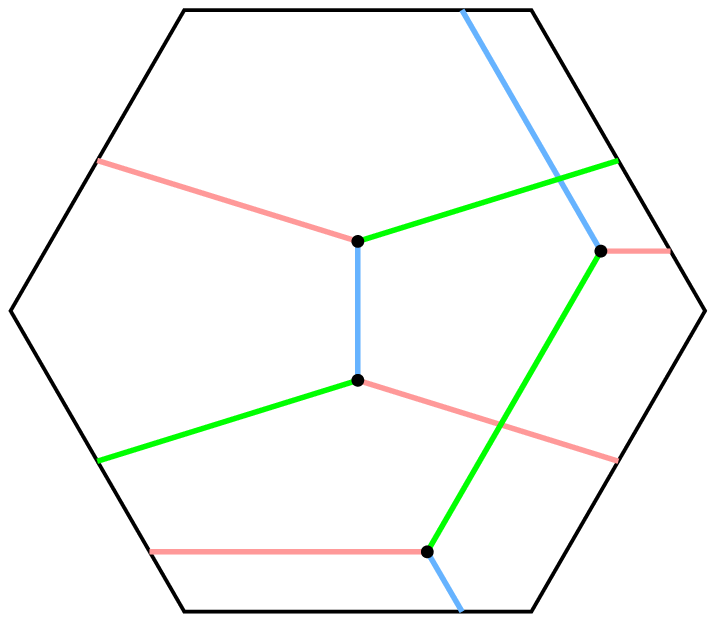}
  \caption{$(A)_2$ and $(D)_1$}
  \label{fig:over2}
\end{subfigure}
	\caption{Examples of overlapping shadow diagrams}
	\label{fig:overlap}
\end{figure}

Given overlapping shadow diagrams $\abc \cup \abcp$, suppose that two arcs, say $a \in \aaa$ and $b' \in \bbb'$ meet in a point $x$ in their interiors.  The \emph{smoothing} of $\abc \cup \abcp$ at $x$ replaces $a \cup b'$ with the arcs obtained by performing the oriented smoothing of $a \cup b'$ in a neighborhood of $x$, as shown from right to left in Figure~\ref{fig:smooth}.  If $\abcs$ is the triple obtained by smoothing each intersection point of $\abc \cup \abcp$, we say that $\abcs$ is the \emph{smoothing} of $\abc \cup \abcp$.

\begin{remark}
A priori, we have no guarantee that the smoothing $\abcs$ of a given overlapping shadow diagram $\abc \cup \abcp$ is itself a shadow diagram, since $\aaa^* \cup (-\bbb^*)$ may or may not determine an unlink in $\pd X_1$, for instance.  In our particular cases, however, we will show that the smoothings in question do indeed yield additional hexagonal lattice diagrams.
\end{remark}

Overlapping shadow diagrams give rise to a pair of embedded surfaces in $\CP^2$ that intersect transversely in some number of points.  To resolve these intersections, we will take the viewpoint presented in Section 2.1 of~\cite{GS}:  If $S$ and $S'$ are two oriented surfaces in a 4-manifold $X$ that intersect transversely in a point $y$, then a small 4-ball neighborhood $N$ of $y$ has the property that $\pd N$ intersects $S \cup S'$ in a Hopf link.  Additionally, the surface $S^*$ obtained by removing $(S \cup S') \cap N$ from $S \cup S'$ and replacing it with an annulus in $\pd N$ bounded by $(S \cup S') \cap \pd N$ is called the \emph{resolution} of $S \cup S'$ at the point $y$.  Paying heed to orientations, there is a unique such resolution up to isotopy.  In this context, the surface $\C_d$ can be obtained by taking $d$ parallel copies of $\C_1$, such that pairs meet generically in a single point, and resolving the $d(d-1)/2$ intersections to obtain the smooth surface $\C_d$.  (Alternatively, $\C_d$ can be obtained by taking generic copies of $\C_j$ and $\C_k$ that meet in $jk$ points and resolving intersections, where $d = j+k$.)

We can also understand resolution from a 3-dimensional perspective:  Suppose that two disks $D$ and $D'$ in $S^3$ have a clasp intersection, shown at left in Figure~\ref{fig:clasp}.  Viewing $S^3$ as $\pd B^4$ and pushing $D$ and $D'$ slightly into $\text{int}(B^4)$, we can turn the clasp intersection into a transverse intersection point, whose resolution yields an annulus $A$ such that $A$ can be isotoped back into $S^3$, as shown at right in Figure~\ref{fig:clasp}.

\begin{figure}[h!]
	\centering
	\includegraphics[width=.4\textwidth]{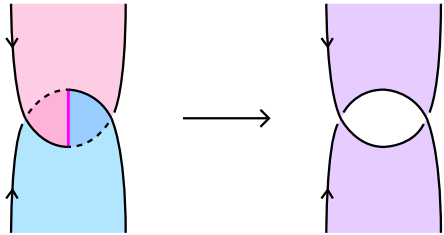}
	\caption{Resolving a clasp intersection of two disks to obtain an annulus}
	\label{fig:clasp}
\end{figure}

\section{The eight families determine complex curves}\label{sec:fam}

In this section, we prove that the surfaces $\K$ resulting from the eight families of hexagonal lattice diagrams included in Theorem~\ref{thm:family} are isotopic to complex curves $\C_d$.  Recall that $(A)_d$ denotes the degree $d$ diagram from family (A), and we notate diagrams from other families similarly.  Additionally, all diagrams in each of the eight families correspond to surfaces that minimize genus in their homology classes.  Our general strategy is to proceed inductively, showing that each hexagonal lattice diagram is the smoothing $\abcs$ of two particular diagrams of smaller degree.  Since a hexagonal lattice diagram $(\aaa,\bbb,\ccc)$ is determined uniquely by the homology classes of $\ab$ and $\bc$, and smoothing intersections preserves these homology classes, it is straightforward to verify that certain smoothings are in the desired classes.

\subsection{The families $(D)_d$ and $(H)_d$}

\begin{lemma}\label{lem:dh}
Each diagram $(D)_d$ and $(H)_d$ corresponds to a surface isotopic to a complex curve $\C_d$.
\end{lemma}

\begin{proof}
First, consider the diagram $(D)_1$, shown at left in Figure~\ref{fig:D}.  Note that $(D)_1$ can be transversely oriented, and thus by Corollary~\ref{cor:17}, its corresponding surface $\K$ is isotopic to $\C_1$.  Now, suppose by way of induction that the diagram $(D)_{d-1}$ can be obtained by taking $d-1$ parallel copies of $(D)_1$ and resolving intersections, and that the corresponding surface $\K'$ is isotopic to $\C_{d-1}$.  Consider the overlapping shadow diagrams $\abc$ and $\abcp$, where $\abc$ is the diagram $(D)_1$ and $\abcp$ is the diagram $(D)_{d-1}$, both transversely oriented and positioned as shown at center in Figure~\ref{fig:D}.  The smoothing $\abcs$ of $\abc \cup \abcp$ is shown at right in Figure~\ref{fig:D}.  We will prove that $\abcs$ is the diagram $(D)_d$ and corresponds the the surface obtained by resolving the $d-1$ intersections of $\K$ and $\K'$.

\begin{figure}[h!]
\begin{subfigure}{.32\textwidth}
  \centering
  \includegraphics[width=.9\linewidth]{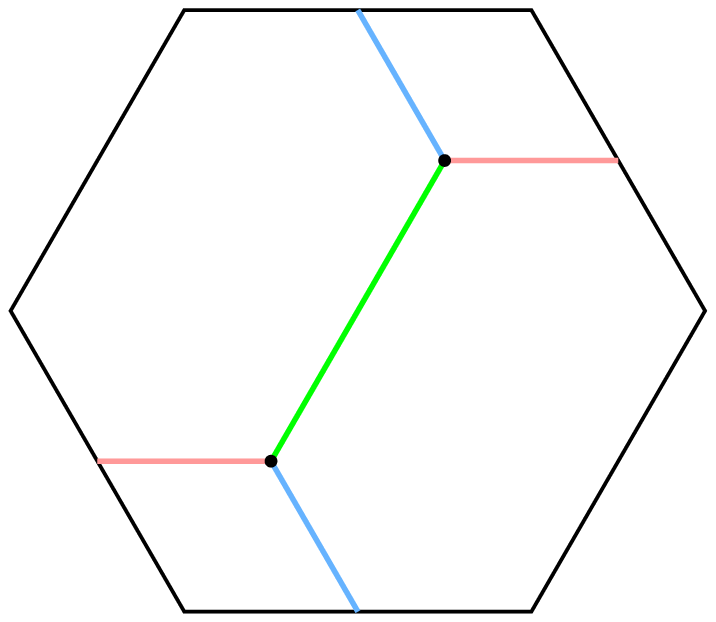}
  \label{fig:D1}
  \caption{$(D)_1$}
\end{subfigure}
\begin{subfigure}{.32\textwidth}
  \centering
  \includegraphics[width=.9\linewidth]{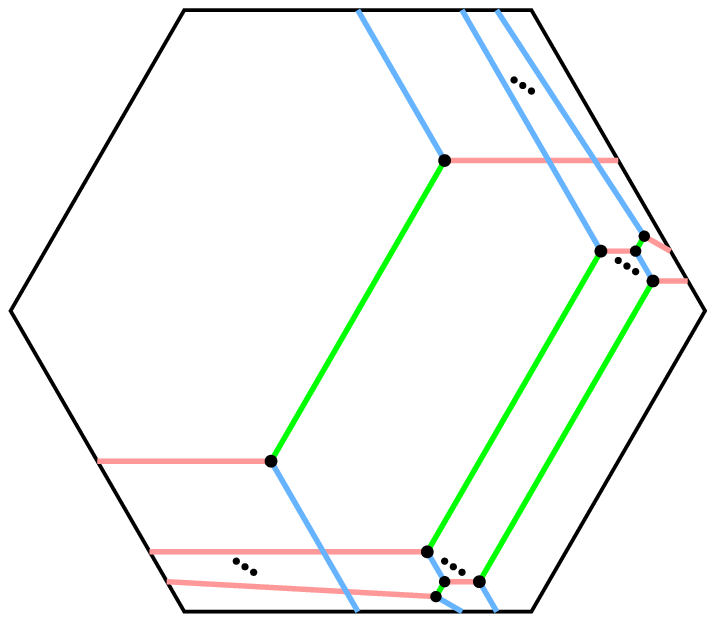}
  \label{fig:D2}
  \caption{$(D)_1 \cup (D)_{d-1}$}
\end{subfigure}
\begin{subfigure}{.32\textwidth}
  \centering
  \includegraphics[width=.9\linewidth]{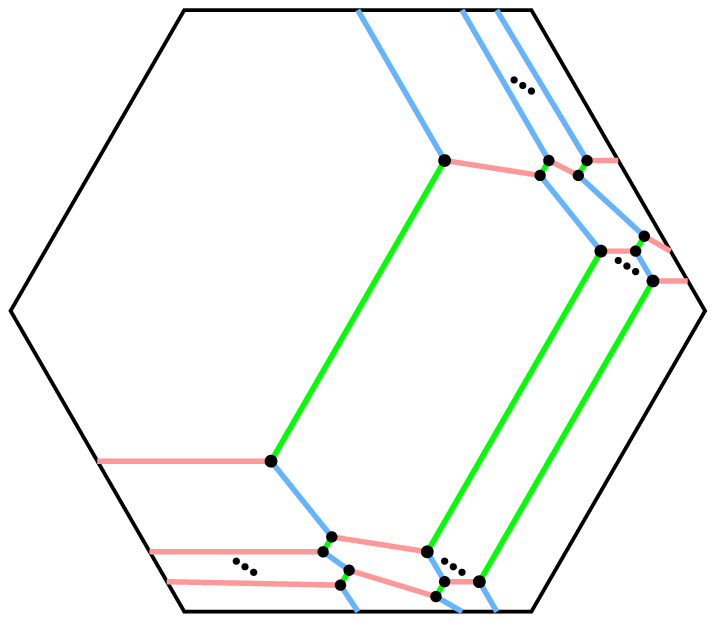}
  \label{fig:D3}
  \caption{Smoothing of $(D)_1 \cup (D)_{d-1}$}
\end{subfigure}
	\caption{}
	\label{fig:D}
\end{figure}

First, observe that $\abcs$ is indeed a hexagonal lattice diagram, and that as elements of $H_1(\Sigma)$,
\[\aaa^*\bbb^* = \ab + \abp = d\n \qquad \text{ and } \bbb^*\ccc^* = \bc + \bcp = d\g,\]
confirming that $(\aaa^*,\bbb^*,\ccc^*)$ is a hexagonal lattice diagram that agrees with $(D)_d$.  Now, observe that $\ab$ contains a single curve $C$, while $\abp$ contains $d-1$ parallel curves $C'_1,\dots,C'_{d-1}$, numbered in order so that highest red arc of $\aaa'$ at center in Figure~\ref{fig:D} belongs to $C'_1$, and so that $C \cup C'_1$ cuts out a bigon from $\Sigma$ that meets no other curves $C_i'$.  We see that curves of $\bc$ and $\bcp$ bound a collection of disjoint meridian disks of $H_3$, and curves of $\ca$ and $\capp$ bound a collection of disjoint meridian disks of $H_1$.  If we push arcs of $\aaa$ and $\aaa'$ that cross arcs of $\bbb'$ and $\bbb$ slightly into $H_1$, then for each $i$, the link $C \cup C_i'$ is a Hopf link in $\pd X_1$, and the corresponding patches $D$ and $D_i'$ can be chosen to meet transversely in a single point in $X_1$.  The sum total of these $d-1$ intersections account for the $d-1$ generic intersections of $\K$ and $\K'$.

Let $a \in \aaa$, $b \in \bbb$, $a' \in \aaa'$, and $b' \in \bbb'$ denote the arcs of $C$ and $C_1'$ corresponding to the two intersections of $C$ and $C_1'$ in $\Sigma$.  Note that $[C] = [C_1'] = [\n]$, and thus both curves bound meridian disks of $H_2$.  Pushing $a$ and $a'$ slightly into $H_1$ near the intersection points, we can isotope $D$ and $D_1'$ into $\pd X_1$ so that they intersect in a clasp as shown at far left in Figure~\ref{fig:clasp1}.  Let $A$ denote the annulus in $\pd X_1$ obtained by resolving the clasp intersection of $D$ and $D_1'$.  We embed $A$ in $\pd X_1 = H_1 \cup H_2$ so that $A \cap H_2$ is two disks connected by a pair of half-twisted bands $B$ and $B'$ in $H_1$ along $a \cup b'$ and $b \cup a'$.  Moreover, we can isotope two small rectangles $R \subset B$ and $R' \subset B'$ into the surface $\Sigma$, as shown at center left in Figure~\ref{fig:clasp1}.

\begin{figure}[h!]
\begin{subfigure}{.24\textwidth}
  \centering
  \includegraphics[width=.95\linewidth]{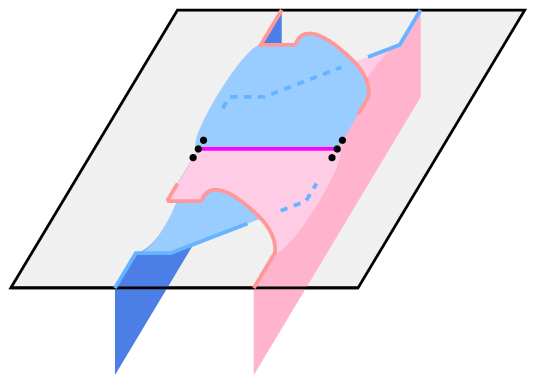}
  \label{fig:C1A}
\end{subfigure}
\begin{subfigure}{.24\textwidth}
  \centering
  \includegraphics[width=.95\linewidth]{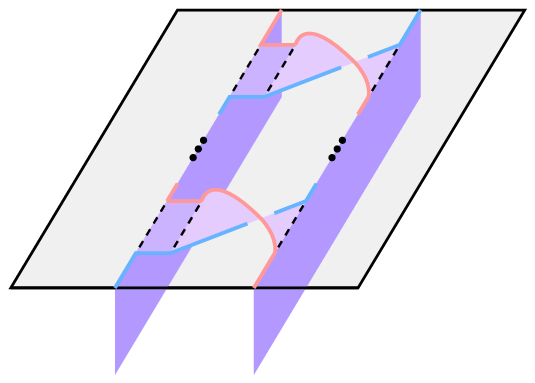}
  \label{fig:C1B}
\end{subfigure}
\begin{subfigure}{.24\textwidth}
  \centering
  \includegraphics[width=.95\linewidth]{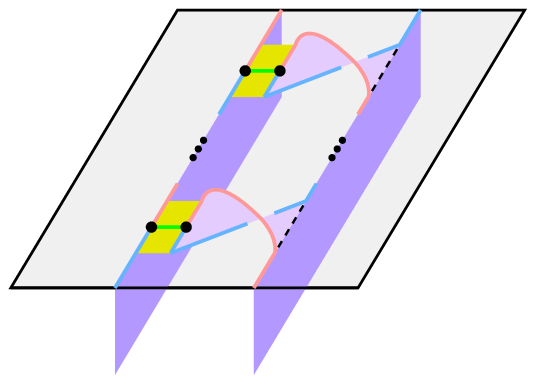}
    \label{fig:C1C}
\end{subfigure}
\begin{subfigure}{.24\textwidth}
  \centering
  \includegraphics[width=.95\linewidth]{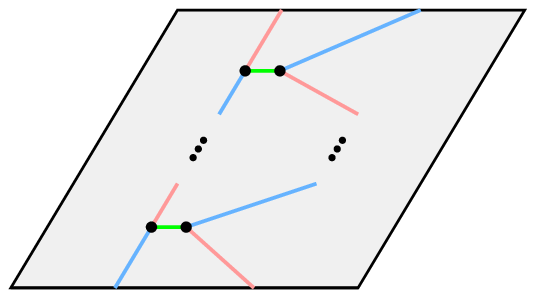}
    \label{fig:C1D}
\end{subfigure}
	\caption{At far left, the disks $D$ and $D_1'$ meet in a clasp.  At center left, the annulus $A$ consists of two disks in $H_2$ connected by half-twisted bands $B$ and $B'$ in $H_1$.  At center right, the rectangles $R$ and $R'$ are pushed into $X_2 \cup X_3$.  At far right, the induced shadows agree with smoothing the crossings.}
	\label{fig:clasp1}
\end{figure}

Now, isotope the rectangles $R$ and $R'$ away from $\pd X_1$ and into $\text{int}(X_2 \cup X_3)$.  Locally, this process is identical to that of bridge trisection perturbation, which is shown in Figures 23 and 24 of~\cite{MZB1}.  This isotopy excises the rectangles $R$ and $R'$ from $A$, attaching half of each rectangle to a patch along $\bc$, $\bcp$, $\ca$, or $\capp$, so that the two disk components of $A \setminus (R \cup R')$ remain in $X_1$, splitting $a$, $a'$, $b$, and $b'$, and adding two shadow new arcs, one across the middle of each rectangle, to $\ccc^*$ as shown at center right in Figure~\ref{fig:clasp1}.  Finally, after isotoping all arcs back into $\Sigma$ so that they meet efficiently, we see that at the level of shadow diagrams, the end result of this process is that the two intersection points of $C$ and $C_1'$ have been smoothed, as shown at far right in Figure~\ref{fig:clasp1}.

As a result of this modification, there is a curve $C^*$ obtained by smoothing $C$ and $C_1'$ that meets each of $C_2',\dots,C_{d-1}'$ in two points.  With $C^*$ taking the place of $C$, we can repeat the above argument with $C^*$ and $C_2'$, continuing in this manner until $d-1$ points of $\K \cap \K'$ have been resolved, and the resulting diagram $\abcs$, obtained by smoothing the $2(d-1)$ intersection points of $\abc \cup \abcp$, represents the embedded surface $\K^*$, obtained by smoothing the $d-1$ intersections of $\K$ and $\K'$

A similar proof follows for the diagrams $(H)_d$, starting with $(H)_1$ as shown at left in Figure~\ref{fig:H} and proceeding analogously, with overlapping diagram $(H)_1 \cup (H)_{d-1}$ shown at center and the smoothing $(H)_d$ shown at right in the same figure.
\end{proof}

\begin{figure}[h!]
\begin{subfigure}{.32\textwidth}
  \centering
  \includegraphics[width=.9\linewidth]{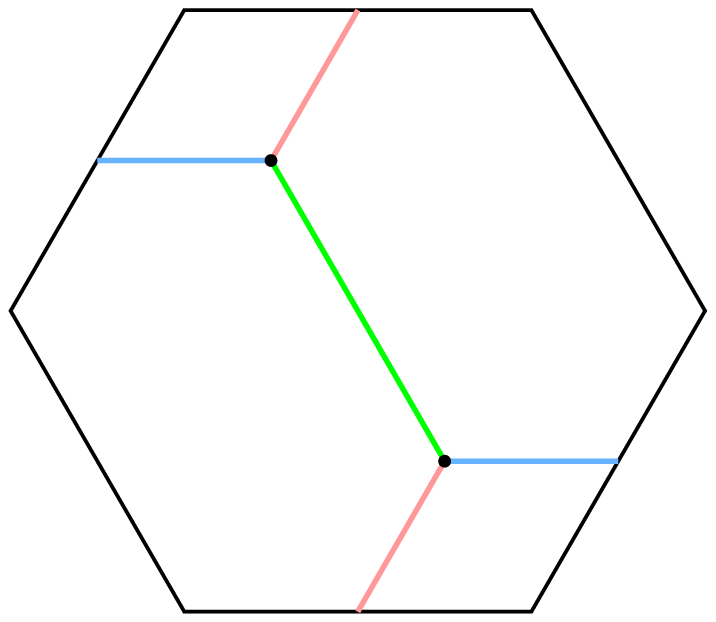}
  \label{fig:H1}
  \caption{$(H)_1$}
\end{subfigure}
\begin{subfigure}{.32\textwidth}
  \centering
  \includegraphics[width=.9\linewidth]{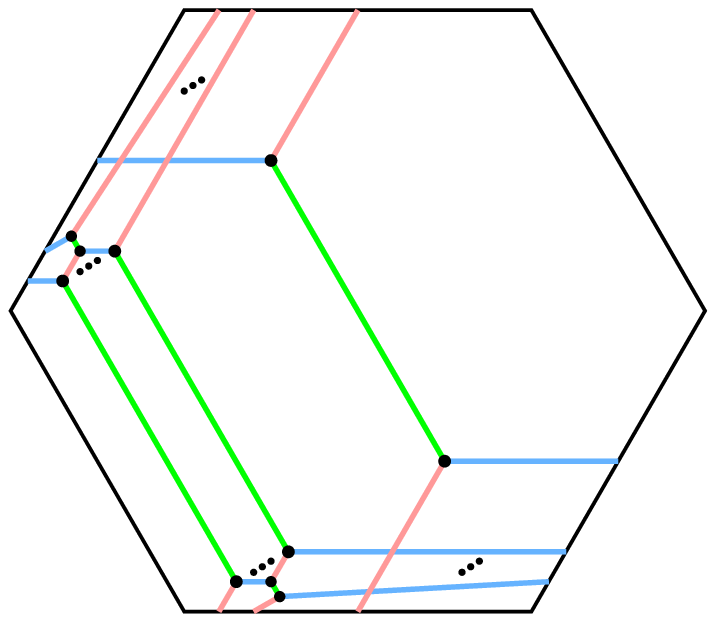}
  \label{fig:H2}
  \caption{$(H)_1 \cup (H)_{d-1}$}
\end{subfigure}
\begin{subfigure}{.32\textwidth}
  \centering
  \includegraphics[width=.9\linewidth]{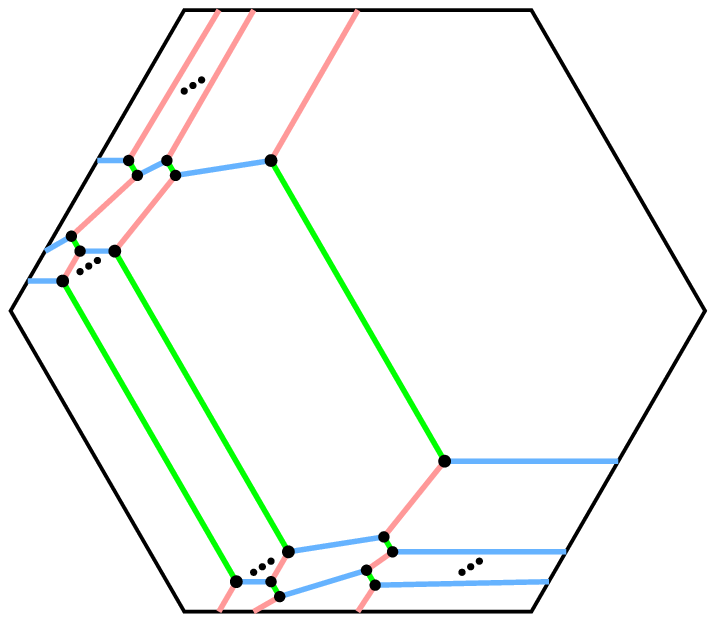}
  \label{fig:H3}
  \caption{Smoothing of $(H)_1 \cup (H)_{d-1}$}
\end{subfigure}
	\caption{}
	\label{fig:H}
\end{figure}

\subsection{The families $(B)_d$ and $(F)_d$}

\begin{lemma}\label{lem:bf}
Each diagram $(B)_d$ and $(F)_d$ corresponds to a surface isotopic to a complex curve $\C_d$.
\end{lemma}

\begin{proof}
Note that the first occurring diagram in this family is $(B)_2$, since $(d-1)\n$ must be nonzero for the construction to make sense.  A depiction of $(B)_2$ is shown at left in Figure~\ref{fig:B}.  Note that $(B)_2$ admits a transverse orientation, and thus by Corollary~\ref{cor:17}, its corresponding surface $\K$ is isotopic to $\C_2$.  Consider the overlapping shadow diagrams $\abc \cup \abcp$ as shown at center in Figure~\ref{fig:B}, where $\abc$ is the diagram $(B)_2$ and $\abcp$ is the transversely oriented diagram $(D)_{d-2}$, representing the surface $\K'$, which is isotopic to $\C_{d-2}$ by Lemma~\ref{lem:dh}.  The smoothing $\abcs$ is shown at right in Figure~\ref{fig:B}.  We will prove that $\abcs$ is the diagram $(B)_d$ and corresponds to the surface obtained by resolving the $2(d-2)$ intersections of $\K$ and $\K'$.

\begin{figure}[h!]
\begin{subfigure}{.32\textwidth}
  \centering
  \includegraphics[width=.9\linewidth]{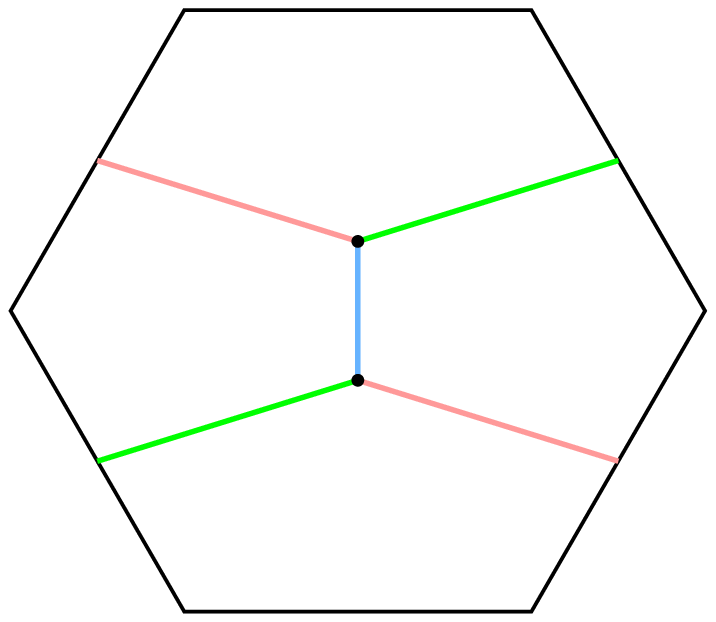}
  \label{fig:B1}
  \caption{$(B)_2$}
\end{subfigure}
\begin{subfigure}{.32\textwidth}
  \centering
  \includegraphics[width=.9\linewidth]{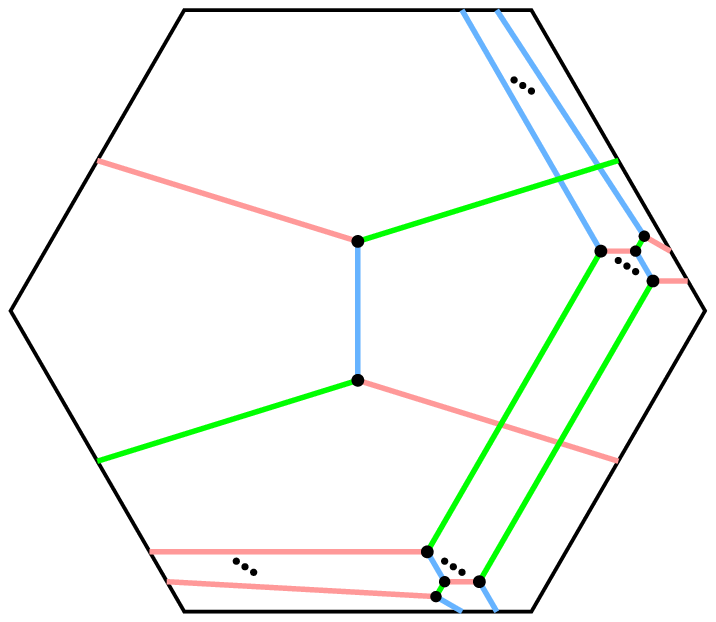}
  \label{fig:B2}
  \caption{$(B)_2 \cup (D)_{d-2}$}
\end{subfigure}
\begin{subfigure}{.32\textwidth}
  \centering
  \includegraphics[width=.9\linewidth]{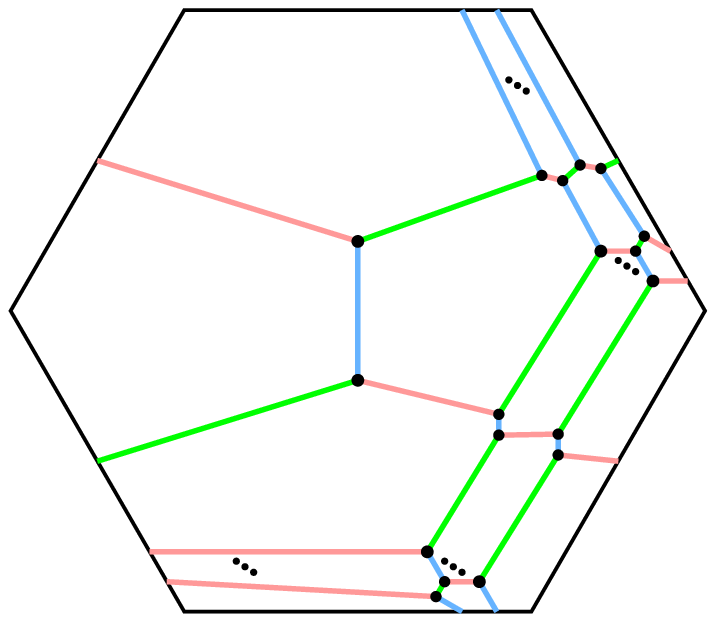}
  \label{fig:B3}
  \caption{Smoothing of $(B)_2 \cup (D)_{d-2}$}
\end{subfigure}
	\caption{}
	\label{fig:B}
\end{figure}

Observe that $\abcs$ is indeed a hexagonal lattice diagram, and that as elements of $H_1(\Sigma)$,
\[\aaa^*\bbb^* = \ab + \abp = (d-1)\n \qquad \text{ and } \bbb^*\ccc^* = \bc + \bcp = \n+(d-1)\g,\]
confirming that $\abcs$ is a hexagonal lattice diagram that agrees with $(B)_d$.  Curves of $\ab$ and $\abp$ bound a collection of disjoint meridian disks of $H_2$, so these disks contribute no point of intersection to $\K \cap \K'$.  Now, $\bc$ contains a single curve $C$, while $\bcp$ contains $d-2$ parallel curves $C'_1,\dots,C'_{d-2}$, numbered in order so that leftmost blue arc of $\bbb'$ at center in Figure~\ref{fig:B} belongs to $C'_1$.  In $\pd X_2 = H_2 \cup H_3$, each curve $C'_i$ bounds a meridian disk $D_i'$ of $H_3$, whereas the curve $C$ bounds a disk $D$ obtained by connecting a meridian disk of $H_2$ to a meridian of $H_3$ with a quarter-twisted band.  If we push the arc $c \in \ccc$ that crosses arcs of $\bbb'$ slightly into $H_3$ near the crossing, we see that for for each $i$, the link $C \cup C_i'$ is a Hopf link in $\pd X_2$, and so the corresponding disks $D$ and $D_i'$ meet once in $X_2$, the total of which account for $d-2$ of the intersections of $\K$ and $\K'$.

Let $b' \in \bbb'$ denote the arc of $C_1'$ that intersects $c$ in $\Sigma$, and let $A$ denote the annulus in $\pd X_2$ obtained by resolving the clasp intersection of $D$ and $D_1'$, where this clasp intersection is shown at left in Figure~\ref{fig:clasp2}.  We isotope $A$ so that near the intersection point of $b'$ and $c$, we have that $A$ meets $\Sigma$ in two arcs connecting $b'$ and $c$, as shown at center in Figure~\ref{fig:clasp2}.  Next, isotope $A$ so that a small rectangular neighborhood $R$ of one of these arcs is leveled into $\Sigma$, as shown at right Figure~\ref{fig:clasp2}.

\begin{figure}[h!]
\begin{subfigure}{.24\textwidth}
  \centering
  \includegraphics[width=.95\linewidth]{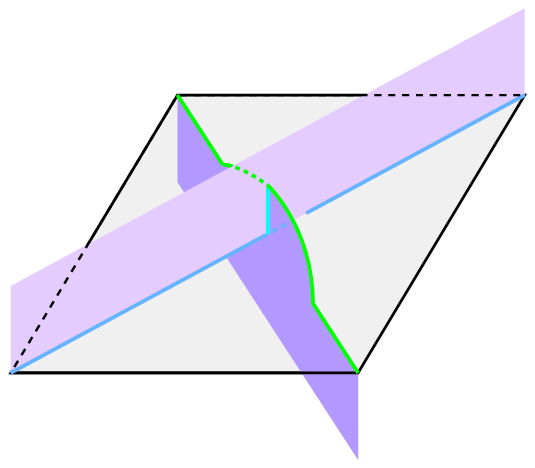}
  \label{fig:C2A}
\end{subfigure}
\begin{subfigure}{.24\textwidth}
  \centering
  \includegraphics[width=.95\linewidth]{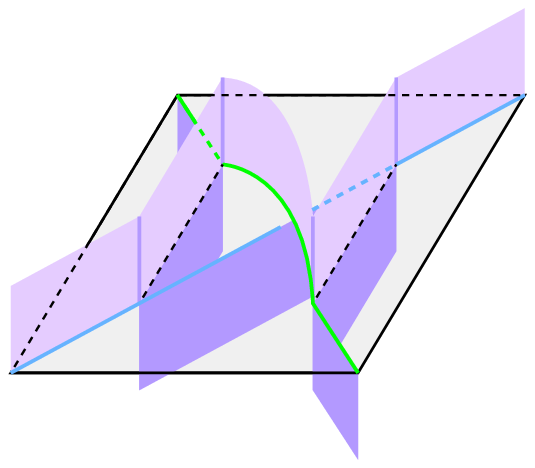}
  \label{fig:C2B}
\end{subfigure}
\begin{subfigure}{.24\textwidth}
  \centering
  \includegraphics[width=.95\linewidth]{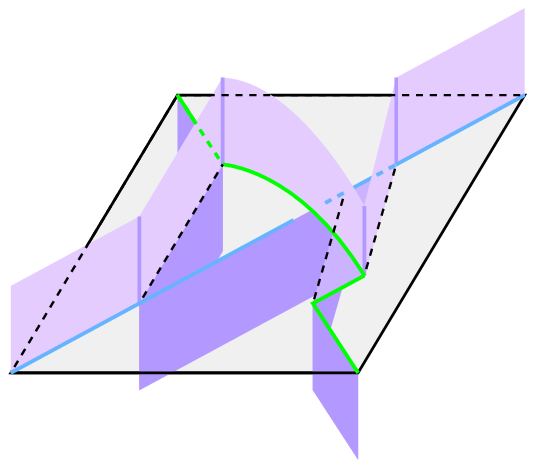}
    \label{fig:C2C}
\end{subfigure}
	\caption{At left, the disks $D$ and $D_1'$ meet in a clasp.  At center, the annulus $A$ meets $\Sigma$ in two arcs in its interior.  At right, the rectangle $R \subset A$ is leveled into $\Sigma$.}
	\label{fig:clasp2}
\end{figure}

As in the proof of Lemma~\ref{lem:dh}, we push $R$ out of $\pd X_2$ and into $\text{int}(X_1 \cup X_3)$, where it is absorbed by patches of $\ca$ and $\abp$, and so that the disk $A \setminus R$ remains in $\pd X_2$.  This process also splits the arcs $b'$ and $c$ and adds a new shadow arc to $\aaa^*$, as shown at left in Figure~\ref{fig:T2}.  At right in the same figure, these arcs have been isotoped to remove inessential intersections, and the end result of this process is that the intersection point of $C$ and $C_1'$ has been smoothed.

\begin{figure}[h!]
\begin{subfigure}{.24\textwidth}
  \centering
  \includegraphics[width=.95\linewidth]{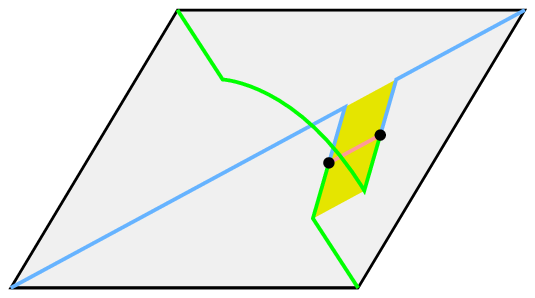}
    \label{fig:T2C}
\end{subfigure}
\begin{subfigure}{.24\textwidth}
  \centering
  \includegraphics[width=.95\linewidth]{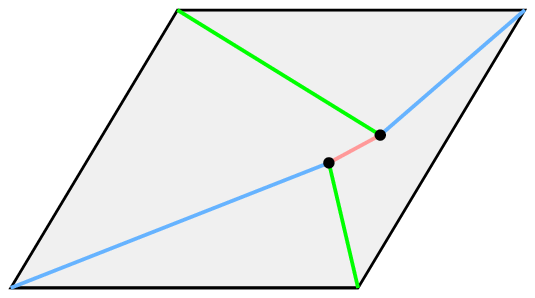}
    \label{fig:T2D}
\end{subfigure}
	\caption{At left, we see the shadows induced by pushing $R$ out of $\pd X_2$, and right, the final position of the induced shadows agrees with the smoothing of the crossing.}
	\label{fig:T2}
\end{figure}

The newly created curve $C^*$ has homology $C + C_1' = \n + 2\g$, so that it bounds a disk $D^*$ obtained by connecting a meridian for $H_2$ to two copies of a meridian for $H_3$ with quarter-twisted bands.  With $C^*$ taking the place of $C$, we can repeat the above argument with $C^*$ and $C_2'$, continuing in this manner until $d-2$ points of $\K \cap \K'$ have been resolved.  Note that the resolutions and smoothing are local modifications, leaving other components of the bridge trisection unaffected.  Finally, the curve $\ca$ meets each of $d-2$ curves $\capp$ in a single point, where the union of $\ca$ and any curve of $\capp$ is a Hopf link, and corresponding disks meet once in $X_3$.  Thus, the above argument can be repeated in $X_3$ a total of $d-2$ times to resolve the remaining $d-2$ intersection points of $\K$ and $\K'$ while simultaneously smoothing the remaining $d-2$ intersections.  We conclude that the resulting diagram $\abcs$, obtained by smoothing the $2(d-2)$ crossings of $\abc \cup \abcp$, represents the embedded surface $\K^*$, obtained by resolving the $2(d-2)$ intersections of $\K$ and $\K'$.

A similar proof follows for the diagrams $(F)_d$, starting with $(F)_2$ as shown at left in Figure~\ref{fig:F} and proceeding analogously, with overlapping diagram $(F)_2 \cup (H)_{d-2}$ shown at center and the smoothing $(F)_d$ shown at right in the same figure.
\end{proof}

\begin{figure}[h!]
\begin{subfigure}{.32\textwidth}
  \centering
  \includegraphics[width=.9\linewidth]{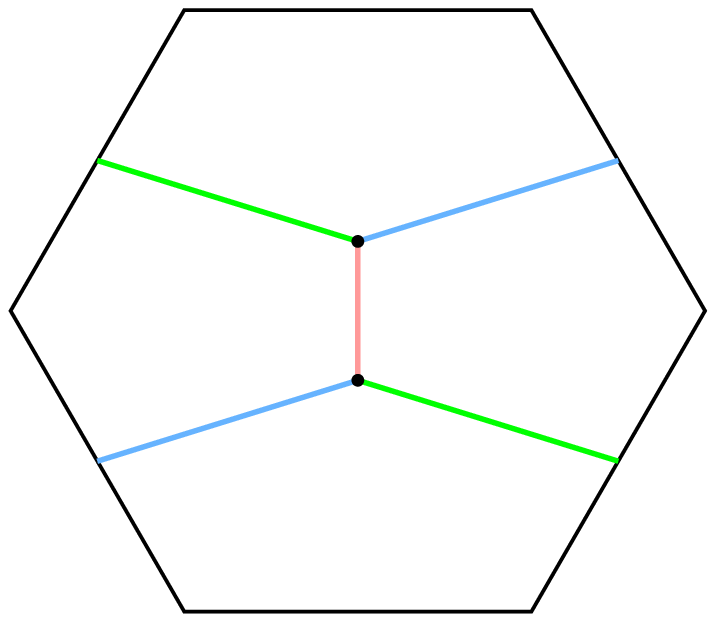}
  \label{fig:F1}
  \caption{$(F)_2$}
\end{subfigure}
\begin{subfigure}{.32\textwidth}
  \centering
  \includegraphics[width=.9\linewidth]{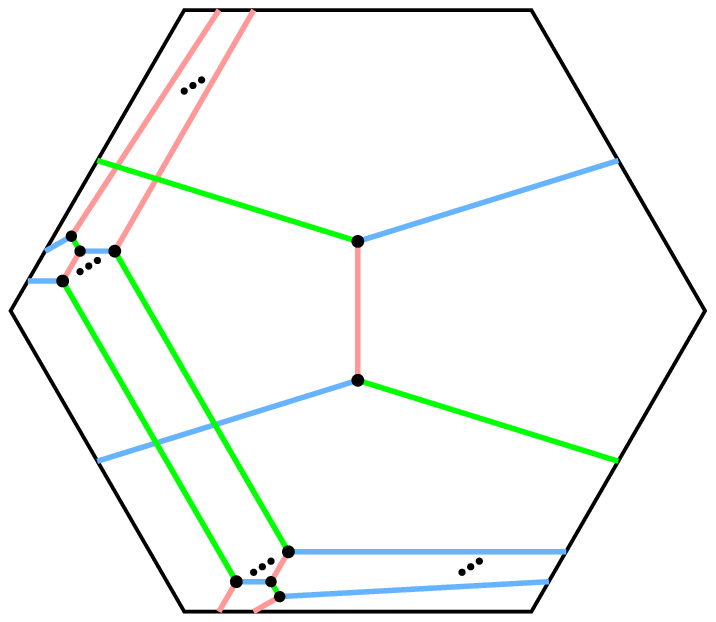}
  \label{fig:F2}
  \caption{$(F)_2 \cup (H)_{d-2}$}
\end{subfigure}
\begin{subfigure}{.32\textwidth}
  \centering
  \includegraphics[width=.9\linewidth]{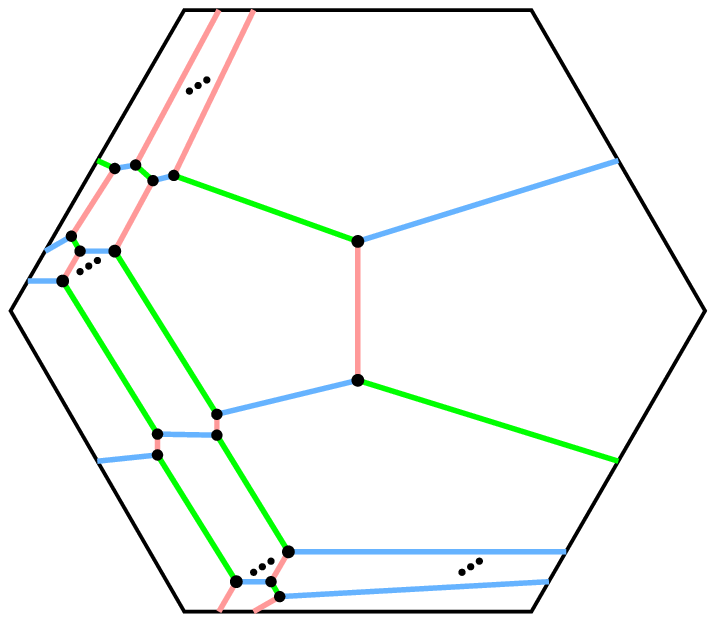}
  \label{fig:F3}
  \caption{Smoothing of $(F)_2 \cup (H)_{d-2}$}
\end{subfigure}
	\caption{}
	\label{fig:F}
\end{figure}

\subsection{The families $(C)_d$ and $(G)_d$}

\begin{lemma}\label{lem:cg}
Each diagram $(C)_d$ and $(G)_d$ corresponds to a surface isotopic to a complex curve $\C_d$.
\end{lemma}

\begin{proof}
As in the proof of Lemma~\ref{lem:bf}, the first occurring curve in the family $(C)_d$ is $(C)_2$.  The shadow diagram $(C)_2$ is shown at left in Figure~\ref{fig:C}, and it admits a transverse orientation, so that the surface $\K$ it represents is isotopic to $\C_2$ by Corollary~\ref{cor:17}.  Now, for any $d>2$, consider the overlapping shadow diagram $\abc \cup \abcp$ shown at center in Figure~\ref{fig:C}, where $\abc$ is the diagram $(C)_2$ and $\abcp$ is the transversely oriented diagram $(D)_{d-2}$ representing the surface $\K'$, which is isotopic to $\C_{d-2}$ by Lemma~\ref{lem:dh}.  The smoothing $\abcs$ is shown at right in Figure~\ref{fig:C}.  We will prove that $\abcs$ is the diagram $(C)_d$ and corresponds to the surface obtained by resolving the $2(d-2)$ intersections of $\K$ and $\K'$.

\begin{figure}[h!]
\begin{subfigure}{.32\textwidth}
  \centering
  \includegraphics[width=.9\linewidth]{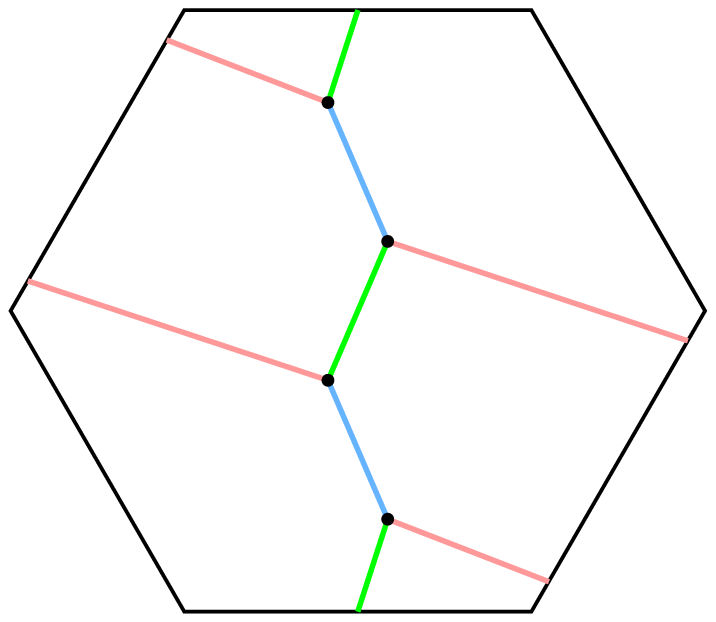}
  \label{fig:C1}
  \caption{$(C)_2$}
\end{subfigure}
\begin{subfigure}{.32\textwidth}
  \centering
  \includegraphics[width=.9\linewidth]{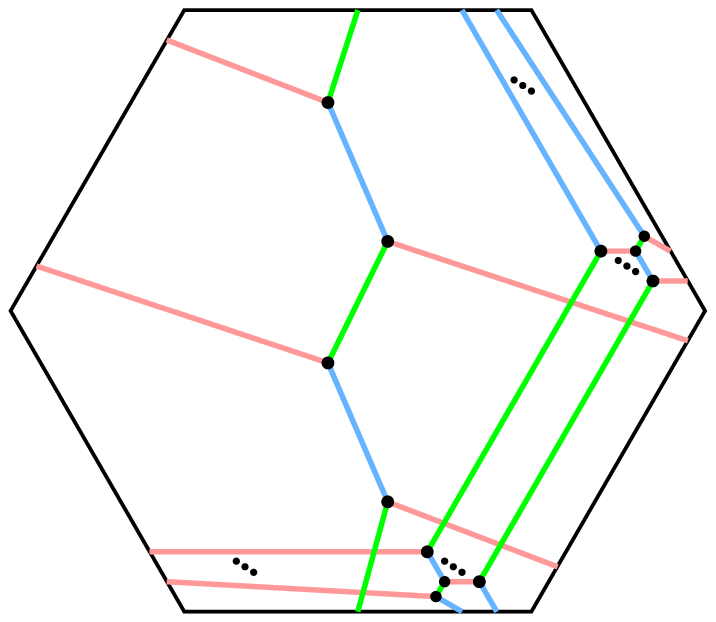}
  \caption{$(C)_2 \cup (D)_{d-2}$}
  \label{fig:C2}
\end{subfigure}
\begin{subfigure}{.32\textwidth}
  \centering
  \includegraphics[width=.9\linewidth]{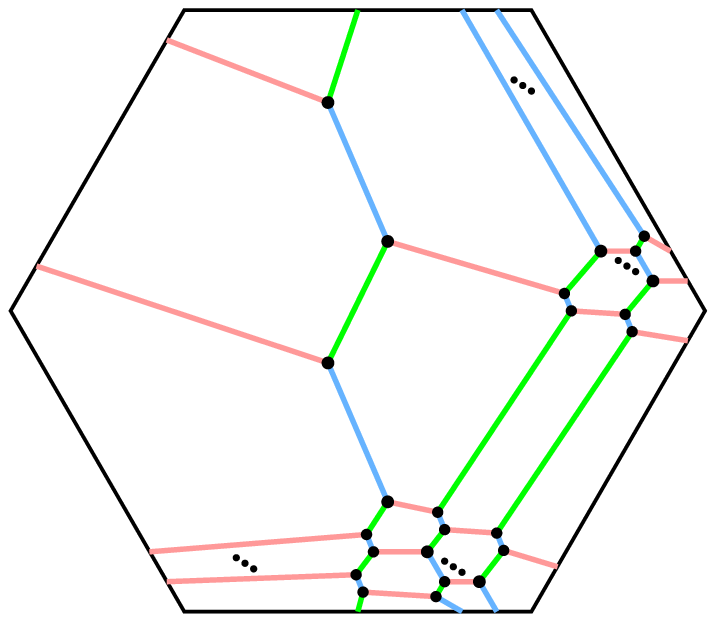}
  \label{fig:C3}
  \caption{Smoothing of $(C)_2 \cup (D)_{d-2}$}
\end{subfigure}
	\caption{}
	\label{fig:C}
\end{figure}

Observe that $\abcs$ is indeed a hexagonal lattice diagram, and that as elements of $H_1(\Sigma)$,
\[\aaa^*\bbb^* = \ab + \abp = d\n \qquad \text{ and } \bbb^*\ccc^* = \bc + \bcp = (d-1)\g,\]
confirming that $\abcs$ is a hexagonal lattice diagram that agrees with $(C)_d$.  Curves of $\ab$ and $\abp$ bound a collection of disjoint meridian disks of $H_2$, and curves of $\bc$ and $\bcp$ bound a collection of disjoint meridian disks of $H_3$, so these disks contribute no intersection points to $\K \cap \K'$.  Now, $\ca$ contains a single curve $C$, while $\capp$ contains $d-2$ parallel curves $C'_1,\dots,C'_{d-2}$, numbered in order so that topmost red arc of $\aaa'$ at bottom center in Figure~\ref{fig:C} belongs to $C'_1$.  Suppose $C$ bounds a disk $D$ in $X_3$ and for each $i$, the curve $C_i'$ bounds a disk $D_i'$ in $X_3$.  In homology, $C = 2\A + \g$, and $C' = \A$, so that we can push $D$ into $\pd X_3$ so that it is obtained by connecting two meridians of $H_1$ to a meridian of $H_3$ with two quarter-twisted bands, while $D_i'$ is isotopic to a meridian of $H_3$.

Let $C \cap C_i' = \{x^i_1,x^i_2,x^i_3\}$, where $x^i_1$ and $x^i_2$ are vertices of a bigon as in the proof of Lemma~\ref{lem:dh}.  We can arrange the disks $D$ and $D_i'$ so that a collar of $C$ and $C_1'$ is contained in $H_3$, and thus $D$ and $D_1'$ have a clasp intersection near $x^1_1$ and $x^1_2$ as in the proof of Lemma~\ref{lem:dh}.  Following the local modifications described in that proof and shown in Figure~\ref{fig:clasp1}, we can resolve this intersection, which at the level of shadow diagrams has the effect of smoothing the crossings $x^1_1$ and $x^1_2$.  One of the resulting curves, $C^*$, intersects $C_2'$ at points $x^2_1$ and $x^2_2$, and we can repeat the process a total of $d-2$ times to resolve $d-2$ intersections and smooth $2(d-2)$ crossings of $\abc \cup \abcp$.

For the remaining $d-2$ crossings, observe that the smoothed $\ca$ and $\capp$ curves contain one curve, call it $\wh C$, with homology $\wh C = 2\A + \g$, and $d-2$ curves, call these $\wh C_i$, with homology $\wh C_i = \A$, where $\wh C$ meets each $\wh C_i$ in a single point $x^i_3$.  By pushing the arc of $\wh C$ containing $x^i_3$ slightly into $H_1$, the corresponding disks bounded by $\wh C$ and $\wh C_i$ in $\pd X_3$ meet in a clasp, locally identical that appearing at left in Figure~\ref{fig:clasp2}.  Following the proof of Lemma~\ref{lem:bf}, we can resolve the $d-2$ corresponding intersections as shown in Figures~\ref{fig:clasp2} and~\ref{fig:T2}, which at the level of shadow diagrams has the effect of smoothing the crossings $x^i_3$.  We conclude that the resulting diagram $\abcs$, obtained by smoothing the $3(d-2)$ crossings of $\abc \cup \abcp$, represents the embedded surface $\K^*$, obtained by resolving the $2(d-2)$ intersections of $\K$ and $\K'$.

A similar proof follows for the diagrams $(G)_d$, starting with $(G)_2$ as shown at left in Figure~\ref{fig:G} and proceeding analogously, with overlapping diagram $(G)_2 \cup (H)_{d-2}$ shown at center and the smoothing $(G)_d$ shown at right in the same figure.
\end{proof}

\begin{figure}[h!]
\begin{subfigure}{.32\textwidth}
  \centering
  \includegraphics[width=.9\linewidth]{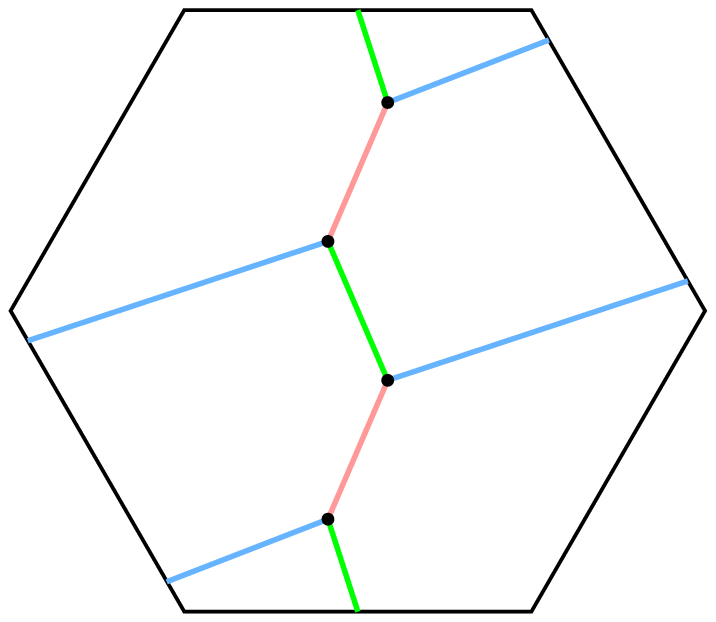}
  \label{fig:G1}
  \caption{$(G)_2$}
\end{subfigure}
\begin{subfigure}{.32\textwidth}
  \centering
  \includegraphics[width=.9\linewidth]{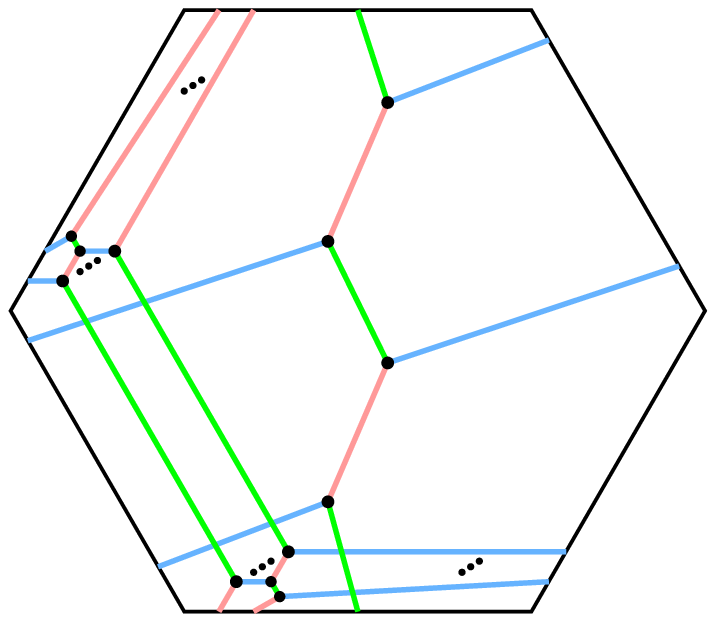}
  \label{fig:G2}
  \caption{$(G)_2 \cup (H)_{d-2}$}
\end{subfigure}
\begin{subfigure}{.32\textwidth}
  \centering
  \includegraphics[width=.9\linewidth]{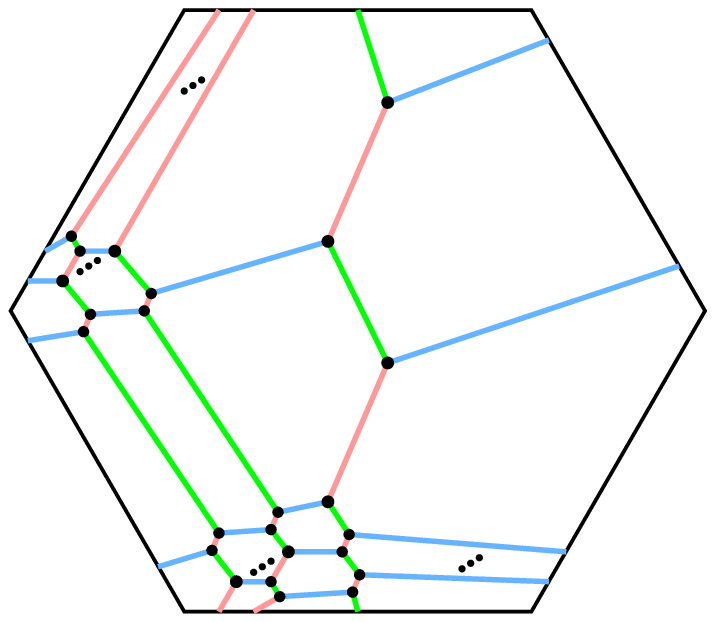}
  \label{fig:G3}
  \caption{Smoothing of $(G)_2 \cup (H)_{d-2}$}
\end{subfigure}
	\caption{}
	\label{fig:G}
\end{figure}

\subsection{The families $(A)_d$ and $(E)_d$}

\begin{lemma}\label{lem:ae}
Each diagram $(A)_d$ and $(E)_d$ corresponds to a surface isotopic to a complex curve $\C_d$.
\end{lemma}

\begin{proof}
Departing from the previous cases, we separate the cases $d=1,2$ from the cases $d \geq 3$.  Note that the diagram $(A)_1$ coincides with $(H)_1$ (pictured at left in Figure~\ref{fig:H}) and $(E)_1$ coincides with $(D)_1$ (pictured at left in Figure~\ref{fig:D}).  In addition, diagrams $(A)_2$ and $(E)_2$ are identical, and this diagram is pictured in Figure~\ref{fig:verBF2}.  The statement for the cases $d=1,2$ follows immediately from the observation that these diagrams admit transverse orientations and an application of Corollary~\ref{cor:17}.

\begin{figure}[h!]
	\centering
	\includegraphics[width=.3\textwidth]{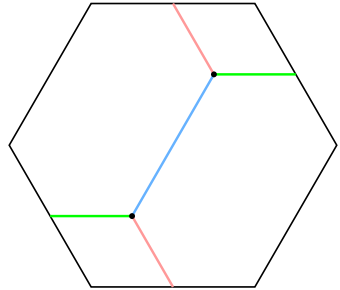}
	\caption{The diagram $(A)_2 = (E)_2$}
	\label{fig:verBF2}
\end{figure}

Turning our attention to the cases $d \geq 3$, the diagram $(A)_3$ is shown at left in Figure~\ref{fig:A}, where $(A)_3$ admits a transverse orientation, so this its corresponding surface $\K$ is isotopic to $\C_3$ by Corollary~\ref{cor:17}.  Now, for any $d>3$, consider the overlapping shadow diagram $\abc \cup \abcp$ shown at center in Figure~\ref{fig:C}, where $\abc$ is the diagram $(A)_3$ and $\abcp$ is the transversely oriented diagram $(D)_{d-3}$ representing the surface $\K'$, which is isotopic to $\C_{d-3}$ by Lemma~\ref{lem:dh}.  The smoothing $\abcs$ is shown at right in Figure~\ref{fig:A}.  We will prove that $\abcs$ is the diagram $(A)_d$ and corresponds to the surface obtained by resolving the $3(d-3)$ intersections of $\K$ and $\K'$.

\begin{figure}[h!]
\begin{subfigure}{.32\textwidth}
  \centering
  \includegraphics[width=.9\linewidth]{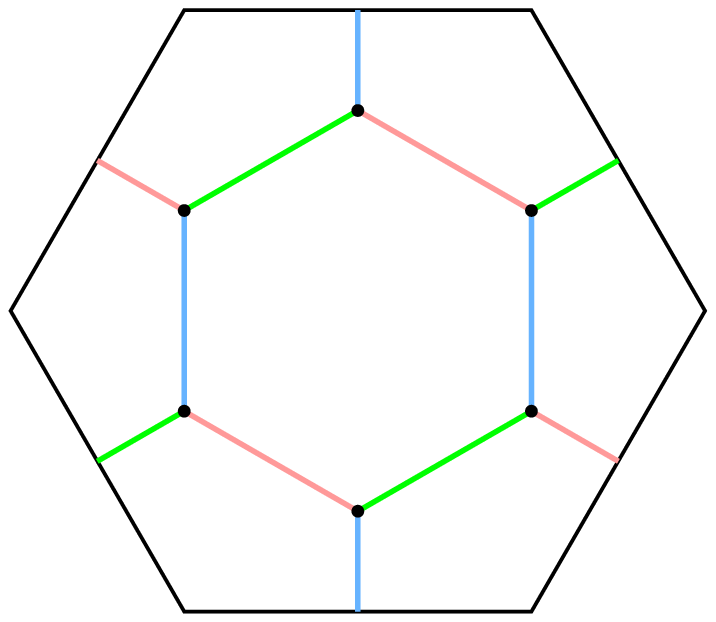}
  \label{fig:A1}
  \caption{$(A)_3$}
\end{subfigure}
\begin{subfigure}{.32\textwidth}
  \centering
  \includegraphics[width=.9\linewidth]{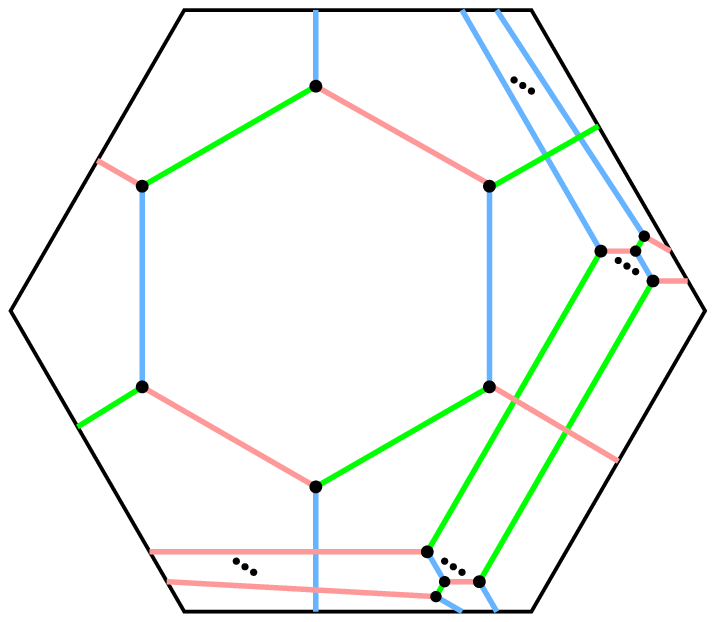}
  \caption{$(A)_3 \cup (D)_{d-3}$}
  \label{fig:A2}
\end{subfigure}
\begin{subfigure}{.32\textwidth}
  \centering
  \includegraphics[width=.9\linewidth]{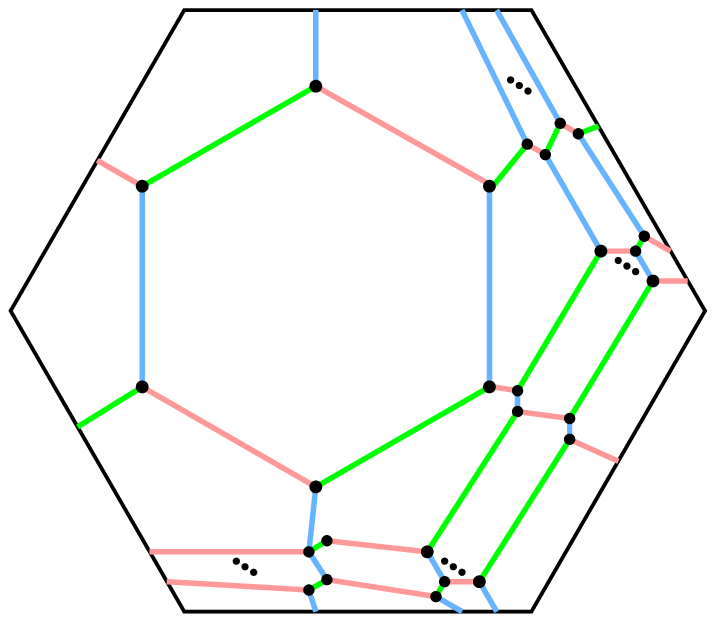}
  \label{fig:A3}
  \caption{Smoothing of $(A)_3 \cup (D)_{d-3}$}
\end{subfigure}
	\caption{}
	\label{fig:A}
\end{figure}

Observe that $\abcs$ is indeed a hexagonal lattice diagram, and that as elements of $H_1(\Sigma)$,
\[\aaa^*\bbb^* = \ab + \abp = \A + (d-1)\n \qquad \text{ and } \bbb^*\ccc^* = \bc + \bcp = \n + (d-1)\g,\]
confirming that $\abcs$ is a hexagonal lattice diagram that agrees with $(A)_d$.  Let $C$ be the curve in $\ab$ and let $C_1',\dots,C_{d-3}'$ be the $d-3$ curves of $\capp$.  In homology, $C = \A + 2\n$ and $C_i' = \n$, so that $C$ bounds a disk $D$ obtained by attaching a meridian of $H_1$ to two meridians of $H_2$ with quarter-twisted bands, and $C_i'$ bounds a meridian disk $D_i'$ of $H_2$.  Let $x^i_1$ denote the intersection point of $C \cap C_i'$.  Pushing the arc of $\aaa'$ slightly into $H_1$ near $x^i_1$, we see that $D$ and $D_i'$ have a clasp intersection locally equivalent to that pictured at left in Figure~\ref{fig:clasp2}.  By the process described in Lemma~\ref{lem:bf}, we can resolve the intersection of $D$ and $D_i'$, in the process smoothing the overlapping shadow diagram at the points $x^i_1$.

Note that this modification changes $\bc$, $\ca$, $\bcp$, and $\capp$ only by a slight isotopy.  The argument is symmetric, and so we can repeat this process using $\bc$ and $\bcp$ to resolve an additional $d-3$ intersections of $\K$ and $\K'$, smoothing $d-3$ more crossings of $\abc \cup \abcp$, and we can repeat once more using $\ca$ and $\capp$ to resolve the final $d-3$ intersections of $\K \cap \K'$ and smooth the final $d-3$ crossings of $\abc \cup \abcp$.  We conclude that the resulting diagram $\abcs$, obtained by smoothing the $3(d-3)$ crossings of $\abc \cup \abcp$, represents the embedded surface $\K^*$, obtained by resolving the $3(d-3)$ intersections of $\K$ and $\K'$.

A similar proof follows for the diagrams $(E)_d$, starting with $(E)_3$ as shown at left in Figure~\ref{fig:E} and proceeding analogously, with overlapping diagram $(E)_3 \cup (H)_{d-3}$ shown at center and the smoothing $(E)_d$ shown at right in the same figure.
\end{proof}

\begin{figure}[h!]
\begin{subfigure}{.32\textwidth}
  \centering
  \includegraphics[width=.9\linewidth]{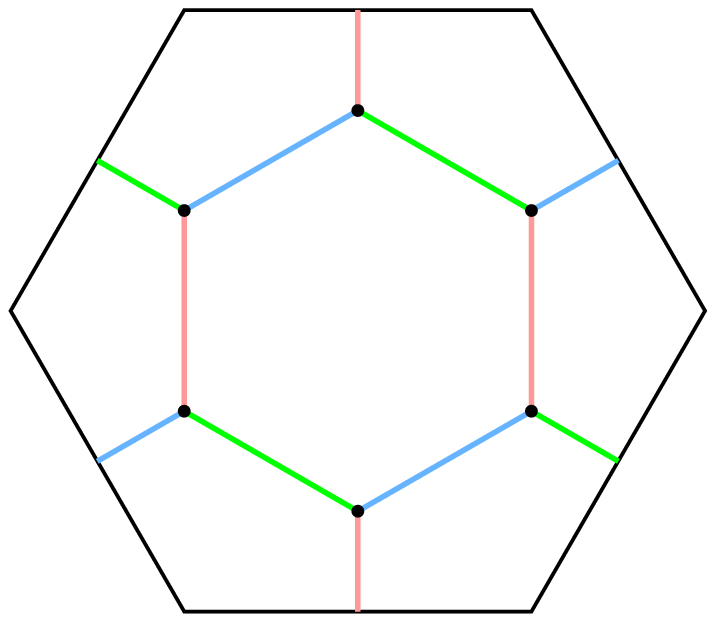}
  \label{fig:E1}
  \caption{$(E)_3$}
\end{subfigure}
\begin{subfigure}{.32\textwidth}
  \centering
  \includegraphics[width=.9\linewidth]{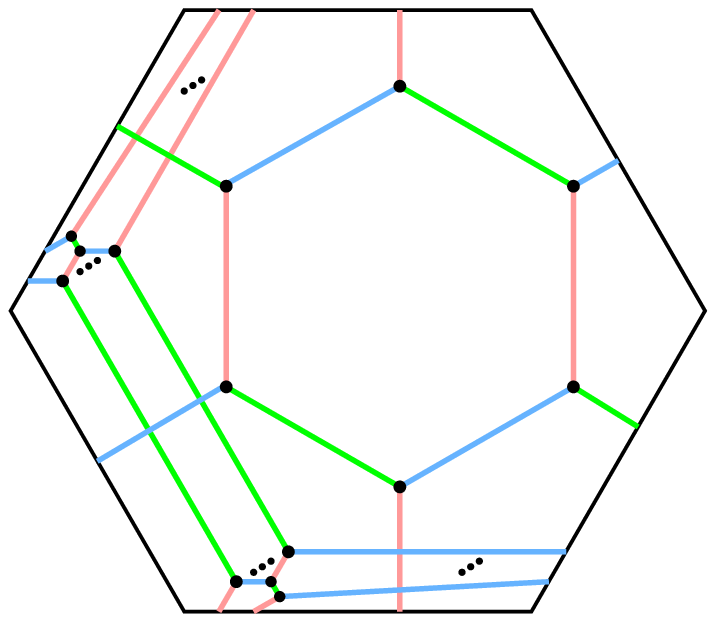}
  \label{fig:E2}
  \caption{$(E)_3 \cup (H)_{d-3}$}
\end{subfigure}
\begin{subfigure}{.32\textwidth}
  \centering
  \includegraphics[width=.9\linewidth]{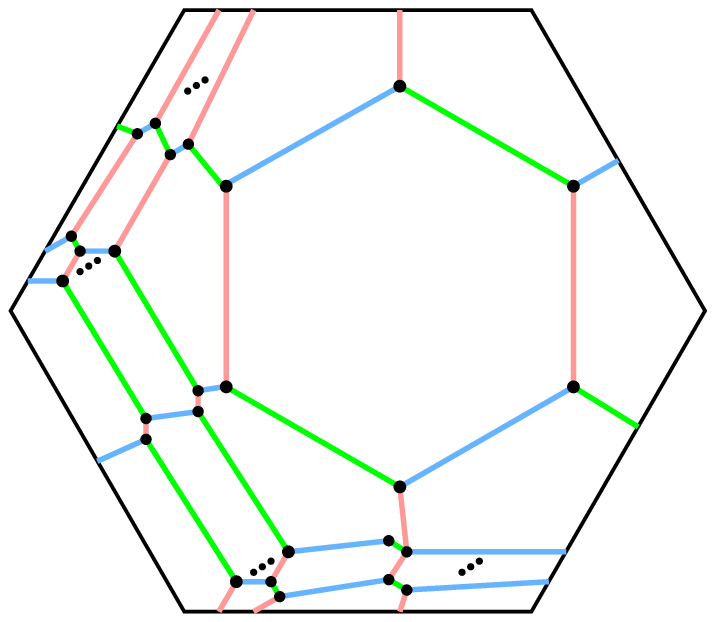}
  \label{fig:E3}
  \caption{Smoothing of $(E)_3 \cup (H)_{d-3}$}
\end{subfigure}
	\caption{}
	\label{fig:E}
\end{figure}

Using these lemmas, we have all of the ingredients we need to swiftly prove Theorem~\ref{thm:support}, that a positive genus genus-minimizing surface $\K \subset \CP^2$ with a hexagonal lattice diagram is isotopic to a complex curve $\C_d$.

\begin{proof}[Proof of Theorem~\ref{thm:support}]
Suppose that $\K$ has positive genus, minimizes genus in its homology class, and admits a hexagonal lattice diagram $(\aaa,\bbb,\ccc)$.  Then the degree $d$ of $\K$ satisfies $d \geq 3$, so by Theorem~\ref{thm:family}, the hexagonal lattice diagram must fall into one of the eight families $(A)_d$ through $(H)_d$.  By Lemma~\ref{lem:dh}, \ref{lem:bf}, \ref{lem:cg}, or~\ref{lem:ae}, it follows that $\K$ is isotopic to a complex curve $\C_d$.
\end{proof}

\section{Stein trisections of complex hypersurfaces in $\CP^3$}\label{sec:stein}

In this section, we prove that there are complex curves, determined as the zero sets of homogeneous polynomials in $\CP^2$, that intersect the standard trisection of $\CP^2$ as bridge trisected surfaces (without needing any isotopy).  Moreover, these bridge trisections happen to be isotopic to the bridge trisections determined by the hexagonal lattice diagrams $(A)_d$ and $(E)_d$.  In the process, we obtain new information about the trisections of complex hypersurfaces obtained as branched covers of $\CP^2$ over these complex curves.  Recall that $\CP^3$ is the set of equivalence classes $[z_1:z_2:z_3:z_4]$ of nonzero vectors in $(\Cc^4)^*$.  A \emph{degree $d$ hypersurface $\Ss_d$} is defined to be the 4-manifold in $\CP^3$ obtained by taking the zero set of any homogeneous degree $d$ polynomial in the variables $z_1,z_2,z_3,z_4$.  Up to diffeomorphism, these manifolds depend only of the degree $d$ of the polynomial.  Additionally, recall that a Stein trisection of a complex 4-manifold $X$ is a trisection $X = X_1 \cup X_2 \cup X_3$ such that each 4-dimensional handlebody $X_j$ is an analytic polyhedron (see~\cite{PLCstein} for further details and examples).  Lambert-Cole and Meier showed that the standard trisection of $\CP^2$ is Stein and asked the following:

\begin{question}\cite{LCM}\label{q:lcm}
Does every complex projective surface admit a Stein trisection?
\end{question}

They also proved that each $\Ss_d$ admits an \emph{efficient trisection}, a trisection $\Ss_d = X_1 \cup X_2 \cup X_3$ in which each $X_j$ is a 4-ball.  The work in this section yields a partial answer to Question~\ref{q:lcm}, obtained by lifting a bridge trisection of a complex curve in $\CP^2$.  For $d \geq 1$, recall that we defined that the varieties $\V_d$ and $\V_d'$ are defined as
\begin{eqnarray*}
\V_d &=& \{[z_1:z_2:z_3] \in \CP^2 : z_1z_2^{d-1} + z_2z_3^{d-1} + z_3z_1^{d-1} = 0\}, \\
\V'_d &=& \{[z_1:z_2:z_3] \in \CP^2 : z_1^{d-1}z_2 + z_2^{d-1}z_3 + z_3^{d-1}z_1= 0\},
\end{eqnarray*}
where both are smoothly isotopic to $\C_d$ (see Proposition 4.5 of~\cite{LCM}, for example).  We say that a variety $\V$ is in \emph{bridge position} with respect to the standard trisection of $\CP^2$ if for each $j$, we have that $\V \cap X_j$ is a collection of trivial disks and $\V \cap H_j$ is a collection of trivial arcs.  In~\cite{LCM}, the authors demonstrate that the complex curve $\V_1$ is in bridge position; however, as a variety, the curve $\C_d$ (defined to be the zero set of $z_1^d + z_2^d + z_3^d$) is \emph{not} in bridge position.  A bridge position is called \emph{efficient} if for each $j$, we have that $\V \cap X_j$ is a single disk.  The main result in this section is proof of Theorem~\ref{thm:main2}, which asserts that $\V_d$ and its counterpart $\V_d'$ are in bridge position, and shadow diagrams for these bridge trisections agree with $(A)_d$ and $(E)_d$, respectively.

We can use Theorem~\ref{thm:main2} to prove Corollary~\ref{cor:stein}, which involves the complex projective surfaces $\Ss_d$ and $\Ss_d'$ determined by

\begin{eqnarray*}
\Ss_d &=& \{[z_1:z_2:z_3:z_4] \in \CP^3 : z_1z_2^{d-1} + z_2z_3^{d-1} + z_3z_1^{d-1} + z_4^d= 0\}, \\
\Ss'_d &=& \{[z_1:z_2:z_3:z_4] \in \CP^3 : z_1^{d-1}z_2 + z_2^{d-1}z_3 + z_3^{d-1}z_1 + z_4^d= 0\}
\end{eqnarray*}

We are grateful to Peter Lambert-Cole for directing us to his work on Stein trisections, and for pointing the proof of the next corollary.

\begin{proof}[Proof of Corollary~\ref{cor:stein}]
The projection $\pi:\CP^3 \rightarrow \CP^2$ given by $\pi([z_1:z_2:z_3:z_4]) = [z_1:z_2:z_3]$ restricts to a degree $d$ branched covering map from $\Ss_d$ onto $\CP^2$, with branch locus equal to $\V_d$.  Since $\V_d$ is in bridge position with respect to the standard trisection $\CP^2 = X_1 \cup X_2 \cup X_3$, and since the map $\pi$ is analytic, the standard trisection lifts to a Stein trisection $\Ss_d = \pi^{-1}(X_1) \cup \pi^{-1}(X_2) \cup \pi^{-1}(X_3)$.  Moreover, since the bridge position of $\V_d$ is efficient, each $\pi^{-1}(X_j)$ is a 4-ball; hence $\Ss_d$ admits an efficient Stein trisection.  A similar proof holds using $\V'_d$ and $\Ss'_d$.
\end{proof}

\subsection{Technical lemmas}

In this subsection, we develop several technical lemmas to be employed in our proof of Theorem~\ref{thm:main2}.  We will consider $S^1$ to be the complex unit circle and $D^2$ to be the complex unit disk, so that $D^2 \X S^1$ is the subset of $\Cc^2$ given by $D^2 \X S^1 = \{(re^{i\nu},e^{i\omega}): 0 \leq r \leq 1\}$.  These lemmas will help us analyze the intersection of $\V_d$ with the standard trisection of $\CP^2$.

\begin{lemma}\label{arc1}
The set of points $\{(re^{i\nu},e^{i\omega}) : re^{i \nu} + re^{i\omega} = -1\} \subset D^2 \X S^1$ is a properly embedded arc that is isotopic, fixing its endpoints, to the straight-line arc $\{(e^{is},e^{i(2\pi-s)}) : s \in [2\pi/3,4\pi/3]\} \subset \pd(D^2 \X S^1)$.
\end{lemma}

\begin{proof}
First, note that for any solution $(re^{i\nu},e^{i\omega})$ to $re^{i\nu} + re^{i\omega} = -1$, we have that $r \neq 0$.  Equating real parts yields $r \cos \nu + r \cos \omega = -1$, and equating imaginary parts yields $r \sin \nu + r \sin \omega = 0$.  Since $0<r\leq 1$, both $\cos \nu$ and $\cos \omega$ must be negative, so that $\omega, \nu \in [\pi/2,3\pi/2]$.  In addition, $\sin \nu = -\sin \omega$ implies $\omega = 2\pi - \nu$, and thus $\cos \omega = \cos \nu$.  It follows that $r \cos \nu + r \cos \omega = 2 r \cos \nu = -1$; hence, $r \cos \nu = -\frac{1}{2}$.  This equation will have a solution for $r \in (0,1]$ if and only if $\omega \in [2\pi /3,4\pi/3]$, in which case $r = -\frac{1}{2 \cos \nu}$.  Thus, for $s \in [2\pi/3,4\pi/3]$, the solutions can be parametrized by $(-\frac{1}{2 \cos s} e^{is},e^{i(2\pi-s)})$.  Finally, for any $t \in [0,1]$, we see that $\{((t-\frac{1-t}{2 \cos s})e^{is},e^{i(2\pi-s)}): s\in [2\pi/3,4\pi/3]\}$ is a path in $D^2 \X S^1$ connecting the points $(e^{2\pi i/3},e^{4\pi i/3})$ and $(e^{4\pi i/3},e^{2\pi i/3})$, yielding an isotopy from the solution set (the case $t=0$) to the arc claimed in the statement of the lemma (the case $t=1$).
\end{proof}

\begin{lemma}\label{arc2}
For any $t \geq 1$, the set of points $\{(re^{i\nu},e^{i\omega}) : r^te^{i\nu} + re^{i\omega} = -1\} \subset D^2 \X S^1$ is a properly embedded arc that is isotopic, fixing its endpoints, to the arc $\{(re^{i \nu},e^{i\omega}) : re^{i\nu} + re^{i\omega} = -1\} \subset D^2 \X S^1$.
\end{lemma}

\begin{proof}
Suppose $t \geq 1$.  If $(re^{i\nu},e^{i\omega})$ satisfies $r^te^{i\nu} + re^{i\omega} = -1$, then as above, equating real parts yields $r^t \cos \nu + r \cos \omega = -1$, and equating imaginary parts yields $r^t \sin \nu + r \sin \omega = 0$.  Once again, since $0 < r \leq 1$, it follows that $\nu,\omega \in [\frac{\pi}{2},\frac{3\pi}{2}]$.  Suppose first that $\nu \in [\frac{\pi}{2},\pi]$.  Then $\omega \in [\pi,\frac{3\pi}{2}]$, and there there is a triangle with side lengths 1, $r$, and $r^t$ and angles $\pi - \nu$, $\omega - \pi$, and $\nu-\omega + \pi$, as shown at left in Figure~\ref{fig:triangle}.  Such a triangle exists if and only if $r+r^t \geq 1$, $1+r \geq r^t$ and $1 + r^t \geq r$.  The second and third inequalities are always true since $r \in (0,1]$.  For the first inequality, let $f:\R_{\geq 0} \rightarrow \R$ be given by $f(r) = r^t + r$.  Then $f(0) = 0$, $f(1) = 2$, and $\frac{df}{dr} = 1+tr^{t-1} > 0$ for all $r \in \R_{\geq 0}$.  By the Intermediate Value Theorem, there is a unique $r_0 \in (0,1)$ such that $r_0^t + r_0 = 1$, and thus a solution $(r^te^{i\nu},e^{i\omega})$ exists if and only if $r \in [r_0,1]$.  Moreover, by the Implicit Function Theorem, $r_0 = r_0(t)$ is continuous as a function of $t$.

\begin{figure}[h!]
\begin{subfigure}{.24\textwidth}
  \centering
  \includegraphics[width=.95\linewidth]{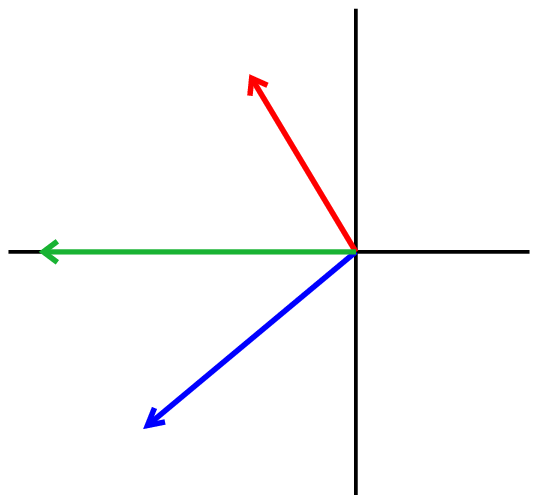}
    \put (-91,52) {\textcolor{ForestGreen}{$-1$}}
    \put (-61,82) {\textcolor{red}{$r^te^{i\nu}$}}
    \put (-61,14) {\textcolor{blue}{$re^{i\omega}$}}

  \label{fig:triA}
\end{subfigure}
\begin{subfigure}{.24\textwidth}
  \centering
  \includegraphics[width=.95\linewidth]{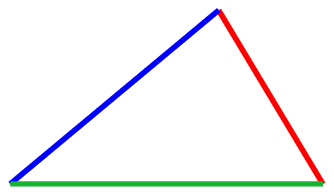}
      \put (-51,24) {\textcolor{ForestGreen}{$1$}}
    \put (-25,52) {\textcolor{red}{$r^t$}}
    \put (-65,56) {\textcolor{blue}{$r$}}
    \put (-44,38) {\tiny{$\pi-\nu$}}
    \put (-72,38) {\tiny{$\omega-\pi$}}
  \label{fig:triB}
\end{subfigure}
\begin{subfigure}{.24\textwidth}
  \centering
  \includegraphics[width=.95\linewidth]{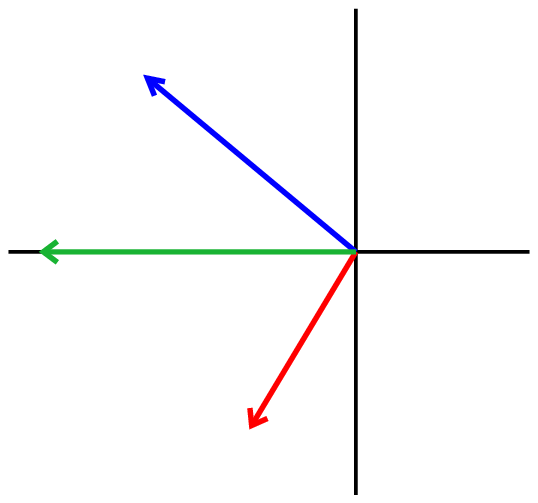}
      \put (-91,52) {\textcolor{ForestGreen}{$-1$}}
    \put (-73,17) {\textcolor{red}{$r^te^{i\nu}$}}
    \put (-59,70) {\textcolor{blue}{$re^{i\omega}$}}

    \label{fig:triC}
\end{subfigure}
\begin{subfigure}{.24\textwidth}
  \centering
  \includegraphics[width=.95\linewidth]{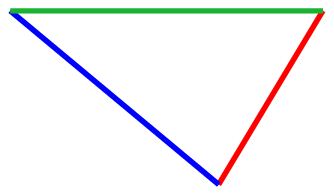}
        \put (-51,68) {\textcolor{ForestGreen}{$1$}}
    \put (-27,46) {\textcolor{red}{$r^t$}}
    \put (-68,45) {\textcolor{blue}{$r$}}
    \put (-44,61) {\tiny{$\nu-\pi$}}
    \put (-72,61) {\tiny{$\pi-\omega$}}
    \label{fig:triD}
\end{subfigure}
	\caption{At left, the case $\nu \in [\pi/2,2\pi]$ and corresponding triangle.  At right, the case $\nu \in [\pi,3\pi/2]$ and corresponding triangle}
	\label{fig:triangle}
\end{figure}

Now, applying the Law of Cosines to the triangle in Figure~\ref{fig:triangle} allows us to write $\nu$ and $\omega$ as continuous functions of $r$ and $t$, and the function $p^*_t:[r_0,1] \rightarrow D^2 \X S^1$ that maps $r$ to $(re^{\nu i},e^{\omega i})$ is injective.  Moreover, $\nu = \pi$ if and only if $\omega = \pi$, and these conditions are true if and only if $r + r^t = 1$ and $r = r_0$.  Thus, $p^*_t([r_0,1])$ is an embedded arc in $D^2 \X S^1$, and by inspection, the endpoints of this arc are $(-r_0,-1)$ and $(e^{2\pi i/3},e^{4\pi i/3})$.

On the other hand, suppose that $\nu \in [\pi,\frac{3\pi}{2}]$.  In this case, there is a triangle with side lengths 1, $r$, and $r^t$ and angles $\nu - \pi$, $\pi - \omega$, and $\omega - \nu + \pi$, as shown at right in Figure~\ref{fig:triangle}.   Applying the Law of Cosines again yields $\nu$ and $\omega$ as continuous functions of $r$ and $t$.  As above, the function $p^{**}_t:[r_0,1] \rightarrow D^2 \X S^1$ that maps $r$ to $(re^{i\nu},e^{i\omega})$ is injective, and $\nu = \pi$ if and only if $r = r_0$.  Thus, $p^{**}_t([r_0,1])$ is an embedded arc in $D^2 \X S^1$ with endpoints $(-r_0,-1)$ and $(e^{4\pi i/3},e^{2 \pi i/3})$.  Now, we paste these two arcs together to get an embedding $p_t:[0,1] \rightarrow D^2 \X S^1$, defined formally as
\[ p_t(s) = \begin{cases}
p^*_t((2r_0-2)s+1) & \text{if } s \leq \frac{1}{2} \\
p^{**}_t((2-2r_0)s+2r_0-1) & \text{if } s \geq \frac{1}{2}
\end{cases}.\]
We verify that $p_t(0) = p^*_t(1) = (e^{2\pi i/3},e^{4\pi i/3})$, $p_t(1) = p^{**}_t(1) = (e^{4\pi i/3},e^{2\pi i/3})$, and $p_t(\frac{1}{2}) = p^*_t(r_0) = p^{**}_t(r_0) = (-r_0,-1)$, so that $p_t([0,1])$ is an embedded arc.  Moreover, $P:[0,1] \X [0,1] \rightarrow D^2 \X S^1$ given by $P(s,t) = p_t(s)$ is continuous (since $r_0$ is continuous as a function of $t$), and thus the arc $p_t([0,1])$ is isotopic rel boundary to the arc $p_1([0,1])$; that is, the set of points $\{(re^{i\nu},e^{i\omega}) : r^te^{i\nu} + re^{i\omega} = -1\}$ as claimed.
\end{proof}

\subsection{Determining $\V_d \cap H_i$}\label{sub:arcs}

Now, we consider the intersection of $\V_d$ with the solid torus $H_1 = \{[z_1:z_2:z_3]: |z_1| \leq |z_2| = |z_3|\}$.  Using the affine chart that maps $z_3$ to 1, we see that
\[ \V_d \cap H_1 = \{(z_1,z_2) \in D^2 \X S^1 : z_1z_2^{d-1} + z_2 + z_1^{d-1} = 0\}.\]
We understand this intersection in the next lemma.

\begin{proposition}\label{shadow}
$\V_d \cap H_1$ is a collection of boundary parallel arcs, which are isotopic via an isotopy fixing their endpoints to straight line arcs in $\Sigma$ from $ [e^{i\theta^-_j}:e^{i\psi^-_j}:1]$ to $[e^{i\theta^+_j},e^{i\psi^+_j}:1]$, where
\setcounter{tblEqCounter}{\theequation}
\begin{center}
\begin{tabular}{cccccccccc}
$\theta^-_j$ &$=$& $\dfrac{\frac{4\pi}{3}d + 2\pi(j-1)}{d^2-3d+3}$ & & & & $\theta^+_j$ &$=$& $\dfrac{\frac{2\pi}{3}d + 2\pi j}{d^2-3d+3}$ & \numberTblEq{eq1} \\
$\psi^-_j$&$=$  & $\dfrac{\frac{2\pi}{3}d + 2\pi(d-1)j - 2\pi}{d^2-3d+3}$ & & & & $\psi^+_j$&$=$  & $\dfrac{\frac{4\pi}{3}d + 2\pi(d-1)j - 2\pi}{d^2-3d+3}$ &  \numberTblEq{eq2} \\
\end{tabular}
\end{center}
for $1 \leq j \leq d^2-3d+3$.
\end{proposition}

\begin{proof}
Let $z_1 = re^{i\theta}$, and $z_2 = e^{i\psi}$.  Then $\V_d \cap H_1 = \{(re^{i\theta},e^{i\psi}) : re^{i(\theta + (d-1)\psi)} + e^{i\psi} + r^{d-1}e^{i(d-1)\theta} = 0\}$.  Dividing the equation by $e^{i\psi}$ and setting $\nu = (d-1)\theta - \psi$ and $\omega = \theta + (d-2)\psi$, we can rewrite
\[ \V_d \cap H_1 = \{(re^{i\theta},se^{i\psi}) : r^{d-1}e^{i\nu} + re^{i\omega}  = -1\}.\]
Observe that $\psi = (d-1)\theta - \nu$, and by substitution $\omega = \theta + (d-2)((d-1)\theta - \nu)$, or equivalently, $(d^2-3d+3)\theta = (d-2)\nu + \omega$.  For each $j$ with $1 \leq j \leq d^2-3d+3$ define $\theta_j$ and $\psi_j$ as
\begin{eqnarray*}
\theta_j &=& \frac{(d-2)\nu + \omega + 2\pi j}{d^2-3d+3} \\
\psi_j = (d-1)\theta_j - \nu &=& \frac{(d-1)\omega - \nu + 2\pi(d-1) j}{d^2-3d+3}
\end{eqnarray*}
By Lemmas~\ref{arc1} and \ref{arc2}, we have that for each $j$, there is an arc of solutions $\{(re^{i\theta_j},e^{i\psi_j}):r^{d-1}e^{i(d-1)\theta_j} + re^{i(\theta_j + (d-1)\psi_j)} + e^{i\psi_j} = 0\}$, which is isotopic rel endpoints to the straight line arc from the point satisfying $(\nu,\omega) = (\frac{2\pi}{3},\frac{4\pi}{3})$ to the point satisfying $(\nu,\omega) = (\frac{4\pi}{3},\frac{2\pi}{3})$.  

We wish to show that these isotopies can be carried out simultaneously.  To this end, for each $j$ with $1 \leq j \leq d^2-3d+3$, let $I_j$ be a copy of a closed interval and let $p_t^j:I_j \rightarrow D^2 \X S^1$ be given by $p_t^j(s) = (re^{i\theta_j},e^{i\psi_j})$, where $r$, $\nu$, and $\omega$ are determined by the function $p_t$ given in Lemma~\ref{arc2}.  We claim that these embeddings never intersect.  In other words, if $p_t^j(s_1) = p_t^k(s_2)$, then $j = k$ and $s_1 = s_2$, implying that $P_t: \bigsqcup I_j \rightarrow D^2 \X S^1$ is continuous family of embeddings.  If $p_t^j(s_1) = p_t^k(s_2)$, then $\theta_j(s_1) =\theta_k(s_2)$ and $\psi_j(s_1) = \psi_k(s_2)$, but then $\nu(s_1) = \nu(s_2)$ and $\omega(s_1) = \omega(s_2)$.  Using the definition of $\theta_j$, it follows that $j = k$, and since $p_t^j$ is injective, we have $s_1 = s_2$.  

The function $P_t$ yields an isotopy of the arcs $p_0^j(I_j)$ onto the corresponding arcs $\{(re^{i \nu},e^{i\omega}) : re^{i\nu} + re^{i\omega} = -1\}$, and the isotopy given by Lemma~\ref{arc1} is the resulting of projecting onto the torus with $(\nu,\omega)$ coordinates.  Using the definitions of $\theta_j$ and $\psi_j$ above, we can instead project onto the torus $\Sigma$ with $(\theta,\psi)$ coordinates, in which case the $d^2 -3d+ 3$ segments are parallel, and so they project to disjoint (parallel) segments in $\Sigma$ via Lemma~\ref{arc1}.  The endpoints of $p_0^j(I_j)$ are calculated by plugging the endpoints $(\nu^-,\omega^-) = (e^{4\pi i/3},e^{2\pi i/3})$ and $(\nu^+,\omega^+) = (e^{2\pi i/3},e^{4\pi i /3})$ into the formulas for $\theta_j$ and $\psi_j$ above.
\end{proof}

From this point forward, we let $\aaa$ denote the straight-line arcs in $\Sigma$ produced by Lemma~\ref{shadow}.  By a symmetric argument, we have that $\V_d \cap H_2$ is isotopic rel boundary to a collection $\bbb$ of straight-line arcs in $\Sigma$, and $\V_d \cap H_3$ is isotopic rel boundary to a collection $\ccc$ of straight-line arcs in $\Sigma$.  We let $\mathbf{x}$ be the collection of all of the endpoints of these arcs, $\mathbf{x} =  \{[\theta_j^{\pm}:\psi_j^{\pm}:1] \in \CP^2\}$, and we let $\mathbf{x}^+ = \{\theta_j^+,\psi_j^+\}$ and $\mathbf{x}^- = \{\theta_j^-,\psi_j^-\}$.  Recall that $\ab$, $\bc$, and $\ca$ denote the curves $\aaa \cup \overline{\bbb}$, $\bbb \cup \overline{\ccc}$, and $\ccc \cup \overline{\aaa}$, respectively, although a priori, in this setting these are potentially immersed curves, since we have not yet established that the arcs in $\aaa$, $\bbb$, and $\ccc$ avoid each other in their interiors.  In addition, recall that $\A = \{[e^{i\theta}:1:1]\}$, $\n = \{[1:e^{i\psi}:1]\}$, and $\gamma = \{[e^{-i\theta}:e^{-i\theta}:0]\}$.

We also observe that by the formulas in Proposition~\ref{shadow}, any translation that sends one point in $\mathbf{x}^+$ to another preserves the set $\mathbf{x}^+$.  Moreover, any pair of arcs in $\aaa$, $\bbb$, or $\ccc$ are parallel, and thus any translation of $\Sigma$ sending one point in $\mathbf{x}^+$ to another preserves all three sets of arcs $a_{ij}$.  Additionally, the natural map on $\CP^2$ that sends $[z_0:z_1:z_2]$ to $[z_1:z_2:z_0]$ induces a $2\pi/3$ rotation of the hexagonal picture of $\Sigma$ and cyclic permutation of $\aaa$, $\bbb$, and $\ccc$.

\begin{lemma}\label{homol}
In $H_1(\Sigma)$, we have $\ab = [\A] + (d-1)[\n]$.
\end{lemma}

\begin{proof}
Using Equations~\ref{eq1} and~\ref{eq2} from Proposition~\ref{shadow}, we define
\begin{equation}
\Delta \theta_j := \theta_j^+ - \theta_j^- = \frac{2\pi - \frac{2\pi}{3}d}{d^2-3d+3} \text{ and } \Delta \psi_j := \psi_j^+ - \psi_j^- = \frac{\frac{2\pi}{3}d}{d^2-3d+3}.
\end{equation}
In addition, by Proposition~\ref{shadow} and the symmetry described above, the arcs $\overline{\bbb}$ are straight-line arcs from $[1:e^{i\theta_j^+}:e^{i\psi_j^+}]$ to $[1:e^{i\theta_j^-}:e^{i\psi_j^-}]$, or equivalently, from $[e^{i(-\psi_j^+)}:e^{i(\theta_j^+ - \psi_j^+)}:1]$ to $[e^{i(-\psi_j^-)}:e^{i(\theta_j^- - \psi_j^-)}:1]$.  Therefore, the total change in the argument of $z_1$ in the curve $\ab$ is
\[ \int_{\ab} d \theta = \sum_j \left(\Delta \theta_j -(-\Delta \psi_j) \right)= \frac{(d^2-3d+3)(2\pi - \frac{2\pi}{3}d + \frac{2\pi}{3}d)}{d^2-3d+3} = 2\pi.\]
Similarly, the total change in the argument of $z_2$ in the curve $\ab$ is
\[ \int_{\ab} d\psi = \sum_j\left(\Delta \psi_j - (\Delta \theta_j - \Delta \psi_j)\right) = \frac{(d^2-3d+3)(\frac{2\pi}{3}d - (2\pi - \frac{2\pi}{3}d) + \frac{2\pi}{3}d)}{d^2-3d+3} = 2\pi(d-1).\]
The desired statement follows immediately.
\end{proof}

The previous lemma and the following lemma will show that $\ab$ is, in fact, an embedded $(1,d-1)$-curve in the torus $\Sigma$, viewed as a Heegaard surface for $\pd X_1 = H_1 \cup H_2$.  In the proof below, we consider closed intervals of the form $[\theta_1,\theta_2]$, where $\theta_1,\theta_2 \in \R/2\pi \Z$.  To avoid ambiguity from the fact that $\theta_2  = \theta_2 + 2\pi$ in $\R/2\pi \Z$, we let $[\theta_1,\theta_2] \subset \R/2\pi \Z$ be the smallest possible such interval in $\R/2\pi \Z$.  

\begin{lemma}\label{disjoint}
The arcs $\aaa$ and $\bbb$ do not intersect in their interiors.
\end{lemma}

\begin{proof}
By the translational symmetry of the arcs $\aaa$ and $\bbb$, it suffices to show that the single arc $a_1 \in \aaa$ corresponding to the value $j=1$ from Proposition~\ref{shadow} has interior disjoint from the arcs in $\bbb$.  Reindex the points in $\mathbf{x}^+$ and $\mathbf{x}^-$ so that they occur in order along the oriented curve $\ab$, where $a_1$ connects $x_1^-$ to $x_1^+$ and each arc in $\overline{b}_k \in \overline{\bbb}$ connects $x_{k-1}^+$ to $x_{k}^-$ (taken modulo $d^2-3d+3$).

Note that for each $k$, we have $x_k^{\pm} = [e^{i\theta_j^{\pm}}:e^{i\psi_j^{\pm}}:1]$ for some $j$ (with possibly different values for $j$ in the $\pm$ cases).  To simplify notation, for $j$ and $k$ satisfying this equation we let $\theta(x_k^{\pm}) = \theta_j^{\pm}$ and $\psi(x_k^{\pm}) = \psi_j^{\pm}$. In the proof of Lemma~\ref{homol}, we showed that the arcs $\overline{\bbb}$ are straight-line arcs from $[e^{i(-\psi_j^+)}:e^{i(\theta_j^+ - \psi_j^+)}:1]$ to $[e^{i(-\psi_j^-)}:e^{i(\theta_j^- - \psi_j^-)}:1]$, and thus we have
\begin{center}
\begin{tabular}{ccccccccc}
$\theta(x_k^{-}) - \theta(x_{k-1}^{+})$ & $=$ & $\dfrac{\frac{2\pi}{3}d}{d^2-3d+3}$ & & & & $\psi(x_k^{-}) - \psi(x_{k-1}^{+})$ & $=$ & $\dfrac{\frac{4\pi}{3}d-2\pi}{d^2-3d+3}$\vspace{.2cm}\\
$\theta(x_k^+) - \theta(x_k^-)$ & $=$ & $\Delta \theta_j = \dfrac{2\pi -\frac{2\pi}{3}d}{d^2-3d+3}$ & & & & $\psi(x_k^+) - \psi(x_{k-1}^-)$ & $=$ & $\Delta \psi_j = \dfrac{\frac{2\pi}{3}d}{d^2-3d+3}$
\end{tabular}
\end{center}
Suppose by way of contradiction that some arc $\overline{b}_k \in \overline{\bbb}$ intersects the interior of $a_1$.  Observe that $\ab$ is monotonically increasing in the $\psi$ direction, and thus if $k=1$, we have that $a_1 \cup b_1$ advances more than $2\pi$ in the direction of $\psi$, so that
\[ 2\pi < \int_{a_1 \cup \overline{b}_1} \, d\psi = \psi(x_1^+) - \psi(x_0^+) = \frac{2\pi(d-1)}{d^2-3d+3},\]
a contradiction.  Similarly, if $k=2$, then once again
\[ 2\pi < \int_{a_1 \cup \overline{b}_2} \, d\psi = \psi(x_2^-) - \psi(x_1^-) = \frac{2\pi(d-1)}{d^2-3d+3}.\]
Henceforth, we assume $k\neq 1,2$, and so the lemma holds in the cases $d = 1$ or $d=2$, in which each of $\aaa$ and $\bbb$ contains a single arc.  Thus, assume that $d \geq 3$, so that $\theta(x_k^+) - \theta(x_k^-) = \Delta \theta_j \leq 0$.  Note that the set of coordinates $\{\theta(x_k^+)\}$ are spaced evenly and are in increasing order, with consecutive values $\frac{2\pi}{d^3-3d+3}$ apart.  It follows that $\theta(x_k^+) \notin [\theta(x_0^+),\theta(x_1^+)]$.  A similar argument shows that $\theta(x_k^-) \notin [\theta(x_1^-),\theta(x_2^-)]$.  Since $a_1$ and $\overline{b}_k$ intersect, the intervals $[\theta(x_1^+),\theta(x_1^-)]$ and $[\theta(x_{k-1}^+),\theta(x_k^-)]$ must overlap.  In addition, $[\theta(x_{k-1}^+),\theta(x_k^-)]$ cannot contain both intervals $[\theta(x_0^+),\theta(x_1^+)]$ and $[\theta(x_1^-),\theta(x_2^-)]$; otherwise, $[\theta(x_{k-1}^+),\theta(x_k^-)]$ would contain the entire interval $[\theta(x_0^+),\theta(x_2^-)]$.  In this case, we compute
\begin{eqnarray*}
\dfrac{\frac{2\pi}{3}d}{d^2-3d+3} &=& \theta(x_k^{-}) - \theta(x_{k-1}^{+}) \\
&>& \theta(x_2^-) - \theta(x_0^+) \\
&=& (\theta(x_2^-) - \theta(x_1^-)) + (\theta(x_1^-) - \theta(x_0^+)) \\
&=& \dfrac{2\pi}{d^2-3d+3} +\dfrac{\frac{2\pi}{3}d}{d^2-3d+3},
\end{eqnarray*}
which is not possible.  It follows that either $\theta(x_{k-1}^+) \in [\theta(x_1^+),\theta(x_1^-)]$, or $\theta(x_k^-) \in [\theta(x_1^+),\theta(x_1^-)]$.

In the first case, let $\delta_*$ denote the union of arcs in $\ab$ from $x_1^-$ to $x_k^-$,
\[\delta_* = a_1 \cup \overline{b}_2 \cup a_2 \cup \dots \cup \overline{b}_k.\]
We will show that $\delta_*$ advances ``too much" in the $\psi$ direction and ``too little" in the $\theta$ direction.  Since $a_1$ and $\overline{b}_k$ intersect, and each arc of $\aaa$ and $\overline{\bbb}$ travels in the positive direction with respect to $\psi$, we have $\int_{\delta_*} d\psi > 2\pi$.  Additionally, since $\delta_*$ is not the entirety of $\ab$, we have $0 < \int_{\delta_*} d\theta < 2\pi$.  Using $\theta(x_{k-1}^+) \in [\theta(x_1^+),\theta(x_1^-)]$ and choosing a representative for $\theta(x^-_k)$ such that $\theta(x^-_1) < \theta(x^-_k) < \theta(x^-_1) + 2\pi$, we have
\[ \int_{\delta_*} d\theta = \theta(x^-_{k}) - \theta(x_1^-) < \theta(x^-_{k}) - \theta(x^+_{k-1}) = \frac{\frac{2\pi}{3}d}{d^2-3d+3}.\]
Let $(d-1)\delta_*$ denote the union of arcs in $\ab$ obtained by translating $\delta_*$ end-to-end a total of $d-1$ times.  Then we have
\[ \int_{(d-1)\delta_*} d \psi = (d-1) \int_{\delta_*} d\psi > 2\pi(d-1).\]
Since $\int_{K_1} d \psi = 2\pi(d-1)$, this implies $(d-1){\delta}_*$ traverses each arc of $\aaa$ and $\overline{\bbb}$ (and then some).  On the other hand,
\[ \int_{(d-1)\delta_*} d \theta = (d-1) \int_{\delta_*} d\theta < \frac{\frac{2\pi}{3}d(d-1)}{d^2-3d+3} < 2\pi = \int_{\ab} d\theta.\]
For the second inequality, we note that the polynomial $2d^2-8d+9 = 3(d^2-3d+3) - d(d-1)$ is always positive, so that $\frac{d(d-1)}{3(d^2-3d+3)} < 1$.  The second integral calculation implies that $(d-1){\delta}_*$ does not traverse all arcs of $\aaa$ and $\overline{\bbb}$, contradicting the first integral calculation.  In the case that $\theta(x_1^+) < \theta(x_{k_0}^-) < \theta(x_1^-)$, we let $\delta_*$ traverse $\ab$ from $x_0^+$ to $x_k^+$, make similar calculations, and arrive at a similar contradiction.  We conclude that the arcs of $\aaa$ and $\bbb$ intersect only at their endpoints.
\end{proof}

\begin{corollary}\label{cor:hexclassify}
The triple of arcs $(\aaa,\bbb,\ccc)$ agrees with the hexagonal lattice diagram $(A)_d$.
\end{corollary}

\begin{proof}
By Lemma~\ref{disjoint}, none of the arcs $(\aaa,\bbb,\ccc)$ intersect in their interiors.  Since arcs in each collection are mutually parallel and since each of the bridge points $\mathbf{x}$ is adjacent to a single arc from each of the three collections, the triple $(\aaa,\bbb,\ccc)$ forms a hexagonal lattice diagram.  By Lemma~\ref{homol}, we have $\ab = \A + (d-1)\n$, and the other two pairings match those of $(A)_d$ by the symmetry of the construction.
\end{proof}

\subsection{Determining $\V_d \cap X_i$}\label{sub:disks}

In Subsection~\ref{sub:arcs}, we characterized the arcs of intersection $\V_d \cap H_i$.  What now remains is to show that $\V_d \cap X_i$ is a single disk for each $i$.  First, we show that $\V_d \cap X_i$ is a smooth surface.

\begin{lemma}\label{lem:smooth}
The intersection $\V_d \cap X_i$ is a smooth surface.
\end{lemma}

\begin{proof}
Using the affine chart that sends $z_3$ to 1, we have
\[ \V_d \cap X_1 = \{(z_1,z_2) \in D^2 \X D^2 : z_1z_2^{d-1} + z_2 + z_1^{d-1} = 0\}.\]
Consider the map $f: \mathbb C^2 \rightarrow \mathbb C$ given by $f(z_1,z_2) = z_1z_2^{d-1} + z_2 + z_1^{d-1}$, where the derivative of $f$ is given by the matrix
\[ df_{(z_1,z_2)} = \begin{bmatrix}
z_2^{d-1} + (d-1)z_1^{d-2} & (d-1)z_1z_2^{d-2} + 1 \end{bmatrix}\]
We claim that $0$ is a regular value of $f$.  Suppose by way of contradiction that there exists $(z_1,z_2)$ such that $f(z_1,z_2) = 0$ and $df_{(z_1,z_2)}$ is the zero map.  Then $z_2^{d-1} + (d-1)z_1^{d-2} = 0$ implies $z_2^{d-1} = (1-d)z_1^{d-2}$.  In addition, $(d-1)z_1z_2^{d-2} + 1 = 0$ implies that $(d-1)z_1z_2^{d-1} + z_2 = 0$, so that
\[ z_2 = (1-d)z_1z_2^{d-1} = (1-d)^2z_1^{d-1}.\]
Putting these equations together and using $f(z_1,z_2) = 0$, we get
\[ 0 = z_1z_2^{d-1} + z_2 + z_1^{d-1} = (1-d)z_1^{d-1} + (1-d)^2z_1^{d-1} + z_1^{d-1} = (3-3d+d^2)z_1^{d-1}.\]
Since $3-3d+d^2 \neq 0$, we have $z_1 = 0$, which implies that $z_2 = 0$ as well.  But in this case, $(d-1)z_1z_2^{d-2} + 1 \neq 0$, a contradiction.  We conclude that 0 is a regular value of $f$, and by the complex version of the Preimage Theorem, $f^{-1}(0)$ is a complex curve in $\mathbb C^2$.  It follows that $f^{-1}(0) \cap (D^2 \X D^2)$ is a smooth surface.
\end{proof}

Next, for each $d \geq 2$, define
\[ R_d = \{(r,s) \in I \X I : r^{d-1} - rs^{d-1} \leq s \leq r^{d-1} + rs^{d-1}\}.\]
Eventually, we will show that $\V_d \cap X_1$ can be constructed by embedding $d^2-3d+2$ copies of $R_d$ in $X_1$, where these copies meet along arcs in their boundaries.  First, we prove that $R_d$ is a disk.

\begin{lemma}\label{disko}
$R_d$ is a topological disk.
\end{lemma}

\begin{proof}
Define functions $F,G:I \X I \rightarrow \R$ by $F(r,s) = r^{d-1} - rs^{d-1} - s$ and $G(r,s) = r^{d-1} + rs^{d-1} - s$, so that $R_d = \{(r,s) \in I \X I: F(r,s) \leq 0 \leq G(r,s)\}$.  Then $\frac{\pd F}{\pd s} = -(d-1)rs^{d-2} - 1 < 0$, and so by the Preimage Theorem, the zero set of $F(r,s)$ is locally a smooth curve.  Similarly, $\frac{\pd G}{\pd r} = (d-1)r^{d-2} + s^{d-1} > 0$ if $(r,s) \neq (0,0)$, so that the zero set of $G(r,s)$ is locally a smooth curve away from $(0,0)$, and $\frac{\pd G}{\pd s}(0,0) = -1$, so $G(r,s) = 0$ is also locally a smooth curve at $(0,0)$.

Note that if $(r,s) \in R_d$, then $r = 0$ if and only if $s = 0$.  We show that for any $s_0 \in (0,1]$, the intersection of the line $s = s_0$ with $R_d$ is a line segment, from which the lemma will follow.  To this end, fix $s_0 \in (0,1]$, and consider the functions $f(r) = r^{d-1} + rs_0^{d-1}$ and $g(r) = r^{d-1} - rs_0^{d-1}$, so that the intersection in question is the set
\[ R_d(s_0) = \{(r,s_0) : g(r) \leq s_0 \leq f(r)\}.\]
Since $f(0) = 0$, $f(1) = 1+s_0^{d-1} > 1 \geq s_0$, and $f'(r) > 0$ for all $r \in (0,1]$, it follows from the Intermediate Value Theorem that there exists a unique $r_0 \in (0,1)$ such that $f(r_0) = s_0$, and we have that $s_0 \leq f(r)$ if and only if $r_0 \leq r$.  The function $g(r)$ is only slightly more complicated, in that $g$ has two real zeros at $r = 0$ and $r = s_0^{(d-1)/(d-2)}$.  By inspection, $g(r) < 0$ if $0 < r < s_0^{(d-1)/(d-2)}$ and $g'(r) > 0$ for all $r \geq s_0^{(d-1)/(d-2)}$.  Thus, if $g(1) \leq s_0$, then $g(r) \leq s_0$ for all $r \in [0,1]$, and we let $r_0' = 1$.  Otherwise, $g(1) > s_0$, and by the Intermediate Value Theorem, there exists a unique $r_0' \in (0,1)$ such that $g(r_0') = s_0$.  In either case, for $r \in [0,1]$ we have that $g(r) \leq s_0$ if and only if $r \leq r_0'$.  Finally, note that if $r_0' = 1$, then $r_0 < r_0'$, and alternatively, if $g(r_0') = s_0$, then $g(r_0) = r_0^{d-1} - r_0s_0^{d-1} < r_0^{d-1} + r_0s_0^{d-1} = s_0 = g(r_0')$, so that $r_0 < r_0'$.  We conclude that $R_d(s_0) = [r_0,r_0'] \X \{s_0\}$, and thus $R_d$ is a disk, as desired.
\end{proof}

As noted above, we consider the image of $\V_d \cap X_1$ under the standard affine chart which sets $z_3 = 1$, so that
\[ \V_d \cap X_1 = \{(z_1,z_2) \in D^2 \X D^2: z_1z_2^{d-1} + z_2 + z_1^{d-1} = 0\}.\]
Writing $z_1 = re^{i\theta}$ and $z_2 = se^{i\psi}$, we have a natural projection map $\pi:X_1 \rightarrow I \X I$ given by $\pi(z_1,z_2) = (r,s)$.  We note that $\pi(\pd X_1) = (I \X \{1\}) \cup (\{1\} \X I)$ and $(0,0) \in X_1$ is the unique point mapping to $(0,0) \in I \X I$.  We say that a (PL) singular foliation $f:I \X I \rightarrow I$ is \emph{radial} if $f((I \X \{1\}) \cup (\{1\} \X I)) = 1$ and $f$ has exactly one critical point $p = (0,0)$, where $f(p) = 0$.  It follows that for all $t$ with $0 < t < 1$, $f^{-1}(t)$ is an arc connecting $\{0\} \X I$ to $I \X \{0\}$.

\begin{lemma}\label{induce}
For any radial singular foliation $f: I \X I \rightarrow I$, the map $f \circ \pi:X_1 \rightarrow I$ is a singular foliation with level sets embedded copies of $S^3$ and one critical point at $(0,0)$.
\end{lemma}

\begin{proof}
Note that for any $r_0,s_0 \in (0,1)$, we have that $\pi^{-1}(r_0,0) = \{(r_0e^{i\theta},0) \in X_1\}$, an embedded copy of $S^1$, $\pi^{-1}(0,s_0) = \{(0,s_0e^{i\psi}) \in X_1\}$, another embedded copy of $S^1$, and $\pi^{-1}(r_0,s_0) = \{ r_0e^{i\theta},s_0e^{i \psi}\}$ an embedded torus.  Thus, for any $t_0 \in (0,1]$, we can see that $f \circ \pi:\pi^{-1}(f^{-1}(t_0)) \rightarrow I$ is the standard sweepout of $S^3$, since the arc $f^{-1}(t_0)$ is isotopic in $I \X I$ through arcs $f^{-1}(t)$ to the arc $f^{-1}(1)$.  It follows that these embedded copies of $S^3$ determine a singular foliation of $X_1 - \{(0,0)\}$ with a single critical point at $(0,0)$.
\end{proof}

\begin{proposition}\label{prop:cone}
The intersection $\V_d \cap X_i$ isotopic to the cone on the knot $\V_d \cap \pd X_i$.
\end{proposition}

\begin{proof}
We prove the statement for $i = 1$; the other cases follow symmetrically.  As above, we use the affine chart in which $z_3 = 1$, and we let $z_1 = re^{i\theta}$ and $z_2 = se^{i\psi}$.  Then $\V_d \cap X_1 = \{(re^{i\theta},se^{i\psi}) : rs^{d-1}e^{i(\theta + (d-1)\psi)} + se^{i\psi} + r^{d-1}e^{i(d-1)\theta}= 0\}$.  Dividing the equation by $e^{i\psi}$ and setting $\nu = (d-1)\theta - \psi$ and $\omega = \theta + (d-2)\psi$, we can rewrite
\[ \V_d \cap X_1 = \{(re^{i\theta},se^{i\psi}) : rs^{d-1}e^{i\omega} + s + r^{d-1}e^{i\nu}= 0\}.\]
It follows that the three complex numbers $rs^{d-1}e^{i\omega}$, $s$, and $r^{d-1}e^{i\nu}$ form a triangle.  If $\nu \in [0,\pi]$, then $\omega \in [\pi,2\pi]$, and the angles of the triangle are $\pi - \nu$ and $\omega - \pi$.  Note that such a triangle exists if and only if the side lengths satisfy $s \leq r^{d-1} + rs^{d-1}$, $r^{d-1} \leq s + rs^{d-1}$, and $rs^{d-1} \leq s + r^{d-1}$, where the third inequality always holds since $r,s\in [0,1]$, so that $rs^{d-1} \leq s$.  Thus such a triangle exists if and only if $(r,s) \in R_d$, where $R_d$ is the topological disk defined above.  Using the Law of Cosines, each $(r,s) \in R_d$ determines a value $\nu = \nu^+(r,s)$ and $\omega = \omega^+(r,s)$, where $\nu^+$ and $\omega^+$ are continuous functions from $R_d$ into $[0,\pi]$ and $[\pi,2\pi]$, respectively.  Note that $s = r^{d-1} + rs^{d-1}$ if and only if $\nu^+ = \pi$ and $\omega^+ = \pi$, and $r^{d-1} = s + rs^{d-1}$ if and only if $\nu^+ = \pi$ and $\omega^+ = 2\pi$.

On the other hand, if $\nu \in [\pi,2\pi]$, then $\omega \in [0,\pi]$, and the angles in the associated triangle are instead $\nu - \pi$ and $\pi - \omega$.  Similar to the work above, there are continuous functions $\nu^-:R_d \rightarrow [\pi,2\pi]$ and $\omega^-:\rightarrow [0,\pi]$ taking $(r,s)$ to the associated angles $\nu$ and $\omega$ using the Law of Cosines.  Note that $s = r^{d-1} + rs^{d-1}$ if and only if $\nu^- = \pi$ and $\omega^- = \pi$, and $r^{d-1} + rs^{d-1}$ if and only if $\nu^- = \pi$ and $\omega^- = 0$.  Recall that $\psi = (d-1)\theta - \nu^{\pm}$, and then by substitution, $\omega^{\pm} = \theta + (d-2)((d-1)\theta - \nu^{\pm})$, or equivalently,
\[ (d^2-3d+3)\theta = (d-2)\nu^{\pm} - \omega^{\pm}.\]
For each $j$ with $1 \leq j \leq d^2-3d+3$ define $\theta_j^{\pm}$ and $\psi_j^{\pm}$ as the functions on $R_d$ given by
\begin{eqnarray*}
\theta^{\pm}_j &=& \frac{(d-2)\nu^{\pm} - \omega^{\pm}}{d^2-3d+3} + \frac{2\pi j}{d^2 - 3d + 3}\\
\psi_j^{\pm} &=& (d-1)\theta_j^{\pm} - \nu^{\pm}.
\end{eqnarray*}
For each $j$, we let $(R_d)_j^+$ and $(R_d)_j^-$ be disjoint copies of $R_d$ (to be pasted together to form a disk later in the proof).  For each $(R_d)_j^{\pm}$, let $\delta_j^{\pm}$ be the boundary arc $\delta_j^{\pm} = \{(r,s) \in I \X I: s = r^{d-1} + rs^{d-1}$, and let $\eps_j^{\pm}$ be the boundary arc $\eps_j^{\pm} = \{(r,s) \in I \X I : r^{d-1} = s + rs^{d-1}\}$ (these are the boundary arcs along which we will paste copies of $R_d$).

By the work above, each function $\iota_j^{\pm}: (R_d)_j^{\pm} \rightarrow \V_d \cap X_1$ by $\iota_j^{\pm}(r,s) = (re^{i\theta^{\pm}_j},se^{i\psi^{\pm}_j})$ is an embedding of $R_d$ into $\V_d \cap X_1$, and in addition, every point in $\V_d \cap X_1$ is contained in the image of at least one of these embeddings.  We wish to show that these $2(d^2-3d+3)$ embeddings overlap only on their boundaries and can be pasted together to give an embedding of a disk onto $\V_d \cap X_1$.  It follows from remarks above that for any $(r,s) \in \delta_j^{\pm}$, $\nu_j^{\pm} = \pi$ and $\omega_j^{\pm} = \pi$, while for $(r,s) \in \eps_j^{\pm}$, $\nu_j^{\pm} = \pi$, $\omega^+_j = 2\pi$, and $\omega^-_j = 0$.  Thus, whenever $s = r^{d-1} + rs^{d-1}$, we have $\theta^+_j = \theta^-_j$ and $\omega^+_j = \omega^-_j$.  Similarly, whenever $r^{d-1} = s + rs^{d-1}$, we have $\theta^+_j = \theta^-_{j-1}$ and thus $\omega^+_j = \omega^-_{j-1}$.  This implies that we can paste $(R_d)^+_1$ to $(R_d)^-_1$ along $\delta_1^{\pm}$, paste $(R_d)_2^-$ to $(R_d)_1^+$ along $\eps_2^-$ and $\eps_1^+$, and so on, finally pasting $(R_d)_{d^2-3d+3}^+$ to $(R_d)_1^-$ along $\eps_{d^2-3d+3}^+$ and $\eps_1^-$ to obtain a disk $D$ and a continuous, surjective map $\iota: D \rightarrow \V_d \cap X_1$ which is an embedding on each component $(R_d)_j^{\pm}$.  Note that this pasting identifies all $2(d^2-3d+3)$ copies of the origin $(0,0)$ as the ``center" of the disk $D$.

To prove that $\iota$ is an embedding, we need only show that $\iota$ is injective.  Dropping sub- and superscripts for ease of notation, suppose that $(r,s)$ and $(r',s')$ are in distinct copies of $R_d$ and $\iota(r,s) = \iota(r',s') = (re^{i\theta},se^{i\psi})$.  Then $r = r'$ and $s = s'$.  In addition, $\theta(r,s) = \theta(r',s')$ and $\psi(r,s) = \psi(r',s')$, which implies that $\nu(r,s) = \nu(r',s')$ and $\omega(r,s) - \omega(r',s')$ is a multiple of $2\pi$, so either $\omega(r,s) = \omega(r',s')$, or (without loss of generality) $\omega(r,s) = 0$ and $\omega(r',s') = 2\pi$.  In the first case, we have that there exists $j$ such that $(r,s) \in (R_d)_j^+$ and $(r',s') \in (R_d)_j^-$.  Then $\omega(r,s) \in [\pi,2\pi]$ and $\omega(r',s') \in [0,\pi]$, which implies that $\omega(r,s) = \omega(r',s') = \pi$ and thus $(r,s)$ and $(r',s')$ are identified via the pasting used to construct $D$.  In the second case, we have $(r,s) \in (R_d)_j^+$ and $(r',s') \in (R_d)_{j-1}^-$ are also two points identified via the pasting used to construct $D$.  In any case, $(r,s)$ and $(r',s')$ are the same point in $D$, and $\iota$ is (globally) injective.

Finally, by the proof of Lemma~\ref{disko}, projection of $I \X I$ onto the second factor induces a singular foliation of $R_d$ by arcs with one critical point at $(0,0)$.  We can use this singular foliation to construct a radial singular foliation $f$ of $I \X I$ that induces a singular foliation of $R_d$ by arcs with one critical point at $(0,0)$.  By Lemma~\ref{induce}, the foliation $f$ induces a singular foliation of $X_1$ by copies of $S^3$, where the intersection of each $S^3$ with $\iota(D)$ is a collection of arcs glued along their boundary; that is, a knot which is isotopic to $\iota(D) \cap \pd X_1$.  It follows that $f$ induces a singular foliation of $D$ by copies of $S^1$ with a single critical point at the center $\iota(0,0) = (0,0)$, and thus $\iota(D)$ is isotopic to the cone on $\iota(D) \cap \pd X_1$.
\end{proof}

\begin{proof}[Proof of Theorem~\ref{thm:main2}]
First, we consider the intersections $\V_d \cap X_i$.  By Lemma~\ref{lem:smooth}, we know that $\V_d \cap X_i$ is a smooth surface, and by Proposition~\ref{prop:cone}, we also have that $\V_d \cap X_i$ is the cone on the intersection $\V_d \cap \pd X_i$.  Since $\V_d \cap \pd X_i$ is a single unknotted curve by Lemmas~\ref{homol} and~\ref{disjoint}, we have that $\V_d \cap X_i$ is a single boundary parallel disk.

Next, we consider the intersections $\V_d \cap H_i$.  By Proposition~\ref{shadow} and using the three-fold symmetry of the construction, $\V_d \cap H_i$ is a collection of boundary parallel arcs.  It follows that $\V_d$ is in bridge trisected position with respect to the standard trisection of $\CP^2$.  Finally, Corollary~\ref{cor:hexclassify} asserts that the corresponding shadow diagram $(\aaa,\bbb,\ccc)$ coincides with the family $(A)_d$.

The careful reader will note that the arguments in this section also apply, with only slight modification, to the variety
\[ \V'_d = \{[z_1:z_2:z_3] \in \CP^2 : z_1^{d-1}z_2 + z_2^{d-1}z_3 + z_3^{d-1}z_1= 0\}.\]
In particular, in Lemma~\ref{homol}, we get $[K_1] = (d-1)[\A] + [\n]$, and in Corollary~\ref{cor:hexclassify}, we have that the associated shadow diagram agrees instead with $(E)_d$.  We conclude that $\V'_d$ is in efficient bridge position, as desired.
\end{proof}

\section{Questions}\label{sec:questions}

In this section, we discuss several related open questions that may be of interest.  In pursuit of understanding the space of bridge trisections and the relationships between various families, we note that every bridge trisection of $\K$ in $\CP^2$ has infinitely many distinct shadow diagrams.  For example, let $\abc$ be a shadow diagram, with $a$ an arc in $\aaa$.  If $a'$ is another arc in $\Sigma$ with the same endpoints as $a$, and $a \cup a'$ is a simple closed curve disjoint from $\aaa -\{a\}$ and homotopic to $\A$, then $(\aaa',\bbb,\ccc)$ is also a shadow diagram for the same bridge trisection, where the arc $a'$ replaces $a$ in $\aaa'$.  This operation is sometimes called a \emph{shadow slide}.

\begin{question}
Does the diagram $(A)_d$ determine the same bridge trisection as the diagram $(E)_d$?  What about the pairs $(B)_d$ and $(F)_d$, $(C)_d$ and $(G)_d$, or $(D)_d$ and $(H)_d$?
\end{question}

In~\cite{JMMZ}, the authors found surfaces in $S^4$ with two distinct bridge trisections that have the same bridge number, although the obstruction is group-theoretic and is not applicable to this case.  On the other hand, it is possible that each corresponding pair does in fact determine the same bridge trisection.  We leave it as an exercise to the reader to show that $(D)_1$ and $(H)_1$ determine isotopic bridge trisections, as do $(D)_2$ and $(H)_2$ (this is a more difficult exercise), and as do $(B)_2$ and $(F)_2$.  In Figure~\ref{fig:slide1}, we see diagram $(A)_3$ at right and the result of three shadow slides at center, and we can verify that the center diagram is isotopic to $(E)_3$, shown at left.  Similarly, in Figure~\ref{fig:slide2}, we see diagram $(C)_2$ at left, the result of four shadow slides at center, and we can verify (by checking the homology classes of $\ab$ and $\bc$, for example) that the center diagram is isotopic to $(G)_2$, shown at right.  We do not know whether this correspondence persists for higher degrees, but we suspect it might be the case.

\begin{figure}[h!]
\begin{subfigure}{.32\textwidth}
  \centering
  \includegraphics[width=.9\linewidth]{fig/verA1.eps}
  \label{fig:s1A}
  \caption{$(A)_3$}
\end{subfigure}
\begin{subfigure}{.32\textwidth}
  \centering
  \includegraphics[width=.9\linewidth]{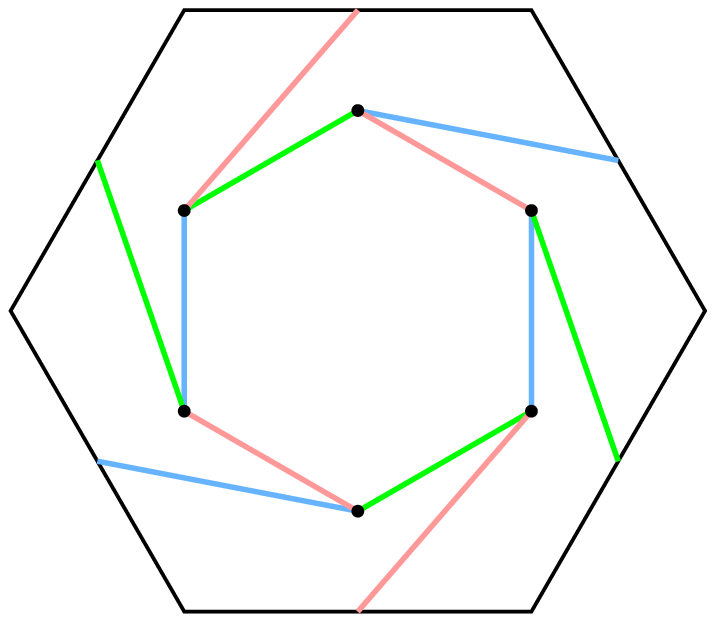}
  \label{fig:s1B}
  \caption{Slides on $(A)_3$}
\end{subfigure}
\begin{subfigure}{.32\textwidth}
  \centering
  \includegraphics[width=.9\linewidth]{fig/verE1.eps}
  \label{fig:s1C}
  \caption{$(E)_3$}
\end{subfigure}
	\caption{}
	\label{fig:slide1}
\end{figure}

\begin{figure}[h!]
\begin{subfigure}{.32\textwidth}
  \centering
  \includegraphics[width=.9\linewidth]{fig/verC1.eps}
  \label{fig:s2A}
  \caption{$(C)_2$}
\end{subfigure}
\begin{subfigure}{.32\textwidth}
  \centering
  \includegraphics[width=.9\linewidth]{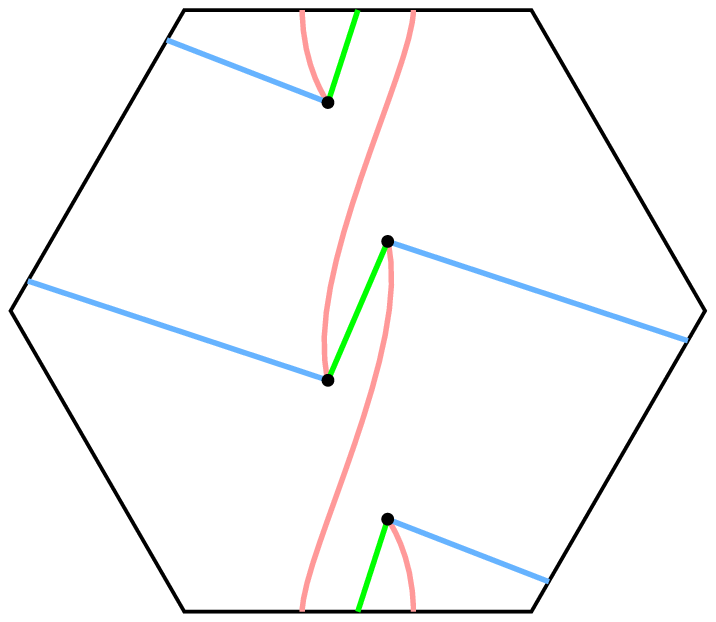}
  \label{fig:s2B}
  \caption{Slides on $(C)_2$}
\end{subfigure}
\begin{subfigure}{.32\textwidth}
  \centering
  \includegraphics[width=.9\linewidth]{fig/verG1.eps}
  \label{fig:s2C}
  \caption{$(G)_2$}
\end{subfigure}
	\caption{}
	\label{fig:slide2}
\end{figure}

For diagrams $(A)_d$ equivalent to $(E)_d$, we can consider the complex-geometric version of their equivalence.

\begin{question}
Suppose that $(A)_d$ is equivalent to $(E)_d$.  In this case, does the straight-line isotopy given by
\[ t(z_1z_2^{d-1} + z_2z_3^{d-1} + z_3z_1^{d-1}) + (1-t)(z_1^{d-1}z_2 + z_2^{d-1}z_3 + z_3^{d-1}z_1) = 0\]
take $\V_d$ to $\V_d'$ via complex surfaces which meet the standard trisection of $\CP^2$ in an efficient bridge position?
\end{question}

It was proved in~\cite{HKM} that any two bridge trisections for isotopic surfaces are related by a sequence of perturbation and deperturbation moves.  See~\cite{MZB1,MZB2} for further details about perturbations.

\begin{question}
Which families of hexagonal lattice diagrams are related only by deperturbations?
\end{question}

For instance, we might conjecture that the bridge trisections determined by $(D)_d$ can be deperturbed to the one determined by $(C)_d$, which in turn can be deperturbed further to that of $(B)_d$, which can finally be deperturbed further still to the bridge trisection determined by $(A)_d$.

We can also consider hexagonal lattice diagrams that do not minimize genus in their homology classes.  If $T^2$ is the unknotted torus in $S^4$, and if $\K \subset \CP^2$ and $(\CP^2,\K') = (\CP^2,\K) \# (S^4,T^2)$, we say that $\K'$ is related to $\K$ by a \emph{trivial 1-handle addition}.  In light of the prevailing theme that combinatorial and topological simplicity are intertwined, we ask

\begin{question}\label{nonmin}
Suppose $\K$ admits a hexagonal lattice diagram but $\K$ is not genus-minimizing in its homology class.  Is $\K$ obtained from some complex curve $\C_d$ by trivial 1-handle additions?
\end{question}

Numerous non-genus-minimizing hexagonal lattice diagrams exist, but there would be many more cases to consider in order to answer Question~\ref{nonmin}.  We might also consider triples $\abc$ that yield hexagonal lattices but such that the paired curves $\ab$, $\bc$, and $\ca$ need not be unknots or unlinks.  In general, these curves are torus links that can be capped off with singular surfaces.  Bridge trisections of singular surfaces in $S^4$ have been studied in~\cite{CK} and~\cite{HKMsing}.

\begin{question}
Does a similar correspondence exist between singular hexagonal lattice diagrams and singular complex surfaces?
\end{question}

Finally, we may ask to what extent these sorts of structures appear in other complex manifolds.  For example, Islambouli, Karimi, Lambert-Cole, and Meier have proved that complex curves in $S^2 \X S^2$ admit (rectangular) lattice diagrams with respect to a minimal 4-section~\cite{IKLCM}.

\begin{question}
Is there a combinatorial characterization of shadow diagrams (with respect to the minimal genus trisection) of complex curves in $S^2 \X S^2$?  What about curves in other complex projective surfaces?
\end{question}

\section{Appendix: Case-by-case analysis}\label{sec:appendix}

In this section, we carry out the case-by-case analysis necessary to prove Theorem~\ref{thm:family}.  We use that notation $(p,q)$ to represent a curve whose homology is $p[\A] + q[\n]$ in $H_1(\Sigma)$.  In addition, recall from Lemma~\ref{lem:dist} that $\eps b = \la \ac, \bc \ra$, so that $b = |\la \ac,\bc \ra|$.  The next lemma describes six possible cases for each of $\ab$, $\ac$, and $\bc$, yielding 216 cases in total.

\begin{lemma}\label{lem:cases}
For any hexagonal lattice diagram $(\aaa,\bbb,\ccc)$, we have that $\ac$ must be one of the following:
\be
\item $(n,1)$ such that $n \in \Z$,
\item $(n,-1)$ such that $n \in \Z$,
\item $(n,0)$ such that $n \in \Z$, $n\neq 0$,
\item $(n+1,n)$ such that $n \in \Z$,
\item $(n-1,n)$ such that $n \in \Z$,
\item $(n,n)$ such that $n \in \Z$, $n \neq 0$.
\ee
Similarly, $\bc$ is one of the following:
\be
\item[(a)] $(1,m)$ such that $m \in \Z$,
\item[(b)] $(-1,m)$ such that $m \in \Z$,
\item[(c)] $(0,m)$ such that $m \in \Z$, $m\neq 0$,
\item[(d)] $(m,m+1)$ such that $m \in \Z$,
\item[(e)] $(m,m-1)$ such that $m \in \Z$,
\item[(f)] $(m,m)$ such that $m \in \Z$, $m \neq 0$.
\ee
Finally, $\ab$ is one of the following:
\be
\item[(i)] $(1,\ell)$ such that $\ell \in \Z$
\item[(ii)] $(-1,\ell)$ such that $\ell \in \Z$,
\item[(iii)] $(\ell,0)$ such that $\ell \in \Z$, $\ell \neq 0$.
\item[(iv)] $(\ell,1)$ such that $\ell \in \Z$,
\item[(v)] $(\ell,-1)$ such that $\ell \in \Z$,
\item[(vi)] $(0,\ell)$ such that $\ell \in \Z$, $\ell \neq 0$.
\ee

\end{lemma}

\begin{proof}
By Lemma~\ref{lem:possible}, we know that one of the following is true, where each possibility is designated by its corresponding case for $\ab$ above:
\[ \ab = \ell\A, \, \, (iii); \quad \ab = \ell\n, \, \,(vi); \quad \la \ab,\A \ra = \pm 1, \, \, (v) \text{ or } (iv); \,\, \,\, \text{or}\,\, \,\, \la \ab,\n \ra = \pm 1, \, \, (i) \text{ or } (ii).\]
These six possibilities coincide precisely with the six options given by the lemma for $\ab$.  Similarly, each of $\ac$ and $\bc$ must satisfy one of the following, with cases designated as above:
\[ \ac = n\A, \, \, (3); \quad \ac = n\g, \, \,(6); \quad \la \ac,\A \ra = \pm 1, \, \, (2) \text{ or } (1); \,\, \,\, \text{or}\,\, \,\, \la \ac,\n \ra = \pm 1, \, \, (4) \text{ or } (5).\]
\[ \bc = m\n, \, \, (c); \quad \bc = m\g, \, \,(f); \quad \la \bc,\n \ra = \pm 1, \, \, (a) \text{ or } (b); \,\, \,\, \text{or}\,\, \,\, \la \bc,\n \ra = \pm 1, \, \, (e) \text{ or } (d).\]
It is straightforward to verify that these possibilities give rise to the six cases listed in the lemma.
\end{proof}

\begin{proposition}\label{prop:work}
Suppose $(\aaa,\bbb,\ccc)$ is a hexagonal lattice diagram for a surface $\K$ that minimizes genus in its homology class.  Then either the degree of $\K$ is $0$, $\pm 1$, or $\pm 2$, or $(\aaa,\bbb,\ccc)$ is a member of one of the 32 types of diagrams contained in Table~\ref{table1}.  The diagrams in turn fall into one of eight equivalence classes.
\end{proposition}

\begin{proof}
Suppose $(\aaa,\bbb,\ccc)$ is a hexagonal lattice diagram for a surface $\K$ that minimizes genus in its homology class.  Let $b$ be the bridge number, $(c_1,c_2,c_3)$ the corresponding patch numbers, and $\eps$ the sign of the bridge points.  We complete a case-by-case analysis of the 216 possible cases given by Lemma~\ref{lem:cases}.  As in Section~\ref{sec:prelim}, we set the convention that
\[ \ab  = p_1[\A] + q_1[\n], \quad \bc = p_2[\n] + q_2[\g], \quad \text{ and } \ca = p_3[\g] + q_3[\A].\]
Additionally, recall the formulas from Lemmas~\ref{int} and~\ref{lem:genus},
\[ [\K]^2 = p_1q_1 + p_2q_2 + p_3q_3 + \eps b.\]
\[ g(\K) = \frac{b - c_1 - c_2 - c_3}{2} + 1.\] 
Note that we terminate each case if it happens that $0$, $\pm 1$, or $\pm 2$ or if we can show that $\K$ is not genus minimizing in its homology class.  Cases are numbered and lettered as in Lemma~\ref{lem:cases}.  Note that some cases terminate before reaching the (i)-(vi) level of examination.  Any cases which satisfy $|d| >2$ and in which $\K$ is genus-minimizing in its homology class are marked with an $*$ to be collected in Table~\ref{table1}.

\be

\item Suppose $\ac = (n,1)$:  Then $\ca = (-n,-1)$, so that $p_3 = 1$, $q_3 = 1-n$, and $p_3q_3 = 1-n$.
{\be
\item $\bc = (1,m)$:  In this case, $p_2 = m-1$, $q_2 = -1$, and $p_2q_2 = 1-m$.  We compute $b = |nm-1|$, $\eps b = nm-1$, and $\ab = (n-1,1-m)$.  Thus
\begin{eqnarray*}
p_1p_2 &=& n - nm -1 +m \\
\text{$[\K]^2$} &=& (n-nm-1+m) + (1-m) + (1-n) + (nm-1) = 0.
 \end{eqnarray*}

\item $\bc = (-1,m)$:  In this case, $p_2 = m+1$, $q_2 = 1$, and $p_2q_2 = m+1$.  We compute $b = |nm+1|$, $\eps b = nm+1$, and $\ab = (n+1,1-m)$.  Thus
\begin{eqnarray*}
p_1p_2 &=& n - nm + 1 -m \\
\text{$[\K]^2$} &=& (n-nm+1-m) + (m+1) + (1-n) + (nm+1) = 4.
 \end{eqnarray*}

\item $\bc = (0,m)$:  In this case, $p_2 = m$, $q_2 = 0$, and $p_2q_2 = 0$.   We compute $b = |nm|$, $\eps b = nm$, and $\ab = (n,1-m)$.  Thus
\begin{eqnarray*}
p_1p_2 &=& n - nm \\
\text{$[\K]^2$} &=& (n-nm) + 0 + (1-n) + nm = 1.
 \end{eqnarray*}
 
 \item $\bc = (m,m+1)$:  In this case, $p_2 = 1$, $q_2 = -m$, and $p_2q_2 = -m$.   We compute $b = |nm+n-m|$, $\eps b = nm+n-m$, and $\ab = (n-m,-m)$.  Thus
\begin{eqnarray*}
p_1p_2 &=& -nm + m^2 \\
\text{$[\K]^2$} &=& (-nm+m^2) + -m + (1-n) + (nm+n-m) = (m-1)^2.
 \end{eqnarray*}
 Next, we compute
 \[ g(\K) = \frac{|nm+n-m| - 2 - c_1}{2} + 1 = \frac{|nm+n-m| - c_1}{2}.\]
 To find possible values of $c_1$, we must proceed further still.  Note that if $\K$ is genus-minimizing, then $g(\K) = \frac{(m-2)(m-3)}{2}$ if $m > 1$ and $g(\K) = \frac{m^2+m}{2}$ if $m < 1$.  We examine six sub-cases:
{\be
\item *$\ab = (1,\ell)$:  In this case, $n = m+1$ and $c_1 = 1$, so that $g(\K) = \frac{|m^2+m+m+1-m|-1}{2} = \frac{m^2+m}{2}$.  Here $\K$ is not genus-minimizing for $m > 1$ but $\K$ is genus-minimizing for $m < 1$.
\item $\ab = (-1,\ell)$:  In this case, $n = m-1$ and $c_1=1$, so that (for $m \neq 0)$, we have $g(\K) = \frac{|m^2-m+m-1-m|-1}{2} = \frac{m^2-m-2}{2} = \frac{(m-2)(m+1)}{2}$.  It follows that $\K$ is not genus-minimizing for any $m \neq -1,0,1,2$.
\item $\ab = (\ell,0)$:  In this case, $m=0$, so $[\K]^2 = 1$.
\item $\ab=(\ell,1)$:  In this case, $m=-1$, so $[\K]^2 = 4$.
\item $\ab = (\ell,-1)$:  In this case, $m=1$, so $[\K]^2 = 0$.
\item *$\ab =(0,\ell)$:  In this case, $n=m$ and $c_1 = |m|$, so that $g(\K) = \frac{m^2 - |m|}{2}$.  Here $\K$ is not genus-minimizing for $m > 1$ but $\K$ is genus-minimizing for $m < 1$.

\ee}

 \item $\bc = (m,m-1)$:  In this case, $p_2 = -1$, $q_2 = -m$, and $p_2q_2 = m$.   We compute $b = |nm-n-m|$, $\eps b = nm-n-m$, and $\ab = (n-m,2-m)$.  Thus
\begin{eqnarray*}
p_1p_2 &=& 2n-nm-2m+m^2 \\
\text{$[\K]^2$} &=& (2n-nm-2m+m^2) + m + (1-n) + (nm-n-m)  = (m-1)^2.
 \end{eqnarray*}
 Next, we compute
 \[ g(\K) = \frac{|nm-n-m| - 2 - c_1}{2} + 1 = \frac{|nm-n-m| - c_1}{2}.\]
 To find possible values of $c_1$, we must proceed further still.  Note that if $\K$ is genus-minimizing, then $g(\K) = \frac{(m-2)(m-3)}{2}$ if $m > 1$ and $g(\K) = \frac{m^2+m}{2}$ if $m < 1$.  We examine six sub-cases:
{\be
\item $\ab = (1,\ell)$.  In this case, $n = m+1$ and $c_1 = 1$, so that (for $m \neq 0$) $g(\K) = \frac{|m^2+m-(m+1)-m| -1}{2} = \frac{m^2-m-2}{2} = \frac{(m-2)(m+1)}{2}$.  It follows that $\K$ is not genus-minimizing for any $m \neq -1,0,1,2$.
\item $\ab = (-1,\ell)$.  In this case, $n = m-1$ and $c_1 = 1$, so that (for $|m| > 3$) $g(\K) = \frac{|m^2-m-(m-1)-m| - 1}{2} = \frac{m(m-3)}{2}$.  It follows that $\K$ is not genus-minimizing for $|m| > 3$.
\item $\ab = (\ell,0)$:  In this case, $m=2$, so $[\K]^2 = 1$.
\item $\ab = (\ell,1)$.  In this case, $m=1$, so $[\K]^2 = 0$.
\item $\ab = (\ell,-1)$.  In this case, $m = 3$, so $[\K]^2 = 4$.
\item $\ab = (0,\ell)$:  In this case, $n=m$ and $c_1 = |m-2|$, so that $g(\K) = \frac{|m^2-m-m| - |m-2|}{2}$.  If $m > 1$, then $g(\K) = \frac{m^2-3m+2}{2} = \frac{(m-1)(m-2)}{2}$.  If $m < 1$, then $g(\K) = \frac{m^2-m-2}{2} = \frac{(m-2)(m+1)}{2}$.  In either case, $\K$ is not genus-minimizing.
\ee}

\item $\bc = (m,m)$:  In this case, $p_2 = 0$, $q_2 = -m$, and $p_2q_2 = 0$.   We compute $b = |nm-m|$, $\eps b = nm-m$, and $\ab = (n-m,1-m)$.  We compute
\begin{eqnarray*}
p_1p_2 &=& n - nm - m + m^2 \\
\text{$[\K]^2$} &=& (n-nm-m+m^2) + 0 + (1-n) + (nm-m) = (m-1)^2.
 \end{eqnarray*}
Next, we compute
 \[ g(\K) = \frac{|nm-m| - 1 - |m| - c_1}{2} + 1 = \frac{|nm-m| - |m| - c_1 + 1}{2}.\]
 To find possible values of $c_1$, we must proceed further still.  Note that if $\K$ is genus-minimizing, then $g(\K) = \frac{(m-2)(m-3)}{2}$ if $m > 1$ and $g(\K) = \frac{m^2+m}{2}$ if $m < 1$.  We examine six sub-cases:
{\be
\item *$\ab = (1,\ell)$.  In this case, $n=m+1$ and $c_1 =1$, so that $g(\K) = \frac{|m^2+m-m| - |m| -1+1}{2} = \frac{m^2 - |m|}{2}$.  Here $\K$ is not genus-minimizing for $m > 1$ but $\K$ is genus-minimizing for $m < 1$.
\item $\ab = (-1,\ell)$.  In this case, $n = m-1$ and $c_1 = 1$, so that $g(\K) = \frac{|m^2-m-m| - |m| - 1 + 1}{2} = \frac{|m^2-2m|-|m|}{2}$.  If $m>1$, then $g(\K) = \frac{m(m-3)}{2}$, and if $m < 1$, then $g(\K) = \frac{m(m-1)}{2}$.  Thus, if $|m| > 3$, we have that $\K$ is not genus-minimizing.
\item $\ab = (\ell,0)$:  In this case, $m=1$, so $[\K]^2 = 0$.
\item $\ab = (\ell,1)$.  In this case, $m = 0$, so $[\K]^2 = 1$.
\item $\ab = (\ell,-1)$.  In this case, $m=2$, so $[\K]^2 = 1$.
\item *$\ab = (0,\ell)$:  In this case, $n = m$ and $c_1 = |1-m|$, so that $g(\K) = \frac{|m^2 - m| - |m| - |1-m| + 1}{2}$.  If $m > 1$, we have $g(\K) = \frac{m^2-m-m-(m-1)+1}{2} = \frac{(m-1)(m-2)}{2}$, so $\K$ is not genus-minimizing.  If $m <1$, we have $g(\K) = \frac{m^2-m+m-(1-m)+1}{2} = \frac{m^2+m}{2}$, so $\K$ is genus-minimizing.
\ee}

\ee}

\item Suppose $\ac = (n,-1)$:  Then $\ca = (-n,1)$, so that $p_3 = -1$, $q_3 = -1-n$, and $p_3q_3 = n+1$.

{\be
\item $\bc = (1,m)$:  In this case, $p_2 = m-1$, $q_2 = -1$, and $p_2q_2 = 1-m$.  We compute $b = |nm+1|$, $\eps b = nm+1$, and $\ab = (n-1,-1-m)$.  Thus
\begin{eqnarray*}
p_1p_2 &=& -n - nm +1 +m \\
\text{$[\K]^2$} &=& (-n-nm+1+m) + (1-m) + (n+1) + (nm+1) = 4.
 \end{eqnarray*}

\item $\bc = (-1,m)$:  In this case, $p_2 = m+1$, $q_2 = 1$, and $p_2q_2 = m+1$.  We compute $b = |nm-1|$, $\eps b = nm-1$, and $\ab = (n+1,-1-m)$.  Thus
\begin{eqnarray*}
p_1p_2 &=& -n - nm - 1 -m \\
\text{$[\K]^2$} &=& (-n-nm-1-m) + (m+1) + (n+1) + (nm-1) = 0.
 \end{eqnarray*}

\item $\bc = (0,m)$:  In this case, $p_2 = m$, $q_2 = 0$, and $p_2q_2 = 0$.   We compute $b = |nm|$, $\eps b = nm$, and $\ab = (n,-1-m)$.  Thus
\begin{eqnarray*}
p_1p_2 &=& -n - nm \\
\text{$[\K]^2$} &=& (-n-nm) + 0 + (n+1) + nm = 1.
 \end{eqnarray*}

\item $\bc = (m,m+1)$:  In this case, $p_2 = 1$, $q_2 = -m$, and $p_2q_2 = -m$.   We compute $b = |nm+n+m|$, $\eps b = nm+n+m$, and $\ab = (n-m,-2-m)$.  Thus
\begin{eqnarray*}
p_1p_2 &=& -2n -nm + 2m + m^2 \\
\text{$[\K]^2$} &=& (-2n-nm + 2m+m^2) + -m + (n+1) + (nm+n+m) = (m+1)^2.
 \end{eqnarray*}
 Next, we compute
 \[ g(\K) = \frac{|nm+n+m| - 2 - c_1}{2} + 1 = \frac{|nm+n+m| - c_1}{2}.\]
 To find possible values of $c_1$, we must proceed further still.  Note that if $\K$ is genus-minimizing, then $g(\K) = \frac{m(m-1)}{2}$ if $m > -1$ and $g(\K) = \frac{(m+3)(m+2)}{2}$ if $m < -1$.  We examine six sub-cases:
 
 {\be
\item $\ab= ( 1,\ell)$:  In this case, $n = m+1$ and $c_1 = 1$, so that (for $|m| > 3$), $g(\K) = \frac{|m^2+m+(m+1)+m| - 1}{2} = \frac{m(m+3)}{2}$, and $\K$ is not genus-minimizing.
\item $\ab = (-1,\ell)$:  In this case, $n = m-1$ and $c_1 = 1$, so that (for $|m| > 1$), $g(\K) = \frac{|m^2-m+(m-1)+m| - 1}{2} = \frac{m^2+m-2}{2} = \frac{(m+2)(m-1)}{2}$, and $\K$ is not genus-minimizing.
 \item $\ab = (\ell,0)$:  In this case, $m=-2$, so $[\K]^2 = 1$.
  \item $\ab = (\ell,1)$:  In this case, $m = -3$, so $[\K]^2 = 4$.
 \item $\ab=(\ell,-1)$:  In this case, $m=-1$, so $[\K]^2 = 0$.
 \item $\ab = (0,\ell)$:  In this case, $n=m$ and $c_1 = |m+2|$, so that $g(\K) = \frac{|m^2 + m +m| - |m+2|}{2}$.  If $m > -1$, then $g(\K) = \frac{m^2+m-2}{2} = \frac{(m+2)(m-1)}{2}$, so $\K$ is not genus-minimizing.  If $m < -1$, then $g(\K) = \frac{m^2+2m-(-2-m)}{2} = \frac{(m+1)(m+2)}{2}$, so again $\K$ is not genus-minimizing.
 \ee}

\item $\bc = (m,m-1)$:  In this case, $p_2 = -1$, $q_2 = -m$, and $p_2q_2 = m$.   We compute $b = |nm-n+m|$, $\eps b = nm-n+m$, and $\ab = (n-m,-m)$.  Thus
\begin{eqnarray*}
p_1p_2 &=& -nm + m^2 \\
\text{$[\K]^2$} &=& (-nm+m^2) + m + (n+1) + (nm-n+m)  = (m+1)^2.
 \end{eqnarray*}
 Next, we compute
 \[ g(\K) = \frac{|nm-n+m| - 2 - c_1}{2} + 1 = \frac{|nm-n+m| - c_1}{2}.\]
 To find possible values of $c_1$, we must proceed further still.  Note that if $\K$ is genus-minimizing, then $g(\K) = \frac{m(m-1)}{2}$ if $m > -1$ and $g(\K) = \frac{(m+3)(m+2)}{2}$ if $m <-1$.  We examine six sub-cases:
 
{\be
\item $\ab =(1,\ell)$:  In this case, $n = m+1$ and $c_1 = 1$, so that (for $|m| > 1$) $g(\K) = \frac{|m^2+m-(m+1)+m|-1}{2} = \frac{m^2+m-2}{2} = \frac{(m+2)(m-1)}{2}$, and $\K$ is not genus-minimizing.
\item *$\ab = (-1,\ell)$:  In this case, $n = m-1$ and $c_1 = 1$, so that $g(K) = \frac{m^2-m-(m-1)+m|-1}{2} = \frac{m^2-m}{2}$ and $\K$ is genus-minimizing for $m > -1$.
\item $\ab =(\ell,0)$:  In this case, $m=0$, so $[\K]^2 = 1$.
\item $\ab =(\ell,1)$:  In this case, $m = -1$, so $[\K]^2 = 0$.
\item $\ab =(\ell,-1)$:  In this case, $m=1$, so $[\K]^2 = 4$.
\item *$\ab=(0,\ell)$:  In this case, $n=m$ and $c_1=|m|$, so that $g(\K) = \frac{|m^2-m+m| - |m|}{2} = \frac{m^2 - |m|}{2}$.  Thus, $\K$ is not genus-minimizing if $m < - 1$ but $\K$ is genus-minimizing if $m > -1$.
\ee}

\item $\bc = (m,m)$:  In this case, $p_2 = 0$, $q_2 = -m$, and $p_2q_2 = 0$.   We compute $b = |nm+m|$, $\eps b = nm+m$, and $\ab = (n-m,-1-m)$.  Thus
\begin{eqnarray*}
p_1p_2 &=& -n - nm + m + m^2 \\
\text{$[\K]^2$} &=& (-n-nm+m+m^2) + 0 + (n+1) + (nm+m) = (m+1)^2.
 \end{eqnarray*}
Next, we compute
 \[ g(\K) = \frac{|nm+m| - 1 - |m| - c_1}{2} + 1 = \frac{|nm+m| - |m| - c_1 + 1}{2}.\]
 To find possible values of $c_1$, we must proceed further still.  Note that if $\K$ is genus-minimizing, then $g(\K) = \frac{m(m-1)}{2}$ if $m > -1$ and $g(\K) = \frac{(m+3)(m+2)}{2}$ if $m < -1$.  We examine six sub-cases:

{\be
\item $\ab = (1,\ell)$:  In this case, $n = m+1$ and $c_1 = 1$, so that $g(\K) = \frac{|m^2+m+m|-|m|}{2}$.  If $m > -1$, then $g(\K) = \frac{m^2+m}{2}$.  If $m < -1$, then $g(\K) = \frac{m^2 + 3m}{2}$.  In either case, $\K$ is not genus-minimizing.
\item *$\ab = (-1,\ell)$:  In this case, $n = m-1$ and $c_1 =1$, so that $g(\K) = \frac{|m^2-m+m|-|m|}{2} = \frac{m^2 - |m|}{2}$.  Thus, $\K$ is not genus-minimizing if $m < -1$ but $\K$ is genus-minimizing if $m > -1$.
\item $\ab = (\ell,0)$:  In this case, $m=-1$, so $[\K]^2 = 0$.
\item $\ab = (\ell,1)$:  In this case, $m = -2$, so $[\K]^2 = 1$.
\item $\ab = (\ell,-1)$:  In this case, $m=0$, so $[\K]^2 = 1$.
\item *$\ab = (0,\ell)$:  In this case, $n=m$ and $c_1 = |m+1|$, so that $g(\K) = \frac{|m^2+m|-|m|-|m+1|+1}{2}$.  If $m > -1$, then $g(\K) = \frac{m(m-1)}{2}$, and $\K$ is genus-minimizing.  If $m < -1$, then $g(\K) = \frac{m^2 + m - (-m) - (-m-1) + 1}{2} = \frac{(m+2)(m+1)}{2}$, and so if, in addition, $m < -3$, $\K$ is not genus-minimizing.
\ee}
 
\ee}

\item Suppose $\ac = (n,0)$:  Then $\ca = (-n,0)$, so that $p_3 = 0$, $q_3 = -n$, and $p_3q_3 = 0$.

{\be
\item $\bc = (1,m)$:  In this case, $p_2 = m-1$, $q_2 = -1$, and $p_2q_2 = 1-m$.  We compute $b = |nm|$, $\eps b = nm$, and $\ab = (n-1,-m)$.  Thus
\begin{eqnarray*}
p_1p_2 &=& -nm+m \\
\text{$[\K]^2$} &=& (-nm+m) + (1-m) + 0 + nm = 1.
 \end{eqnarray*}

\item $\bc = (-1,m)$:  In this case, $p_2 = m+1$, $q_2 = 1$, and $p_2q_2 = m+1$.  We compute $b = |nm|$, $\eps b = nm$, and $\ab = (n+1,-m)$.  Thus
\begin{eqnarray*}
p_1p_2 &=& -nm-m \\
\text{$[\K]^2$} &=& (-nm-m) + (m+1) + 0 + nm = 1.
 \end{eqnarray*}

\item $\bc = (0,m)$:  In this case, $p_2 = m$, $q_2 = 0$, and $p_2q_2 = 0$.   We compute $b = |nm|$, $\eps b = nm$, and $\ab = (n,-m)$.  Thus
\begin{eqnarray*}
p_1p_2 &=& - nm \\
\text{$[\K]^2$} &=& -nm + 0 + 0 + nm = 0.
 \end{eqnarray*}

\item $\bc = (m,m+1)$:  In this case, $p_2 = 1$, $q_2 = -m$, and $p_2q_2 = -m$.   We compute $b = |nm+n|$, $\eps b = nm+n$, and $\ab = (n-m,-m-1)$.  Thus
\begin{eqnarray*}
p_1p_2 &=& -nm -n + m^2 + m \\
\text{$[\K]^2$} &=& (-nm -n + m^2 + m) + -m + 0 + (nm+n) = m^2.
 \end{eqnarray*}
 Next, we compute
 \[ g(\K) = \frac{|nm+n| - |n| - 1 - c_1}{2} + 1 = \frac{|nm+n| - |n| - c_1+1}{2}.\]
 To find possible values of $c_1$, we must proceed further still.  Note that if $\K$ is genus-minimizing, then $g(\K) = \frac{
 (m-1)(m-2)}{2}$ if $m > 0$ and $g(\K) = \frac{(m+2)(m+1)}{2}$ if $m < 0$.  We examine six sub-cases:
 
 {\be
 \item *$\ab =(1,\ell)$:  In this case, $n = m+1$ and $c_1 = 1$, so that $g(\K) = \frac{|m^2+m+m+1| - |m+1| -1+1}{2}$.  If $m > 0$, then $g(\K) = \frac{m^2+m}{2}$, so $\K$ is not genus-minimizing.  If $m < 0$, then $g(\K) = \frac{m^2+m+m+1-(-m-1)-1+1}{2} = \frac{(m+2)(m+1)}{2}$, and $\K$ is genus-minimizing.
\item $\ab = (-1,\ell)$:  In this case, $n = m-1$ and $c_1= 1$, so that $g(\K) = \frac{|m^2-m+m-1| - |m-1|-1+1}{2} = \frac{m^2-1-|m-1|}{2}$.  If $m > 0$, then $g(\K) = \frac{m^2-m}{2}$.  If $m < 0$, then $g(\K) = \frac{m^2+m-2}{2} = \frac{(m+2)(m-1)}{2}$.  In either case, $\K$ is not genus-minimizing.
 \item $\ab = (\ell,0)$.  In this case, $m=-1$, so $[\K]^2 = 1$.
 \item $\ab =(\ell,1)$:  In this case, $m = -2$, so $[\K]^2 = 4$.
\item $\ab = (\ell,-1)$:  In this case, $m=0$, so $[\K]^2 = 0$.
 \item *$\ab = (0,\ell)$.  In this case, $n=m$ and $c_1 = |m+1|$, so that $g(\K) = \frac{|m^2 +m|-|m|-|m+1|+1}{2}$.  If $m > 0$, then $g(\K) = \frac{m(m-1)}{2}$, so $\K$ is not genus-minimizing.  If $m < 0$, then $g(\K) = \frac{m^2+m-(-m)-(-m-1)+1}{2} = \frac{(m+2)(m+1)}{2}$, and $\K$ is genus-minimizing.
\ee}

\item $\bc = (m,m-1)$:  In this case, $p_2 = -1$, $q_2 = -m$, and $p_2q_2 = m$.   We compute $b = |nm-n|$, $\eps b = nm-n$, and $\ab = (n-m,-m+1)$.  Thus
\begin{eqnarray*}
p_1p_2 &=& -nm +n + m^2 -m \\
\text{$[\K]^2$} &=& (-nm +n + m^2 -m) + m + 0 + (nm-n)  = m^2.
 \end{eqnarray*}
 Next, we compute
 \[ g(\K) = \frac{|nm-n| - |n| - 1 - c_1}{2} + 1 = \frac{|nm-n| - |n| - c_1 + 1}{2}.\]
 To find possible values of $c_1$, we must proceed further still.  Note that if $\K$ is genus-minimizing, then $g(\K) = \frac{
 (m-1)(m-2)}{2}$ if $m > 0$ and $g(\K) = \frac{(m+2)(m+1)}{2}$ if $m < 0$.  We examine six sub-cases:

{\be
\item $\ab = (1,\ell)$:  In this case, $n =m +1$ and $c_1 = 1$, so that $g(\K) = \frac{|m^2+m-(m+1)| - |m+1| - 1 +1}{2} = \frac{m^2-1 - |m+1|}{2}$.  If $m > 0$, then $g(\K) = \frac{m^2-m-2}{2} = \frac{(m+1)(m-2)}{2}$, so $\K$ is not genus-minimizing for $m >2$.  If $m < 0$, then $g(\K) = \frac{m^2-1-(-m-1)}{2} = \frac{m^2+m}{2}$, and again $\K$ is not genus-minimizing (if $m \neq -1)$.
\item *$\ab = (-1,\ell)$:  In this case, $n = m-1$ and $c_1=1$, so that (for $|m| \geq 2$) $g(\K) = \frac{|m^2-m-(m-1)|-|m-1|-1+1}{2} = \frac{m^2-2m+1-|m-1|}{2}$.  When $m > 0$, $g(\K) = \frac{(m-1)(m-2)}{2}$, so $\K$ is genus-minimizing.  When $m < 0$, $g(\K) = \frac{m^2-2m+1-(-m+1)}{2} = \frac{m^2-m}{2}$, so $\K$ is not genus-minimizing.
\item $\ab = (\ell,0)$:  In this case, $m=1$, so $[\K]^2 = 1$.
\item $\ab =(\ell,1)$:  In this case, $m=0$, so $[\K]^2 = 0$.
\item $\ab = (\ell,-1)$:  In this case, $m = 2$, so $[\K]^2 = 4$.
\item *$\ab = (0,\ell)$:  In this case, $n=m$ and $c_1 = |m-1|$, so that $g(\K) = \frac{|m^2-m|-|m|-|m-1|+1}{2}$.  If $m > 0$, then $g(\K) = \frac{m^2-3m+2}{2} = \frac{(m-1)(m-2)}{2}$, and $\K$ is genus-minimizing.  If $m < 0$, then $g(\K) = \frac{m^2-m-(-m)-(-m+1)+1}{2} = \frac{m^2+m}{2}$, and $\K$ is not genus-minimizing when $m \neq -1$.
\ee}

\item $\bc = (m,m)$:  In this case, $p_2 = 0$, $q_2 = -m$, and $p_2q_2 = 0$.   We compute $b = |nm|$, $\eps b = nm$, and $\ab = (n-m,-m)$.  Thus
\begin{eqnarray*}
p_1p_2 &=& -nm+m^2 \\
\text{$[\K]^2$} &=& (-nm+m^2) + 0 + 0 + nm = m^2
 \end{eqnarray*}
Next, we compute
 \[ g(\K) = \frac{|nm| - |n| - |m| - c_1}{2} + 1 = \frac{|nm| - |n| - |m| - c_1 + 2}{2}.\]
 To find possible values of $c_1$, we must proceed further still.  Note that if $\K$ is genus-minimizing, then $g(\K) = \frac{
 (m-1)(m-2)}{2}$ if $m > 0$ and $g(\K) = \frac{(m+2)(m+1)}{2}$ if $m < 0$.  We examine six sub-cases:

{\be
\item *$\ab = (1,\ell)$:  In this case, $n = m+1$ and $c_1 = 1$, so that $g(\K) = \frac{|m^2+m|-|m+1|-|m|-1+2}{2}$.  If $m > 0$, then $g(\K) = \frac{(m^2-m)^2}{2}$, so that $\K$ is not genus-minimizing.  If $m < 0$, then $g(\K) = \frac{m^2+m-(-m-1)-(-m)+1}{2} = \frac{m^2+3m+2}{2}$, and thus $\K$ is genus-minimizing.
\item *$\ab  = (-1,\ell)$:  In this case, $n = m -1$ and $c_1 = 1$, so that $g(\K) = \frac{|m^2-m| - |m-1| - |m| -1 + 2}{2}$.  If $m > 0$, then $g(\K) = \frac{m^2-3m+2}{2}$, and $\K$ is genus-minimizing.  If $m < 0$, then $g(\K) = \frac{m^2-m-(-m+1)-(-m)+1}{2} = \frac{m^2-m}{2}$, so that $\K$ is not genus-minimizing.
\item $\ab = (\ell,0)$:  In this case, $m=0$, so $[\K]^2 = 0$.
\item $\ab = (\ell,1)$:  In this case, $m = -1$, so $[\K]^2 = 1$.
\item $\ab = (\ell,-1)$:  In this case, $m=1$, so $[\K]^2 = 1$.
\item *$\ab = (0,\ell)$:  In this case, $n=m$ and $c_1 = |m|$, so that $g(\K) = \frac{|m^2| - 3|m|+2}{2}$, and thus $\K$ is genus-minimizing for all $m \neq 0$.
\ee}

\ee}

\item Suppose $\ac = (n+1,n)$:  Then $\ca = (-n-1,-n)$, so that $p_3 = n$, $q_3 = -1$, and $p_3q_3 = -n$.

{\be
\item $\bc = (1,m)$:  In this case, $p_2 = m-1$, $q_2 = -1$, and $p_2q_2 = 1-m$.  We compute $b = |nm+m-n|$, $\eps b = nm+m-n$, and $\ab = (n,n-m)$.  Thus
\begin{eqnarray*}
p_1p_2 &=& n^2-nm \\
\text{$[\K]^2$} &=& (n^2-nm) + (1-m) + -n + (nm+m-n) = (n-1)^2.
 \end{eqnarray*}
Next, we compute
 \[ g(\K) = \frac{|nm+m-n| - 1 - 1 - c_1}{2} + 1 = \frac{|nm+m-n| - c_1}{2}.\]
 To find possible values of $c_1$, we must proceed further still.  Note that if $\K$ is genus-minimizing, then $g(\K) = \frac{(n-2)(n-3)}{2}$ if $n > 1$ and $g(\K) = \frac{n^2+n}{2}$ if $n < 1$.  We examine six sub-cases:

{\be
\item $\ab =(1,\ell)$:  In this case, $n=1$, so $[\K]^2 = 0$
\item $\ab = (-1,\ell)$:  In this case, $n=-1$, so $[\K^2] = 4$.
\item *$\ab = (\ell,0)$:  In this case, $m=n$ and $c_1 = |n|$, so that $g(\K) = \frac{|n^2+n-n|-|n|}{2}$.  If $n > 1$, then $g(\K) = \frac{n^2-n}{2}$, so $\K$ is not genus-minimizing.  If $n<1$, then $g(\K) = \frac{n^2+n}{2}$, and $\K$ is genus-minimizing.
\item $\ab =(\ell,1)$:  In this case, $m = n-1$ and $c_1 = 1$, so that (for $n \neq 1$) $g(\K) = \frac{|n^2-n+(n-1)-n|-1}{2} = \frac{n^2-n-2}{2} = \frac{(n-2)(n+1)}{2}$, so $\K$ is not genus-minimizing.
\item *$\ab = (\ell,-1)$:  In this case, $m=n+1$ and $c_1 = 1$, so that $g(\K) = \frac{|n^2+n+(n+1)-n|-1}{2} = \frac{n^2+n}{2}$.  Thus, $\K$ is genus-minimizing if $n < 1$.
\item $\ab = (0,\ell)$:  In this case, $n=0$, so $[\K]^2 = 1$.
\ee}

\item $\bc = (-1,m)$:  In this case, $p_2 = m+1$, $q_2 = 1$, and $p_2q_2 = m+1$.  We compute $b = |nm+m+n|$, $\eps b = nm+m+n$, and $\ab = (n+2,n-m)$.  Thus
\begin{eqnarray*}
p_1p_2 &=& n^2+2n-nm-2m \\
\text{$[\K]^2$} &=& (n^2+2n-nm-2m) + (m+1) + (-n) + (nm+m+n) = (n+1)^2.
 \end{eqnarray*}
Next, we compute
 \[ g(\K) = \frac{|nm+m+n| - 1 - 1 - c_1}{2} + 1 = \frac{|nm+m+n| - c_1}{2}.\]
 To find possible values of $c_1$, we must proceed further still.  Note that if $\K$ is genus-minimizing, then $g(\K) = \frac{n(n-1)}{2}$ if $n > -1$ and $g(\K) = \frac{(n+3)(n+2)}{2}$ if $n < -1$.  We examine six sub-cases:
 
{\be
\item $\ab = (1,\ell)$:  In this case, $n=-1$, so $[\K]^2 = 0$.
\item $\ab = (-1,\ell)$:  In this case, $n=-3$, so $[\K]^2 = 4$.
\item $\ab =(\ell,0)$:  In this case, $m=n$ and $c_1 = |n+2|$, so that $g(\K) = \frac{|n^2+n+n|-|n+2|}{2}$.  If $n > -1$, then $g(\K) = \frac{n^2+n-2}{2} = \frac{(n-1)(n+2)}{2}$, and if $n<-1$, then $g(\K) = \frac{n^2+2n - (-n-2)}{2} = \frac{n^2+3n+2}{2} = \frac{(n+1)(n+2)}{2}$.  In both cases, $\K$ is not genus-minimizing.
\item $\ab =(\ell,1)$:  In this case, $m=n-1$ and $c_1=1$, so that (for $|n+1|>1$) $g(\K) = \frac{|n^2-n+(n-1)+n|-1}{2} = \frac{n^2+n-2}{2} = \frac{(n-1)(n+2)}{2}$, and $\K$ is not genus-minimizing.
\item $\ab=(\ell,-1)$:  In this case, $m=n+1$ and $c_1 =1$, so that (for $|n+1|>1$) $g(\K) = \frac{|n^2+n+(n+1)+n|-1}{2} = \frac{n^2+3n}{2}$, so $\K$ is not genus-minimizing.
\item $\ab =(0,\ell)$:  In this case, $n=-2$, so $[\K]^2 = 1$.
\ee}
 
\item $\bc = (0,m)$:  In this case, $p_2 = m$, $q_2 = 0$, and $p_2q_2 = 0$.   We compute $b = |nm+m|$, $\eps b = nm+m$, and $\ab = (n+1,n-m)$.  Thus
\begin{eqnarray*}
p_1p_2 &=& n^2-nm+n-m \\
\text{$[\K]^2$} &=& (n^2-nm+n-m) + 0 + -n + (nm+m) = n^2.
 \end{eqnarray*}
Next, we compute
\[ g(\K) = \frac{|nm+m| - |m| - 1 - c_1}{2} + 1 = \frac{|nm+m| - |m| - c_1+1}{2}.\]
To find possible values of $c_1$, we must proceed further still.  Note that if $\K$ is genus-minimizing, then $g(\K) = \frac{(n-1)(n-2)}{2}$ if $n > 0$ and $g(\K) = \frac{(n+2)(n+1)}{2}$ if $n < 0$.  We examine six sub-cases:

{\be
\item $\ab= (1,\ell)$:  In this case, $n = 0$, so $[\K]^2 = 0$.
\item $\ab = (-1,\ell)$:  In this case, $n=-2$, so $[\K^2] = 4$.
\item *$\ab = (\ell,0)$:  In this case, $m=n$ and $c_1 = |n+1|$, so that $g(\K) = \frac{|n^2+n|-|n|-|n+1|+1}{2}$.  If $n > 0$, then $g(\K) = \frac{n^2-n}{2}$, so $\K$ is not genus-minimizing.  If $n < 0$, then $g(\K) = \frac{n^2+n-(-n)-(-n-1)+1}{2} = \frac{(n+2)(n+1)}{2}$, and $\K$ is genus-minimizing.
\item $\ab = (\ell,1)$:  In this case, $m = n-1$ and $c_1 = 1$, so that $g(\K) = \frac{|n^2-n+(n-1)|-|n-1|-1+1}{2}$.  If $n>0$, then $g(\K) = \frac{n^2-n}{2}$.  If $n < 0$, then $g(\K) = \frac{n^2+n-2}{2} = \frac{(n+1)(n-2)}{2}$.  In either case, $\K$ is not genus-minimizing.
\item *$\ab = (\ell,-1)$:  In this case, $m=n+1$ and $c_1 = 1$, so that $g(\K) = \frac{|n^2+n+n+1|-|n+1|-1+1}{2}$.  If $n > 0$, then $g(\K) = \frac{n^2+n}{2}$, so $\K$ is not genus-minimizing.  If $n < 0$, then $g(\K) = \frac{n^2+3n+2}{2}$ and $\K$ is genus-minimizing.
\item $\ab =(0,\ell)$:  In this case, $n=-1$, so $[\K]^2 = 1$.
\ee}

\item $\bc = (m,m+1)$:  In this case, $p_2 = 1$, $q_2 = -m$, and $p_2q_2 = -m$.   We compute $b = |(n+1)(m+1)-nm| = |n+m+1|$, $\eps b = n+m+1$, and $\ab = (n-m+1,n-m-1)$.  Thus
\begin{eqnarray*}
p_1p_2 &=& (n-m)^2-1 = n^2-2nm+m^2-1 \\
\text{$[\K]^2$} &=& (n^2-2nm+m^2-1) + -m + -n + (n+m+1) = (n-m)^2.
 \end{eqnarray*} 
We proceed further, examining six sub-cases:

{\be
\item $\ab = (1,\ell)$:  In this case, $n-m=0$, so $[\K]^2 = 0$.
\item $\ab = (-1,\ell)$:  In this case, $n-m=-2$, so $[\K]^2 = 4$.
\item $\ab = (\ell,0)$:  In this case, $n-m = 1$, so $[\K]^2 = 1$.
\item $\ab = (\ell,1)$:  In this case, $n-m=2$, so $[\K]^2 = 4$.
\item $\ab = (\ell,-1)$:  In this case, $n-m=0$, so $[\K]^2 = 0$.
\item $\ab = (0,\ell)$:  In this case, $n-m=-1$, so $[\K]^2 = 1$.
\ee}

\item $\bc = (m,m-1)$:  In this case, $p_2 = -1$, $q_2 = -m$, and $p_2q_2 = m$.   We compute $b = |(n+1)(m-1)-nm|= |-n+m-1|$, $\eps b = -n+m-1$, and $\ab = (n-m+1,n-m+1) = (n-m+1)\g$.  Thus
\begin{eqnarray*}
p_1p_2 &=& (n-m+1)^2 = (n-m)^2+2n-2m+1 \\
\text{$[\K]^2$} &=& ((n-m)^2+2n-2m+1) + m + -n + (-n+m-1)  = (n-m)^2.
\end{eqnarray*}
Since $\ab = (c_1,c_1)$, where $c_1 = n-m+1$, this is a bridge trisection if and only if $n-m=0$ or $n-m=-2$.

\item $\bc = (m,m)$:  In this case, $p_2 = 0$, $q_2 = -m$, and $p_2q_2 = 0$.   We compute $b = |m|$, $\eps b = m$, and $\ab = (n-m+1,n-m)$.  Thus
\begin{eqnarray*}
p_1p_2 &=& (n-m)^2+n-m \\
\text{$[\K]^2$} &=& ((n-m)^2+n-m) + 0 + -n + m = (n-m)^2.
 \end{eqnarray*}
 
 We proceed further, examining six sub-cases:

{\be
\item $\ab = (1,\ell)$:  In this case, $n-m = 0$, so $[\K]^2 = 0$.
\item $\ab  = (-1,\ell)$:  In this case, $n-m = -2$, so $[\K]^2 = 4$.
\item $\ab = (\ell,0)$:  In this case, $n-m = 0$, so $[\K]^2 = 0$.
\item $\ab = (\ell,1)$:  In this case, $n-m = 1$, so $[\K]^2 = 1$.
\item $\ab = (\ell,-1)$:  In this case, $n-m = -1$, so $[\K]^2 = 1$.
\item $\ab =(0,\ell)$:  In this case, $n-m = -1$, so $[\K]^2 = 1$.
\ee}

\ee}

\item Suppose $\ac = (n-1,n)$:  Then $\ca = (-n+1,-n)$, so that $p_3 = n$, $q_3 = 1$, and $p_3q_3 = n$.

{\be

\item $\bc = (1,m)$:  In this case, $p_2 = m-1$, $q_2 = -1$, and $p_2q_2 = 1-m$.  We compute $b = |nm-m-n|$, $\eps b = nm-m-n$, and $\ab = (n-2,n-m)$.  Thus
\begin{eqnarray*}
p_1p_2 &=& n^2-nm-2n+2m \\
\text{$[\K]^2$} &=& (n^2-nm-2n+2m) + (1-m) + n + (nm-m-n) = (n-1)^2.
 \end{eqnarray*}
Next, we compute
 \[ g(\K) = \frac{|nm-m-n| - 1 - 1 - c_1}{2} + 1 = \frac{|nm-m-n| - c_1}{2}.\]
 To find possible values of $c_1$, we must proceed further still.  Note that if $\K$ is genus-minimizing, then $g(\K) = \frac{(n-2)(n-3)}{2}$ if $n > 1$ and $g(\K) = \frac{n^2+n}{2}$ if $n < 1$.  We examine six sub-cases:
 
 {\be
  \item $\ab = (1,\ell)$:  In this case, $n=3$, so $[\K]^2 = 4$. 
 \item $\ab = (-1,\ell)$:  In this case, $n=1$, so $[\K]^2 = 0$.
 \item $\ab = (\ell,0)$:  In this case, $m=n$ and $c_1 = |n-2|$, so that $g(\K) = \frac{|n^2-n-n|-|n-2|}{2}$.  If $n > 1$, then $g(\K) = \frac{n^2-3n+2}{2} = \frac{(n-2)(n-1)}{2}$.  If $n < -1$, then $g(\K) = \frac{n^2-2n-(-n+2)}{2} = \frac{(n-2)(n+1)}{2}$.  In both cases, $\K$ is not genus-minimizing.
  \item $\ab = (\ell,1)$:  In this case, $m = n-1$ and $c_1 = 1$, so that (for $n\neq 0,1,2$) $g(\K) = \frac{|n^2-n-(n-1)-n|-1}{2} = \frac{n^2-3n}{2}$.  Thus, $\K$ is not genus-minimizing. 
 \item $\ab = (\ell,-1)$:  In this case, $m=n+1$ and $c_1 = 1$, so that (for $n \neq 0,1$) $g(\K) = \frac{|n^2+n-(n+1)-n| - 1}{2} = \frac{n^2-n-2}{2} = \frac{(n-2)(n+1)}{2}$.  Thus, $\K$ is not genus-minimizing.
 \item $\ab = (0,\ell)$:  In this case, $n=2$, so $[\K]^2 = 1$.
\ee}

\item $\bc = (-1,m)$:  In this case, $p_2 = m+1$, $q_2 = 1$, and $p_2q_2 = m+1$.  We compute $b = |nm-m+n|$, $\eps b = nm-m+n$, and $\ab = (n,n-m)$.  Thus
\begin{eqnarray*}
p_1p_2 &=& n^2-nm \\
\text{$[\K]^2$} &=& (n^2-nm) + (m+1) + n + (nm-m+n) = (n+1)^2.
 \end{eqnarray*}
Next, we compute
 \[ g(\K) = \frac{|nm-m+n| - 1 - 1 - c_1}{2} + 1 = \frac{|nm-m+n| - c_1}{2}.\]
 To find possible values of $c_1$, we must proceed further still.  Note that if $\K$ is genus-minimizing, then $g(\K) = \frac{n(n-1)}{2}$ if $n > -1$ and $g(\K) = \frac{(n+3)(n+2)}{2}$ if $n < -1$.  We examine six sub-cases:

{\be
\item $\ab =(1,\ell)$:  In this case, $n=1$, so $[\K]^2 = 4$.
\item $\ab = (-1,\ell)$:  In this case, $n=-1$, so $[\K]^2 = 0$.
\item *$\ab = (\ell,0)$:  In this case, $n=m$ and $c_1 = |n|$, so that $g(\K) = \frac{|n^2-n+n| - |n|}{2}$.  If $n > -1$, then $g(\K) = \frac{n^2-n}{2}$ and $\K$ is genus-minimizing.  If $n < -1$, then $g(\K) = \frac{n^2+n}{2}$, so that $\K$ is not genus-minimizing.
\item *$\ab =(\ell,1)$:  In this case, $m = n-1$ and $c_1 = 1$, so that $g(\K) = \frac{|n^2-n-(n-1)+n|-1}{2} = \frac{n^2-n}{2}$.  It follows that if $n > -1$, then $\K$ is genus-minimizing.
\item $\ab = (\ell,-1)$:  In this case, $m = n+1$ and $c_1 = 1$, so that (for $n \neq -1, 0$) $g(\K) = \frac{|n^2+n-(n+1)+n|-1}{2} = \frac{n^2+n-2}{2} = \frac{(n-1)(n+2)}{2}$.  Thus, $\K$ is not genus-minimizing.
\item $\ab =(0,\ell)$:  In this case, $n=0$, so $[\K]^2 = 1$.
\ee}

\item $\bc = (0,m)$:  In this case, $p_2 = m$, $q_2 = 0$, and $p_2q_2 = 0$.   We compute $b = |nm-m|$, $\eps b = nm-m$, and $\ab = (n-1,n-m)$.  Thus
\begin{eqnarray*}
p_1p_2 &=& n^2-nm-n+m \\
\text{$[\K]^2$} &=& (n^2-nm-n+m) + 0 + n + (nm-m) = n^2.
 \end{eqnarray*}
Next, we compute
\[ g(\K) = \frac{|nm-m| - |m| - 1 - c_1}{2} + 1 = \frac{|nm-m| - |m| - c_1+1}{2}.\]
To find possible values of $c_1$, we must proceed further still.  Note that if $\K$ is genus-minimizing, then $g(\K) = \frac{(n-1)(n-2)}{2}$ if $n > 0$ and $g(\K) = \frac{(n+2)(n+1)}{2}$ if $n < 0$.  We examine six sub-cases:

{\be
\item $\ab =(1,\ell)$:  In this case, $n=2$, so $[\K]^2 = 4$.
\item $\ab = (-1,\ell)$:  In this case, $n=0$, so $[\K]^2 = 0$.
\item *$\ab = (\ell,0)$:  In this case, $m=n$ and $c_1 = |n-1|$, so that $g(\K) = \frac{|n^2-n|-|n|-|n-1|+1}{2}$.  If $n > 0$, then $g(\K) = \frac{n^2-3n+2}{2}$, and $\K$ is genus minimizing.  If $n < 0$, then $g(\K) = \frac{n^2+n}{2}$, so that $\K$ is not genus-minimizing.
\item *$\ab = (\ell,1)$:  In this case, $m=n-1$ and $c_1 = 1$, so that $g(\K) = \frac{|n^2-n-(n-1)|-|n-1|-1+1}{2}$.  If $n>0$, then $g(\K) = \frac{n^2-3n+2}{2}$, and $\K$ is genus-minimizing.  If $n < 0$, then $g(\K) = \frac{n^2-n}{2}$, so that $\K$ is not genus-minimizing.
\item $\ab = (\ell,-1)$:  In this case, $m=n+1$ and $c_1=1$, so that $g(\K) = \frac{|n^2+n-(n+1)| - |n+1| - 1+1}{2}$.  If $n > 0$, then $g(\K) = \frac{n^2-n-2}{2} = \frac{(n-2)(n+1}{2}$.  If $n < 0$, then $g(\K) = \frac{n^2+n}{2}$.  In either case, $\K$ is not genus-minimizing.
\item $\ab =(0,\ell)$:  In this case, $n=1$, so $[\K]^2 = 1$.
\ee}

\item $\bc = (m,m+1)$:  In this case, $p_2 = 1$, $q_2 = -m$, and $p_2q_2 = -m$.   We compute $b = |(n-1)(m+1)-nm| = |n-m-1|$, $\eps b = n-m-1$, and $\ab = (n-m-1,n-m-1)$.  Thus
\begin{eqnarray*}
p_1p_2 &=& (n-m-1)^2 \\
\text{$[\K]^2$} &=& (n-m-1)^2 + -m + n + (n-m-1) = (n-m)^2.
 \end{eqnarray*} 
Since $\ab = (c_1,c_1)$, where $c_1 = n-m-1$, this is a bridge trisection if and only if $n-m=2$ or $n-m=0$, so $[\K]^2 = 4$ or 0.

\item $\bc = (m,m-1)$:  In this case, $p_2 = -1$, $q_2 = -m$, and $p_2q_2 = m$.   We compute $b = |(n-1)(m-1)-nm|= |-n-m+1|$, $\eps b = -n-m+1$, and $\ab = (n-m-1,n-m+1)$.  Thus
\begin{eqnarray*}
p_1p_2 &=& (n-m)^2 - 1\\
\text{$[\K]^2$} &=& ((n-m)^2-1) + m + n + (-n-m+1)  = (n-m)^2.
\end{eqnarray*}
We proceed further, examining six sub-cases:

{\be
\item $\ab = (1,\ell)$:  In this case, $n-m=2$, so $[\K]^2 = 4$.
\item $\ab = (-1,\ell)$:  In this case, $n-m=0$, so $[\K]^2 = 0$.
\item $\ab =(\ell,0)$:  In this case, $n-m=-1$, so $[\K]^2 = 1$.
\item $\ab = (\ell,1)$:  In this case, $n-m=0$, so $[\K]^2 = 0$.
\item $\ab = (-\ell,1)$:  In this case, $n-m=-2$, so $[\K]^2 = 4$.
\item $\ab = (0,\ell)$:  In this case, $n-m=1$, so $[\K]^2 = 1$.
\ee}

\item $\bc = (m,m)$:  In this case, $p_2 = 0$, $q_2 = -m$, and $p_2q_2 = 0$.   We compute $b = |-m|$, $\eps b = -m$, and $\ab = (n-m-1,n-m)$.  Thus
\begin{eqnarray*}
p_1p_2 &=& (n-m)^2-n+m \\
\text{$[\K]^2$} &=& ((n-m)^2-n+m) + 0 + n - m = (n-m)^2.
 \end{eqnarray*}
We proceed further, examining six sub-cases:

{\be
\item $\ab = (1,\ell)$:  In this case, $n-m=2$, so $[\K]^2 =4$.
\item $\ab  = (-1,\ell)$:  In this case, $n-m=0$, so $[\K]^2 = 0$.
\item $\ab = (\ell,0)$:  In this case, $n-m=0$, so $[\K]^2 = 0$.
\item $\ab  = (\ell,1)$:  In this case, $n-m=1$, so $[\K]^2 = 1$.
\item $\ab = (\ell,-1)$:  In this case, $n-m=-1$, so $[\K]^2 = 1$.
\item $\ab =(0,\ell)$:  In this case, $n-m=1$, so $[\K]^2 = 1$.
\ee}

\ee}

\item Suppose $\ac = (n,n)$:  Then $\ca = (-n,-n)$, so that $p_3 = n$, $q_3 = 0$, and $p_3q_3 = 0$.

{\be

\item $\bc = (1,m)$:  In this case, $p_2 = m-1$, $q_2 = -1$, and $p_2q_2 = 1-m$.  We compute $b = |nm-n|$, $\eps b = nm-n$, and $\ab = (n-1,n-m)$.  Thus
\begin{eqnarray*}
p_1p_2 &=& n^2-nm-n+m \\
\text{$[\K]^2$} &=& (n^2-nm-n+m) + (1-m) + 0 + (nm-n) = (n-1)^2.
 \end{eqnarray*}
Next, we compute
 \[ g(\K) = \frac{|nm-n| - |n| - 1 - c_1}{2} + 1 = \frac{|nm-n| -|n| - c_1+1}{2}.\]
 To find possible values of $c_1$, we must proceed further still.  Note that if $\K$ is genus-minimizing, then $g(\K) = \frac{(n-2)(n-3)}{2}$ if $n > 1$ and $g(\K) = \frac{n^2+n}{2}$ if $n < 1$.  We examine six sub-cases:
 
 {\be
 \item $\ab = (1,\ell)$:  In this case, $n=2$, so $[\K]^2 = 1$.
 \item $\ab  = (-1,\ell)$:  In this case, $n=0$, so $[\K]^2 = 1$.
 \item *$\ab = (\ell,0)$:  In this case, $m=n$ and $c_1 = |n-1|$, so that $g(\K) = \frac{|n^2-n| - |n|-|n-1| +1}{2}$.  If $n > 1$, then $g(\K) = \frac{n^2-3n+2}{2} = \frac{(n-1)(n-2)}{2}$, so that $\K$ is not genus-minimizing.  If $n < 1$, then $g(\K) = \frac{n^2+n}{2}$, and $\K$ is genus-minimizing.
  \item $\ab = (\ell,1)$:  In this case, $m = n-1$ and $c_1= 1$, so that $g(\K) = \frac{n^2-n-n|-|n|-1+1}{2}$.  If $n>1$, then $g(\K) = \frac{n^2-3n}{2}$.  If $n < 1$, then $\g(\K) = \frac{n^2-n}{2}$.  In either case, $\K$ is not genus-minimizing.
 \item *$\ab  = (\ell,-1)$:  In this case, $m=n+1$ and $c_1=1$, so that $g(\K) = \frac{|n^2+n-n|-|n|-1+1}{2} = \frac{n^2-|n|}{2}$.  If $n>1$, then $\K$ is not genus-minimizing, but if $n < 1$, then $\K$ is genus-minimizing.
 \item $\ab = (0,\ell)$:  In this case, $n=1$, so $[\K]^2 = 0$.
 \ee}
 
\item $\bc = (-1,m)$:  In this case, $p_2 = m+1$, $q_2 = 1$, and $p_2q_2 = m+1$.  We compute $b = |nm+n|$, $\eps b = nm+n$, and $\ab = (n+1,n-m)$.  Thus
\begin{eqnarray*}
p_1p_2 &=& n^2-nm+n-m \\
\text{$[\K]^2$} &=& (n^2-nm+n-m) + (m+1) + 0 + (nm+n) = (n+1)^2.
 \end{eqnarray*}
Next, we compute
 \[ g(\K) = \frac{|nm+n| - |n| - 1 - c_1}{2} + 1 = \frac{|nm+n| - |n| - c_1+1}{2}.\]
 To find possible values of $c_1$, we must proceed further still.  Note that if $\K$ is genus-minimizing, then $g(\K) = \frac{n(n-1)}{2}$ if $n > -1$ and $g(\K) = \frac{(n+3)(n+2)}{2}$ if $n < -1$.  We examine six sub-cases:

 {\be
  \item $\ab = (1,\ell)$:  In this case, $n=0$, so $[\K]^2 = 1$.
 \item $\ab = (-1,\ell)$:  In this case, $n=-2$, so $[\K]^2 = 1$.
 \item *$\ab = (\ell,0)$:  In this case, $m=n$ and $c_1 = |n+1|$, so that $g(\K) = \frac{|n^2+n|-|n|-|n+1|+1}{2}$.  If $n > -1$, then $g(\K) = \frac{n^2-n}{2}$, and $\K$ is genus-minimizing.  If $n < -1$, then $g(\K) = \frac{n^2+3n+2}{2} = \frac{(n+1)(n+2)}{2}$, so that $\K$ is not genus-minimizing.
  \item *$\ab = (\ell,1)$:  In this case, $m = n-1$ and $c_1=1$, so that $g(\K) = \frac{|n^2-n+n|-|n|-1+1}{2}$.  If $n>-1$, then $g(\K) = \frac{n^2-n}{2}$ and $\K$ is genus-minimizing.  If $n < -1$, then $g(\K) = \frac{n^2+n}{2}$, so that $\K$ is not genus-minimizing.
 \item $\ab = (\ell,-1)$:  In this case, $m=n+1$ and $c_1 = 1$, so that $g(\K) = \frac{|n^2+n+n|-|n|-1+1}{2}$.  If $ n > -1$, then $g(\K) = \frac{n^2+n}{2}$.  If $n < -1$, then $g(\K) = \frac{n^2+3n}{2}$.  In either case, $\K$ is not genus-minimizing.
 \item $\ab = (0,\ell)$:  In this case, $n=-1$, so $[\K]^2 = 0$.
 \ee}
 
 \item $\bc = (0,m)$:  In this case, $p_2 = m$, $q_2 = 0$, and $p_2q_2 = 0$.   We compute $b = |nm|$, $\eps b = nm$, and $\ab = (n,n-m)$.  Thus
\begin{eqnarray*}
p_1p_2 &=& n^2-nm \\
\text{$[\K]^2$} &=& (n^2-nm) + 0 + 0 + nm = n^2.
 \end{eqnarray*}
Next, we compute
\[ g(\K) = \frac{|nm| - |n| - |m| - c_1}{2} + 1 = \frac{|nm| - |n| - |m| - c_1+2}{2}.\]
To find possible values of $c_1$, we must proceed further still.  Note that if $\K$ is genus-minimizing, then $g(\K) = \frac{(n-1)(n-2)}{2}$ if $n > 0$ and $g(\K) = \frac{(n+2)(n+1)}{2}$ if $n < 0$.  We examine six sub-cases:
 
 {\be
  \item $\ab  = (1,\ell)$:  In this case, $n=1$, so $[\K]^2 = 1$.
   \item $\ab = (-1,\ell)$:  In this case, $n=-1$, so $[\K]^2 = 1$.
 \item *$\ab = (\ell,0)$:  In this case, $m=n$ and $c_1=|n|$, so that $g(\K) = \frac{|n^2|-|n|-|n|-|n|+2}{2} = \frac{n^2-3|n|+2}{2}$.  Thus, $\K$ is genus-minimizing for all $n\neq 0$.
  \item *$\ab = (\ell,1)$:  In this case, $m = n-1$ and $c_1=1$, so that $g(\K) = \frac{|n^2-n|-|n|-|n-1|-1+2}{2}$.  If $n>0$, then $g(\K) = \frac{n^2-3n+2}{2}$, and $\K$ is genus-minimizing.  If $n<0$, then $g(\K) = \frac{n^2+n}{2}$, so that $\K$ is not genus-minimizing.
 \item *$\ab  = (-\ell,-1)$:  In this case, $m=n+1$ and $c_1=1$, so that $g(\K) = \frac{|n^2+n|-|n|-|n+1|-1+2}{2}$.  If $n>0$, then $g(\K) = \frac{n^2-n}{2} = \frac{n(n-1)}{2}$.  If $n < 0$, then $g(\K) = \frac{n^2+3n+2}{2}$, and $\K$ is genus-minimizing.
 \item $\ab =(0,\ell)$:  In this case, $n=0$, so $[\K]^2 = 0$.
\ee}

\item $\bc = (m,m+1)$:  In this case, $p_2 = 1$, $q_2 = -m$, and $p_2q_2 = -m$.   We compute $b = |n|$, $\eps b = n$, and $\ab = (n-m,n-m-1)$.  Thus
\begin{eqnarray*}
p_1p_2 &=& (n-m)^2-n+m \\
\text{$[\K]^2$} &=& ((n-m)^2-n+m) + -m + 0 + n = (n-m)^2.
 \end{eqnarray*} 
We proceed further, examining six sub-cases:

{\be
\item $\ab = (1,\ell)$:  In this case, $n-m=1$, so $[\K]^2 = 1$.
\item $\ab  = (-1,\ell)$:  In this case, $n-m=-1$, so $[\K]^2 = 1$.
\item $\ab = (\ell,0)$:  In this case, $n-m=1$, so $[\K]^2 =1$.
\item $\ab =(\ell,1)$:  In this case, $n-m=2$, so $[\K]^2 = 4$.
\item $\ab = (\ell,-1)$:  In this case, $n-m=0$, so $[\K]^2 = 0$.
\item $\ab = (0,\ell)$:  In this case, $n-m=0$, so $[\K]^2 = 0$.
\ee}

\item $\bc = (m,m-1)$:  In this case, $p_2 = -1$, $q_2 = -m$, and $p_2q_2 = m$.   We compute $b = |-n|$, $\eps b = -n$, and $\ab = (n-m,n-m+1)$.  Thus
\begin{eqnarray*}
p_1p_2 &=& (n-m)^2 +n-m\\
\text{$[\K]^2$} &=& ((n-m)^2+n-m) + m + 0 + -n  = (n-m)^2.
\end{eqnarray*}
We proceed further, examining six sub-cases:

{\be
\item $\ab = (1,\ell)$:  In this case, $n-m=1$, so $[\K]^2 = 1$.
\item $\ab = (-1,\ell)$:  In this case, $n-m=-1$, so $[\K]^2 = 1$.
\item $\ab =(\ell,0)$:  In this case, $n-m=-1$, so $[\K]^2 = 1$.
\item $\ab = (\ell,1)$:  In this case, $n-m=0$, so $[\K]^2 = 0$.
\item $\ab = (\ell,-1)$:  In this case, $n-m=-2$, so $[\K]^2 = 4$.
\item $\ab = (0,\ell)$:  In this case, $n-m=0$, so $[\K]^2 = 0$.
\ee}

\item $\bc = (m,m)$:  In this case, $\ac$ and $\bc$ do not intersect essentially.
\ee}

\ee

In Table~\ref{table1}, we summarize those cases in which $\K$ minimizes genus in its homology class.  In order to compare various cases, we will normalize by letting $d = \sqrt{[\K]^2}$ and writing $m$ and $n$ in terms of $d$ for those values of $m$ and $n$ that produce genus-minimizing $\K$.  For example, in case (1di), we have $[\K]^2 = (m-1)^2$, and $\K$ is genus-minimizing when $m < 1$.  Thus, for $m<1$, we have $d = 1-m$, so that $d$ is positive.

This convention splits two of our cases, cases (3fvi) and (6ciii), since $[\K]^2 = m^2$ and $\K$ is genus-minimizing for all $m \neq 0$ in case (3fvi), and $[\K]^2 = n^2$ and $\K$ is genus-minimizing for all $n \neq 0$ in case (6ciii).  Thus, in case (3fvi+), we assume $m > 0$ and $d=m$, and in case (3fvi-), we assume $m < 0$ and $d = -m$.  Similarly, in case (6ciii+), we assume $n > 0$ and $d = n$, while in case (6ciii-), we assume $n < -$ and $d = -n$.

Note that $d$ does not necessarily coincide with the degree of the surface $\K$ here, since it is possible that the degree of $\K$ is negative, while $d$ is always considered to be positive in Table~\ref{table1}.  Examining the cases in the table, we form equivalence classes of those cases that are related by a cyclic permutation of $\aaa$, $\bbb$, and $\ccc$ or by reversing the orientation of such a permutation.  When we do this, the cases fall into eight different equivalence classes:

\be
\item[(A)] 1di, 2eii
\item[(B)] 1dvi, 1fi, 2evi, 2fii, 3di, 3eii
\item[(C)] 1fvi, 2fvi, 3dvi, 3evi, 3fi, 3fii
\item[(D)] 3fvi+, 3fvi-
\item[(E)] 4av, 5biv
\item[(F)] 4aiii, 4cv, 5biii, 5civ, 6av, 6biv
\item[(G)] 4ciii, 5ciii, 6aiii, 6biii, 6civ, 6cv
\item[(H)] 6ciii+, 6ciii-
\ee

This completes the proof of the proposition.

\begin{table}[ht]
\centering
\caption{Cases yielding genus-minimizing surfaces}
\label{table1}
\begin{tabular}{ccccccc}
\hline
case & $m$ & $ n $ & $\ab $ & $\bc $ & $\ca $  \\
\hline \hline
(1di) & $-d+1$ & $-d+2$ & $\A + (d-1)\n$ & $\n + (d-1)\g$ & $\g + (d-1)\A$ \\
\hline
(1dvi) & $-d+1$  & $-d+1$ & $0\A + (d-1)\n$ & $\n + (d-1)\g$ & $ \g + d\A$  \\
\hline
(1fi) & $-d+1 $ & $-d+2$ & $\A + d\n$ & $0\n + (d-1)\g$ & $\g + (d-1)\A$ \\
\hline
(1fvi) & $-d+1$ & $-d+1$ & $0\A + d\n$ & $0\n + (d-1)\g$ & $\g + d\A$ \\
\hline
(2eii) & $d-1$ & $d-2$ & $-\A + (-d+1)\n$ & $-\n +(-d+1)\g$ & $-\g + (-d+1)\A$ \\
\hline
(2evi) & $d-1$ & $d-1$ & $0\A +(-d+1)\n$ & $-\n + (-d+1)\g$ & $-\g + (-d)\A$ \\
\hline
(2fii) & $d-1$ & $d-2$ & $-\A - d\n$ & $0\n + (-d+1)\g$ & $-\g + (-d+1)\A$ \\
\hline
(2fvi) & $d-1$ & $d-1$ & $0\A - d\n$ & $0\n + (-d+1)\g$ & $-\g + (-d)\A$ \\
\hline
(3di) & $-d$ & $-d+1$ & $\A + (d-1)\n$ & $\n+d\g$ & $0\g+(d-1)\A$ \\
\hline
(3dvi) & $-d$ & $-d$ & $0\A +(d-1)\n$ & $\n + d\g$ & $0\g + d\A$ \\
\hline
(3eii) & $d$ & $d-1$ & $\A + (-d+1)\n$ & $-\n-d\g$ & $0\g + (-d+1)\A$ \\
\hline
(3evi) & $d$ & $d$ & $0\A + (-d+1)\n$ & $ -\n - d\g$ & $0\g - d\A$ \\
\hline
(3fi) & $-d$ & $-d+1$ & $\A + d\n$ & $0\n+d\g$ & $0\g + (d-1)\A$ \\
\hline
(3fii) & $d$ & $d-1$ & $-\A - d\n$ & $0\n -d\g$ & $ 0\g + (-d+1)\A$\\
\hline
(3fvi+) & $d$ & $d$ & $0\A - d\n$ & $0\n -d\g$ & $0\g -d\A$ \\
\hline
(3fvi-) & $-d$ & $-d$ & $0\A + d\n$ & $0\n + d\g$ & $0\g + d\A$ \\
\hline
(4aiii) & $-d+1$ & $-d+1$ & $(-d+1)\A + 0\n$ & $-d\n - \g$ & $(-d+1)\g-\A$\\
\hline
(4av) & $-d+2 $ & $-d+1 $ & $(-d+1)\A -\n$ & $(-d+1)\n - \g$ & $(-d+1)\g-\A$ \\
\hline
(4ciii) & $-d$ & $-d$ &  $(-d+1)\A + 0\n$ & $-d\n + 0\g$ & $-d\g - \A$ \\
\hline
(4cv) & $-d+1$ & $-d$ & $(-d+1)\A - \n$ & $(-d+1)\n + 0\g$ & $-d\g - \A$ \\
\hline
(5biii) & $d-1$ & $d-1$ & $(d-1)\A + 0\n$ & $d\n + \g$ & $(d-1)\g + \A$ \\
\hline
(5biv) & $d-2$ & $d-1$ & $(d-1)\A + \n$ & $(d-1)\n + \g$ & $(d-1)\g + \A$ \\
\hline
(5ciii) & $d$ & $d$ & $(d-1)\A + 0\n$ & $d\n + 0\g$ & $d\g + \A$ \\
\hline
(5civ) & $d-1$ & $d$ & $(d-1)\A + \n$ & $(d-1)\n + 0\g$ & $ d\g + \A$ \\
\hline
(6aiii) & $-d+1$ & $-d+1$ & $-d\A + 0\n$ & $-d\n -\g$ & $(-d+1)\g + 0\A$ \\
\hline
(6av) & $-d+2$ & $-d+1$ & $-d\A - \n$ & $(-d+1)\n - \g$ & $(-d+1)\g + 0\A$\\
\hline
(6biii) & $d-1$ & $d-1$ & $d\A + 0\n$ & $d\n+\g$ & $(d-1)\g + 0\g$ \\
\hline
(6biv) & $d-2$ & $d-1$ & $d\A + \n$ & $(d-1)\n + \g$ & $(d-1)\g + 0\A$\\
\hline
(6ciii+) & $d$ & $d$ & $d\A + 0\n$ & $d\n + 0\g$ & $d\g + 0\n$\\
\hline
(6ciii-) & $-d$ & $-d$ & $-d\A + 0\n$ & $-d\n + 0\g$ & $-d\g + 0\n$\\
\hline
(6civ) & $d-1$ & $d$ & $d\A + \n$ & $(d-1)\n + 0\g$ & $d\g + 0\n$ \\
\hline
(6cv) & $-d+1$ & $-d$ & $-d\A - \n$ & $(-d+1)\n + 0\g$ & $-d\g + 0\n$\\
\hline
\end{tabular}
\end{table}

\end{proof}

\bibliographystyle{amsalpha}
\bibliography{complexbib}

\end{document}